\documentclass[review,onefignum,onetabnum]{siamart171218}

\usepackage{ifthen}
\newboolean{todonotes}
\setboolean{todonotes}{true}
\ifthenelse{\boolean{todonotes}}
{
  \usepackage{marginnote}
  \usepackage{cprotect}
  \newcommand{\todo}[1]{\marginnote{\textcolor{red}{\parbox{4.5cm}{\raggedright\footnotesize #1}}}}
  \paperwidth=\dimexpr\paperwidth + 3cm\relax
  \oddsidemargin=\dimexpr\oddsidemargin + 3cm\relax
  \evensidemargin=\dimexpr\evensidemargin + 3cm\relax
  \marginparwidth=\dimexpr \marginparwidth + 3cm\relax
  \marginparsep=\dimexpr \marginparsep + 1.5em\relax  
}{
  \newcommand{\todo}[1]{}
}

\renewcommand{\arraystretch}{1.3}

\usepackage{amssymb,amsfonts,amsopn,mathrsfs,calligra}
\usepackage{graphicx,epstopdf,subfig}
\usepackage{epstopdf}
\usepackage{array,multirow}
\usepackage{algcompatible}
\usepackage{booktabs}

\usepackage[shortlabels]{enumitem} 

\allowdisplaybreaks
\usepackage{tikz,pgfplots}

\makeatletter
\let\oldtheequation\theequation
\renewcommand\tagform@[1]{\maketag@@@{\ignorespaces#1\unskip\@@italiccorr}}
\renewcommand\theequation{(\oldtheequation)}
\makeatother

\usetikzlibrary{patterns}
\usetikzlibrary{quotes,arrows.meta}
\usetikzlibrary{pgfplots.groupplots}
\usepackage{pgfplotstable}
\usepackage[section]{placeins}
\usepackage{float}
\usepackage{amsopn}

\numberwithin{figure}{section}
\numberwithin{table}{section}

\Crefname{equation}{Eq.}{Eqs.}
\Crefname{definition}{Def.}{Defs.}
\Crefname{subsection}{Section}{Sections}
\Crefname{subsubsection}{Section}{Sections}
\crefname{hypothesis}{Hypothesis}{Hypotheses}

\ifpdf
  \DeclareGraphicsExtensions{.eps,.pdf,.png,.jpg}
\else
  \DeclareGraphicsExtensions{.eps}
\fi

\definecolor{colorq}{RGB}{44, 162, 95}
\definecolor{colorl}{RGB}{43, 140, 190}
\definecolor{colorc}{RGB}{227, 74, 51}

\pgfplotsset{PcwConst/.style={color=colorc, mark=triangle*,mark size=2}}
\pgfplotsset{PcwLin/.style={color=colorl, mark=square*,mark size=1.5}}
\pgfplotsset{PcwQuad/.style={color=colorq, mark=*,mark size=1.5}}
\pgfplotsset{Dashed/.style={dashed, color=black, mark options = solid, mark=x, mark size=3}}
\pgfplotsset{Dashed2/.style={color=gray, dash pattern=on 1pt off 1pt on 1pt off 1pt, mark options={solid} }}
\pgfplotsset{PcwConstComp/.style={color=colorc, mark=triangle,mark size=2, dash pattern=on 4pt off 1pt on 1pt off 1pt, mark options={solid} }}
\pgfplotsset{PcwLinComp/.style={color=colorl, mark=square,mark size=1.5, dash pattern=on 4pt off 1pt on 1pt off 1pt, mark options={solid}}}
\pgfplotsset{PcwQuadComp/.style={color=colorq, mark=o,mark size=1.5, dash pattern=on 4pt off 1pt on 1pt off 1pt, mark options={solid}}}

\tikzstyle{Dashed3}=[color=black, dash pattern=on 1pt off 1pt on 1pt off 1pt, mark options={solid}, line width = 0.5]
\tikzstyle{thick}=[draw=black,line width=0.95,opacity=1.0]
\tikzstyle{DashedDraw}=[draw=black, densely dashed, line width=0.15]

\pgfdeclarelayer{bg}
\pgfsetlayers{bg,main}

\newcommand{\OO}[1]{\mathcal{O}(#1)}
\newcommand{\Boxvars}[1]{\mathcal{#1}}

\newcommand{\Matrix}[1]{\mathbf{#1}}
\def \Amat {\mathbf{A}}
\def \Boxvar {\mathcal{B}}
\def \xsol {\mathbf{x}}
\def \Poly {\mathbf{\Pi}}

\newcommand{\algfull}{Sparse Geometric Factorization}
\newcommand{\algo}{SGF}
\newcommand{\nmat}{filtered interaction matrix}
\newcommand{\nmats}{filtered interaction matrices}
\newcommand{\schemeI}{Nest-2-All}
\newcommand{\schemeIa}{Nest-All-All}
\newcommand{\schemeIb}{Nest-2-2}
\newcommand{\schemeII}{Gen-All-All}

\newsiamremark{remark}{Remark}
\newsiamremark{hypothesis}{Hypothesis}
\newsiamthm{claim}{Claim}
\newsiamthm{observation}{Observation}

\headers{Preconditioners Using Smooth Approximations}{B.Klockiewicz, and  E. Darve}

\title{Sparse Hierarchical Preconditioners Using Piecewise Smooth Approximations of Eigenvectors}

\author{Bazyli Klockiewicz\thanks{Institute for Computational \& Mathematical Engineering, Stanford University, \email{bazyli@stanford.edu} \url{}}
\and Eric Darve \thanks{Department of Mechanical Engineering, Stanford University, \email{darve@stanford.edu} \url{}}}

\algblockdefx{FORALLPAR}{ENDFORPAR}[1]
  {\textbf{for all }#1 \textbf{do in parallel}}
  {\textbf{end for}}
 \algloopdefx{RETURN}[1][]{\textbf{return} #1}

\usepgfplotslibrary{external} 
\ifpdf
\hypersetup{
  pdftitle={},
  pdfauthor={}
}
\fi

\begin{document}

\maketitle
\begin{abstract}
When solving linear systems arising from PDE discretizations, iterative methods (such as Conjugate Gradient, GMRES, or MINRES) are often the only practical choice. To converge in a small number of iterations, however, they have to be coupled with an efficient preconditioner. The efficiency of the preconditioner depends largely on its accuracy on the eigenvectors corresponding to small eigenvalues, and unfortunately, black-box methods typically cannot guarantee sufficient accuracy on these eigenvectors. Thus, constructing the preconditioner becomes a problem-dependent task. However, for a large class of problems, including many elliptic equations, the eigenvectors corresponding to small eigenvalues are smooth functions of the PDE grid. In this paper, we describe a hierarchical approximate factorization approach which focuses on improving accuracy on the smooth eigenvectors. The improved accuracy is achieved by preserving the action of the factorized matrix on piecewise polynomial functions of the grid. Based on the factorization, we propose a family of sparse preconditioners with $\OO{n}$ or $\OO{n \log{n}}$ construction complexities. Our methods exhibit near optimal scalings of solution times in benchmarks run on large elliptic problems, arising for example in flow or mechanical simulations. In the case of the linear elasticity equation the preconditioners are exact on the near-kernel rigid body modes.
\end{abstract}



\section{Introduction}\label{sec:intro}
A significant class of problems in engineering lead to large and sparse symmetric positive definite (SPD) systems: \begin{equation} \label{eq:ax=b} \Amat \xsol = \Matrix{b} \end{equation} where $\Amat \in \mathbb{R}^{n \times n}$ is a sparse SPD matrix, $\Matrix{b} \in \mathbb{R}^n$, and $\Matrix{x} \in \mathbb{R}^n$ is the desired unknown solution. In particular, we are interested in the discretizations of second-order elliptic partial differential equations (PDEs) obtained using finite stencil, finite volume, or finite element methods. Examples include the Laplace, elasticity, or (some cases of) Maxwell equations. A substantial effort in scientific computing has been devoted to efficiently solve \autoref{eq:ax=b} arising from such PDEs.

\subsection{Previous work}\label{sec:previous}
The most reliable method for solving \autoref{eq:ax=b} is the (exact) Cholesky factorization. A naive implementation has $\OO{n^3}$ computational cost but sparsity can be exploited to limit the fill-ins and reduce the cost. Many methods have been designed to limit the fill-ins based on appropriately ordering the variables \cite{gilbert1992highly,ng1993block}. In the context of PDEs, an efficient method is nested dissection \cite{george1973nested,lipton1979generalized}, which can reduce the costs to $\OO{n^{3/2}}$ in 2D, and $\OO{n^2}$ in 3D, under mild assumptions. In fact, nested dissection is at the heart of many state-of-the-art direct solvers \cite{amestoy2000mumps,chen2008algorithm,henon2002pastix}, which are useful for small to middle-size systems. Still, the $\OO{n^2}$ complexity becomes impractical for large-scale problems.

Another group of approaches are the iterative methods such as multigrid \cite{hackbusch2013multi,brandt1977multi,xu1992iterative} or the Krylov-space methods. In particular, algebraic multigrid (AMG) \cite{ruge1987algebraic,vanvek1996algebraic,stuben2001review} does not require the knowledge of the grid geometry and removes some limitations of its predecessor, geometric multigrid. AMG can often achieve the optimal $\OO{n}$ solution cost, with good parallel scalability and low memory requirements. To ensure convergence, however, one has to properly design the smoothing and coarsening strategy. As a result, AMG often has to be fine-tuned or extended to be efficient for specific equations, or problems in question. Krylov-space methods on the other hand, such as Conjugate Gradient \cite{hestenes1952methods}, GMRES \cite{saad1986gmres}, or MINRES \cite{paige1975solution}, are a general group of approaches which utilize only sparse matrix-vector products. Convergence is very sensitive, however, to the conditioning of the given system, and in practice these methods have to be coupled with preconditioners. A popular class of black-box preconditioners are the ones based on incomplete factorizations such as incomplete Cholesky \cite{kershaw1978incomplete,jones1995improved,gustafsson1978class}, or more generally, incomplete LU (ILU) \cite{saad1994ilut,dupont1968approximate}. However, these preconditioners have rather limited applicability, not being able to target the whole spectrum of $\Amat$, and typically specialized problem-dependent preconditioners have to be developed. In fact, in the context of elliptic PDEs, the multigrid methods often prove effective preconditioners.

Another group of preconditioners are hierarchical algorithms which exploit the fact that in the context of elliptic PDEs, certain off-diagonal blocks of $\Amat$ or $\Amat^{-1}$ are numerically low-rank \cite{chandrasekaran2010numerical,bebendorf2005efficient,bebendorf2003existence,borm2010approximation}. 
Hierarchical approaches do not typically make any other assumptions about the system in question and therefore they can be more robust than multigrid methods, e.g., in the presence of non-smooth coefficients, or strong anisotropies. The theoretical framework for hierarchical approaches is provided by $\mathcal{H}$- and $\mathcal{H}^2$-matrices \cite{hackbusch1999sparse,hackbusch2002data} (developed originally for integral equations). They allow for performing algebraic operations on matrices with low-rank structure of off-diagonal blocks or well-separated blocks. In particular, the LU factorization (and applying the inverse) can be performed with linear or quasilinear complexity \cite{grasedyck2009domain} (typically with a given accuracy). In practice, however, the constants involved in the asymptotic scalings may be somewhat large due to the recursive nature of the algorithms which also require specific data-sparse formats for storing the matrices.

Recently, new hierarchical approaches have been proposed that concentrate on efficiently factorizing the matrix using the low-rank structure of the off-diagonal blocks (we further call them \emph{hierarchical solvers}). Hierarchical Interpolative Factorization \cite{ho2016hierarchical} is a sparsified nested dissection multifrontal approach which successively reduces the sizes of the separating fronts. LoRaSp \cite{pouransari2017fast} is a different approach, using extended sparsification, related to the inverse fast multipole method \cite{coulier2017inverse} (a similar method was also described in \cite{sushnikova2018compress}). These methods leverage the low-rank properties to factorize the matrix approximately into products of sparse block triangular factors, obviating the need for hierarchical data-sparse matrix formats. More robust extensions have been proposed since \cite{chen2018distributed,cambier2019spand,feliu2018recursively}. In particular, \cite{cambier2019spand} proposed Sparsified Nested Dissection that is guaranteed to never break and can be applied to any SPD matrix.

However, hierarchical solvers still lack good convergence guarantees. For example, one can prove that---on certain elliptic equations---geometric multigrid leads to a preconditioned system with a bounded condition number. The hierarchical approaches, on the other hand, often cannot guarantee rapid convergence of iterative methods because the accuracy of the approximation $\Amat_{\ell} \approx \Amat$ is controlled in the following sense:
\begin{equation}\label{eq:hs-approx} 
\| \Amat_{\ell} - \Amat \|_2 \leq \varepsilon \| \Amat \|_2
\end{equation}
whereas in fact, a stronger criterion is needed:
\begin{equation}\label{eq:crit}
\| \Matrix{I} - \Amat_{\ell}^{-1/2} \Amat \Amat_{\ell}^{-1/2} \|_2 \leq \varepsilon
\end{equation}
In particular, this means that $\Amat_{\ell}$ needs to be particularly accurate on the eigenvectors corresponding to small eigenvalues (the \emph{near-kernel eigenspace}), which is not assured by \autoref{eq:hs-approx}. However, for many elliptic PDEs, these eigenvectors are smooth and this property can be taken advantage of, to ensure \autoref{eq:crit}. This would make hierarchical solvers truly competitive, by guaranteeing a bounded number of Conjugate Gradient iterations, for instance.

\subsection{Contributions}
We introduce a hierarchical approach to approximately
factorize sparse symmetric positive definite (SPD) matrices arising from elliptic PDE
discretizations. We call it \emph{\algfull}{}, or \algo{} for short. Based on the factorization, we propose a family of preconditioners with varying accuracies and sparsity patterns. Depending on the chosen sparsity pattern, the preconditioners can be computed in $\OO{n}$ or $\OO{n \log{n}}$ operations where $n$ is the number of unknowns, with $\OO{n}$ memory requirements. The obtained cheaply-invertible operator $\Amat_{\ell} \approx \Amat$ retains the SPD property of $\Amat$, and the factorization always succeeds, with all pivots guaranteed to be invertible in exact arithmetic (in this sense, it follows \cite{cambier2019spand,chen2018distributed}).

\algo{} shares the general framework with \cite{cambier2019spand,ho2016hierarchical,chen2018distributed}, ensuring however that the obtained operator $\Amat_{\ell}\approx \Amat$ is also accurate on the critical near-kernel smooth eigenvectors, which addresses the limitation of hierarchical approaches mentioned above. Our choice of approximating the near-kernel subspace are the piecewise polynomial functions of the grid (e.g., piecewise constant, or piecewise linear vectors), but the space approximating the near-kernel subspace can be supplied by the user.

In the context of $\mathcal{H}$-matrix preconditioners for elliptic PDE operators, an improved accuracy on the near-kernel eigenspace can be achieved by preserving the action of $\Amat$ on piecewise constant vectors when defining the low-rank bases (see e.g. \cite{bebendorf2013constraints,bebendorf2016spectral} for details). An attempt to incorporate these ideas into hierarchical solvers was made in \cite{yang2016sparse}, which improved the robustness of LoRaSp \cite{pouransari2017fast} by preserving the action of $\Amat$ on the globally constant vector throughout computations. In our approach, sparsifying the off-diagonal blocks in the factorization is driven primarily by requiring that the action of $\Amat$ on piecewise polynomial functions be preserved. 

Compared to previous work~\cite{bebendorf2013constraints,bebendorf2016spectral,yang2016sparse}, we also investigate more options to select different types of blocks for compressions.  In particular, our methods allow sparsifying matrix blocks that cannot be expected to be low-rank, still resulting in effective preconditioners. Our family of preconditioners adapts the grid partitioning methods from \cite{cambier2019spand,ho2016hierarchical}, unifying them into one framework, and adding new approaches that allow for better control of the preconditioner sparsity pattern. To keep the paper focused, the partitionings are defined on cartesian grids and relatively simple stencils. However, the \algo{} algorithm is very general and does not assume any particular grid structure or discretization scheme, and can be applied to any SPD matrix~\cite{cambier2019spand}.

We benchmark the preconditioners on large systems of different types. The iteration counts of preconditioned Conjugate Gradient grow very slowly with the system sizes, and the solution times scale roughly as $\OO{n}$, also on ill-conditioned systems. We also compare our methods to equivalent approaches using rank-revealing decompositions to approximate the off-diagonal blocks, as in other hierarchical algorithms \cite{cambier2019spand,ho2016hierarchical,chen2018distributed,sushnikova2018compress,pouransari2017fast}. Our methods scale significantly better with growing system sizes on most tested problems.
While only briefly mentioned in this paper, we believe that our preconditioners have very promising parallelization properties (inherited from \cite{ho2016hierarchical,pouransari2017fast,cambier2019spand}). Some possible parallelization strategies have been described in \cite{chen2018distributed,li2017distributed}.

\subsection {Organization of the paper} An overview of the \algo{} algorithm is described in \autoref{sec:idea} along with some useful definitions. In \autoref{sec:compression} we describe our method to preserve piecewise polynomial vectors in the factorization. The \algo{} algorithm is then described in detail in \autoref{sec:detailed-description}. The proposed family of preconditioners, with specific realizations of the generic factorization algorithm, is described in \autoref{sec:solver-family}. Some numerical analysis is provided in \autoref{sec:app-quality}. Numerical experiments are presented in \autoref{sec:experiments}, followed by conclusions in \autoref{sec:conclusions}. 

\section{ Overview of \algfull{} (\algo)}\label{sec:idea} In the PDE applications, it is often convenient to consider a partition of the set of unknowns $V = \{1,2,\ldots,n\}$ into small disjoint subsets (here called \emph{nodes}) $V = B_1 \cup B_2 \cup \cdots \cup B_t$. The partition naturally induces a block representation of $\Amat$. The \emph{interaction matrix} $\Amat_{B_i B_j}$ between $B_i$ and $B_j$ will typically be nonzero only if the stencils of the discretization grid involving $B_i$ and $B_j$, are close to each other, resulting in sparse $\Amat$. 

\subsection{Algorithm} \label{sec:overview}  Given the framework, we perform a sparse approximate factorization of $\Amat,$ proceeding in a hierarchical fashion: 
\begin{enumerate}
\item (\emph{Elimination step}) Eliminate selected nodes (called interior nodes) using the block Cholesky factorization, to obtain $\Amat = \Matrix{G} \Amat_{\left( 1/2 \right)} \Matrix{G}^T$.\label{it:elim}
\item (\emph{Compression step}) Sparsify selected remaining nodes to obtain $\Amat_{\left(1/2 \right)} \approx \Matrix{B} \Amat_{(1)} \Matrix{B}^T,$ where $\Amat_{(1)}$ is sparser than $ \Amat_{\left( 1/2 \right)}$, and $\Matrix{B}$ is sparse block-diagonal.\label{it:comp}
\item Form a new coarser partition. If the new partition is composed of a single element, factorize $\Amat_{(1)}$ using exact Cholesky factorization. \item Otherwise recurse on the new partition and the submatrix of $\Amat_{(1)}$ corresponding to the not-yet-eliminated variables.\label{it:form}
\end{enumerate}
After all $\ell = \OO{\log{n}}$ iterations have been completed (also referred to as \emph{levels}), the obtained approximate operator, which is also SPD, is of the form:
\[ 
    \Amat \approx \Amat_{\ell} = \Matrix{G}_1 \Matrix{B}_1
\Matrix{G}_2 \Matrix{B}_2 \cdots \Matrix{G}_{\ell - 1} \Matrix{B}_{\ell -1} \Matrix{G}_{\ell}
\Matrix{G}_{\ell}^T \Matrix{B}^T_{\ell -1} \Matrix{G}^T_{\ell -1} \cdots \Matrix{B}^T_2
\Matrix{G}^T_2 \Matrix{B}^T_1\Matrix{G}^T_1
\]
The factorization is illustrated in \autoref{fig:illustration}. The partitions can be based on nested dissection \cite{george1973nested} which determines the interior nodes eliminated in \emph{Step~\ref{it:elim}} above. Elimination of an interior node introduces fill-in only between its neighboring separators. Simple domain partitioning may be used as well in which \emph{Step~\ref{it:elim}} is skipped.
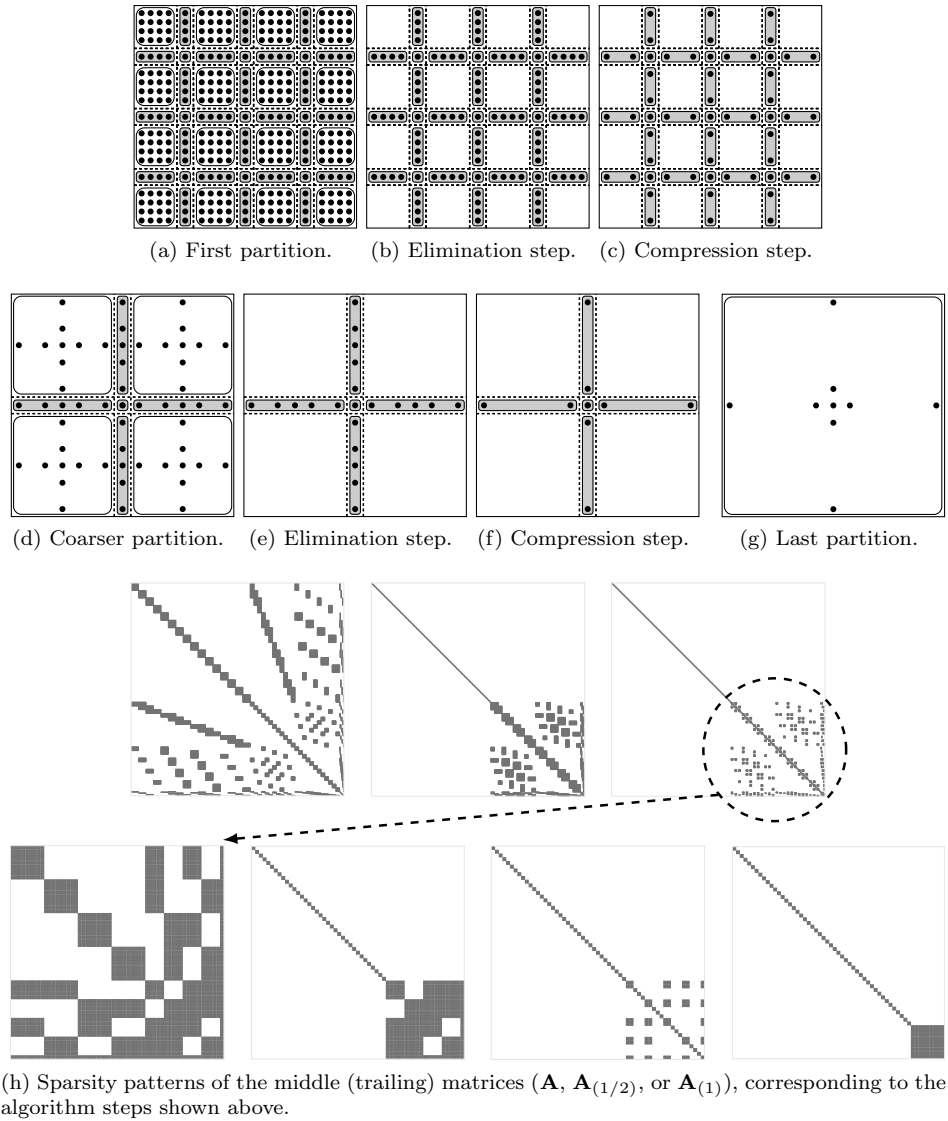
\begin{figure}[htbp]
\centering
\subfloat[First partition.]{
	\begin{tikzpicture}
	\def \p{1.48} 
	\def \m{0.12} 
	\def \d{0.1} 
	\def \n{0.8} 
	\def \k{0.22} 
	\def \o{0.01} 
	\def \oo{0.03} 
	\def \a{0.36} 
	\def \b{0.115} 
	\draw (-\p,-\p) rectangle (\p,\p);
	\foreach \t in {0,...,3} {
		\foreach \s in {0,...,3} {		
			\draw [rounded corners=3pt,fill=black!40,fill opacity=0] 
			(\n*\t-\p+\oo,\n*\s-\p+\oo) rectangle (\n*\t-\p+\a*\p,\n*\s-\p+\a*\p);
			\foreach \x in {0,...,3} {
				\foreach \y in {0,...,3} {
					\node at (-\p+\n*\t+\m*\x+\d,-\p+\n*\s+\m*\y+\d) [circle,fill=black,inner sep=0pt,minimum size=0.08cm] {};
				}
			}
		}
	}
	\foreach \t in {1,...,3} {
		\foreach \s in {0,...,3} {
			\draw[Dashed3] (\n*\t-\p-\k-\o,-\p) -- (\n*\t-\p-\k-\o,\p);
			\draw[Dashed3] (\n*\t-\p-\o,-\p) -- (\n*\t-\p-\o,\p);		
			\draw [rounded corners=1pt,fill=black!40,fill opacity=0.5] 
			(\n*\t-\p-\k+\oo,\n*\s-\p+\oo) rectangle (\n*\t-\p+\b*\p-\k,\n*\s-\p+\a*\p);		
			\foreach \y in {0,...,3} {
				\node at (-\p+\n*\t-\k+\d,-\p+\n*\s+\m*\y+\d) [circle,fill=black,inner sep=0pt,minimum size=0.08cm] {};
			}
		}
	}
	\foreach \t in {1,...,3} {
		\foreach \s in {0,...,3} {
			\draw[Dashed3] (-\p,\n*\t-\p-\k-\o) -- (\p,\n*\t-\p-\k-\o);	
			\draw[Dashed3] (-\p,\n*\t-\p-\k+2*\d+\o) -- (\p,\n*\t-\p-\k+2*\d+\o);				
			\draw [rounded corners=1pt,fill=black!40,fill opacity=0.5] 
			(\n*\s-\p+\oo,\n*\t-\p-\k+\oo) rectangle (\n*\s-\p+\a*\p,\n*\t-\p+\b*\p-\k);	
			\foreach \y in {0,...,3} {
				\node at (-\p+\n*\s+\m*\y+\d,-\p+\n*\t-\k+\d) [circle,fill=black,inner sep=0pt,minimum size=0.08cm] {};
			}
		}
	}
	\foreach \t in {1,...,3} {
		\foreach \s in {1,...,3} {
			\draw [rounded corners=1pt,fill=black!40,fill opacity=0.5] 
			(\n*\t-\p-\k+\oo,\n*\s-\p-\k+\oo) rectangle (\n*\t-\p-\k+\b*\p,\n*\s-\p-\k+\b*\p);
			\node at (-\p+\n*\t-\k+\d,-\p+\n*\s-\k+\d) [circle,fill=black,inner sep=0pt,minimum size=0.08cm] {};		
		}
	}
\end{tikzpicture}}
\subfloat[Elimination step.]{
	\begin{tikzpicture}
	\def \p{1.48} 
	\def \m{0.12} 
	\def \d{0.1} 
	\def \n{0.8} 
	\def \k{0.22} 
	\def \o{0.01} 
	\def \oo{0.03} 
	\def \a{0.36} 
	\def \b{0.115} 
	\draw (-\p,-\p) rectangle (\p,\p);
	\foreach \t in {1,...,3} {
		\foreach \s in {0,...,3} {
			\draw[Dashed3] (\n*\t-\p-\k-\o,-\p) -- (\n*\t-\p-\k-\o,\p);
			\draw[Dashed3] (\n*\t-\p-\o,-\p) -- (\n*\t-\p-\o,\p);			
			\draw [rounded corners=1pt,fill=black!40,fill opacity=0.5] 
			(\n*\t-\p-\k+\oo,\n*\s-\p+\oo) rectangle (\n*\t-\p+\b*\p-\k,\n*\s-\p+\a*\p);
			\foreach \y in {0,...,3} {
				\node at (-\p+\n*\t-\k+\d,-\p+\n*\s+\m*\y+\d) [circle,fill=black,inner sep=0pt,minimum size=0.08cm] {};
			}
		}
	}
	\foreach \t in {1,...,3} {
		\foreach \s in {0,...,3} {
			\draw[Dashed3] (-\p,\n*\t-\p-\k-\o) -- (\p,\n*\t-\p-\k-\o);	
			\draw[Dashed3] (-\p,\n*\t-\p-\k+2*\d+\o) -- (\p,\n*\t-\p-\k+2*\d+\o);		
			\draw [rounded corners=1pt,fill=black!40,fill opacity=0.5] 
			(\n*\s-\p+\oo,\n*\t-\p-\k+\oo) rectangle (\n*\s-\p+\a*\p,\n*\t-\p+\b*\p-\k);
			\foreach \y in {0,...,3} {
				\node at (-\p+\n*\s+\m*\y+\d,-\p+\n*\t-\k+\d) [circle,fill=black,inner sep=0pt,minimum size=0.08cm] {};
			}
		}
	}
	\foreach \t in {1,...,3} {
		\foreach \s in {1,...,3} {
			\draw [rounded corners=1pt,fill=black!40,fill opacity=0.5] 
			(\n*\t-\p-\k+\oo,\n*\s-\p-\k+\oo) rectangle (\n*\t-\p-\k+\b*\p,\n*\s-\p-\k+\b*\p);		
			\node at (-\p+\n*\t-\k+\d,-\p+\n*\s-\k+\d) [circle,fill=black,inner sep=0pt,minimum size=0.08cm] {};		
		}
	}
\end{tikzpicture}}
\subfloat[Compression step.]{
	\begin{tikzpicture}
	\def \p{1.48} 
	\def \m{0.35} 
	\def \d{0.1} 
	\def \n{0.8} 
	\def \k{0.22} 
	\def \o{0.01} 
	\def \oo{0.03} 
	\def \a{0.36} 
	\def \b{0.115} 
	\draw (-\p,-\p) rectangle (\p,\p);
	\foreach \t in {1,...,3} {
		\foreach \s in {0,...,3} {
			\draw[Dashed3] (\n*\t-\p-\k-\o,-\p) -- (\n*\t-\p-\k-\o,\p);
			\draw[Dashed3] (\n*\t-\p-\o,-\p) -- (\n*\t-\p-\o,\p);		
			\draw [rounded corners=1pt,fill=black!40,fill opacity=0.5] 
			(\n*\t-\p-\k+\oo,\n*\s-\p+\oo) rectangle (\n*\t-\p+\b*\p-\k,\n*\s-\p+\a*\p);
			\foreach \y in {0,...,1} {
				\node at (-\p+\n*\t-\k+\d,-\p+\n*\s+\m*\y+\d) [circle,fill=black,inner sep=0pt,minimum size=0.08cm] {};
			}
		}
	}
	\foreach \t in {1,...,3} {
		\foreach \s in {0,...,3} {
			\draw[Dashed3] (-\p,\n*\t-\p-\k-\o) -- (\p,\n*\t-\p-\k-\o);	
			\draw[Dashed3] (-\p,\n*\t-\p-\k+2*\d+\o) -- (\p,\n*\t-\p-\k+2*\d+\o);		
			\draw [rounded corners=1pt,fill=black!40,fill opacity=0.5] 
			(\n*\s-\p+\oo,\n*\t-\p-\k+\oo) rectangle (\n*\s-\p+\a*\p,\n*\t-\p+\b*\p-\k);
			\foreach \y in {0,...,1} {
				\node at (-\p+\n*\s+\m*\y+\d,-\p+\n*\t-\k+\d) [circle,fill=black,inner sep=0pt,minimum size=0.08cm] {};
			}
		}
	}
	\foreach \t in {1,...,3} {
		\foreach \s in {1,...,3} {
			\draw [rounded corners=1pt,fill=black!40,fill opacity=0.5] 
			(\n*\t-\p-\k+\oo,\n*\s-\p-\k+\oo) rectangle (\n*\t-\p-\k+\b*\p,\n*\s-\p-\k+\b*\p);	
			\node at (-\p+\n*\t-\k+\d,-\p+\n*\s-\k+\d) [circle,fill=black,inner sep=0pt,minimum size=0.08cm] {};		
		}
	}
\end{tikzpicture}}\,
\subfloat[Coarser partition.]{
	\begin{tikzpicture}
	\def \p{1.48} 
	\def \m{0.35} 
	\def \d{0.1} 
	\def \n{0.8} 
	\def \k{0.22} 
	\def \o{0.01} 
	\def \oo{0.03} 
	\def \a{0.9} 
	\def \b{0.115} 
	\draw (-\p,-\p) rectangle (\p,\p);
	\foreach \t in {0,2} {
		\foreach \s in {0,2} {		
			\draw [rounded corners=3pt,fill=black!40,fill opacity=0] 
			(\n*\t-\p+\oo,\n*\s-\p+\oo) rectangle (\n*\t-\p+\a*\p,\n*\s-\p+\a*\p);
		}
	}
	\foreach \t in {2} {
		\foreach \s in {0,2} {
			\draw[Dashed3] (\n*\t-\p-\k-\o,-\p) -- (\n*\t-\p-\k-\o,\p);
			\draw[Dashed3] (\n*\t-\p-\o,-\p) -- (\n*\t-\p-\o,\p);
			\draw [rounded corners=1pt,fill=black!40,fill opacity=0.5] 
			(\n*\t-\p-\k+\oo,\n*\s-\p+\oo) rectangle (\n*\t-\p+\b*\p-\k,\n*\s-\p+\a*\p);
		}
	}
	\foreach \t in {1,...,3} {
		\foreach \s in {0,...,3} {
			\foreach \y in {0,...,1} {
				\node at (-\p+\n*\t-\k+\d,-\p+\n*\s+\m*\y+\d) [circle,fill=black,inner sep=0pt,minimum size=0.08cm] {};
			}
		}
	}
	\foreach \t in {2} {
		\foreach \s in {0,2} {
			\draw[Dashed3] (-\p,\n*\t-\p-\k-\o) -- (\p,\n*\t-\p-\k-\o);	
			\draw[Dashed3] (-\p,\n*\t-\p-\k+2*\d+\o) -- (\p,\n*\t-\p-\k+2*\d+\o);	
			\draw [rounded corners=1pt,fill=black!40,fill opacity=0.5] 
			(\n*\s-\p+\oo,\n*\t-\p-\k+\oo) rectangle (\n*\s-\p+\a*\p,\n*\t-\p+\b*\p-\k);
		}
	}
	\foreach \t in {1,...,3} {
		\foreach \s in {0,...,3} {
			\foreach \y in {0,...,1} {
				\node at (-\p+\n*\s+\m*\y+\d,-\p+\n*\t-\k+\d) [circle,fill=black,inner sep=0pt,minimum size=0.08cm] {};
			}
		}
	}
	\foreach \t in {2} {
		\foreach \s in {2} {
			\draw [rounded corners=1pt,fill=black!40,fill opacity=0.5] 
			(\n*\t-\p-\k+\oo,\n*\s-\p-\k+\oo) rectangle (\n*\t-\p-\k+\b*\p,\n*\s-\p-\k+\b*\p);	
			\node at (-\p+\n*\t-\k+\d,-\p+\n*\s-\k+\d) [circle,fill=black,inner sep=0pt,minimum size=0.08cm] {};		
		}
	}
	\foreach \t in {1,...,3} {
		\foreach \s in {1,...,3} {
			\node at (-\p+\n*\t-\k+\d,-\p+\n*\s-\k+\d) [circle,fill=black,inner sep=0pt,minimum size=0.08cm] {};		
		}
	}
\end{tikzpicture}}
\subfloat[Elimination step.]{
	\begin{tikzpicture}
	\def \p{1.48} 
	\def \m{0.35} 
	\def \d{0.1} 
	\def \n{0.8} 
	\def \k{0.22} 
	\def \o{0.01} 
	\def \oo{0.03} 
	\def \a{0.9} 
	\def \b{0.115} 
	\draw (-\p,-\p) rectangle (\p,\p);
	\foreach \t in {2} {
		\foreach \s in {0,2} {
			\draw[Dashed3] (\n*\t-\p-\k-\o,-\p) -- (\n*\t-\p-\k-\o,\p);
			\draw[Dashed3] (\n*\t-\p-\o,-\p) -- (\n*\t-\p-\o,\p);
			\draw [rounded corners=1pt,fill=black!40,fill opacity=0.5] 
			(\n*\t-\p-\k+\oo,\n*\s-\p+\oo) rectangle (\n*\t-\p+\b*\p-\k,\n*\s-\p+\a*\p);
		}
	}
	\foreach \t in {2} {
		\foreach \s in {0,...,3} {
			\foreach \y in {0,...,1} {
				\node at (-\p+\n*\t-\k+\d,-\p+\n*\s+\m*\y+\d) [circle,fill=black,inner sep=0pt,minimum size=0.08cm] {};
			}
		}
	}
	\foreach \t in {2} {
		\foreach \s in {0,2} {
			\draw[Dashed3] (-\p,\n*\t-\p-\k-\o) -- (\p,\n*\t-\p-\k-\o);	
			\draw[Dashed3] (-\p,\n*\t-\p-\k+2*\d+\o) -- (\p,\n*\t-\p-\k+2*\d+\o);
			\draw [rounded corners=1pt,fill=black!40,fill opacity=0.5] 
			(\n*\s-\p+\oo,\n*\t-\p-\k+\oo) rectangle (\n*\s-\p+\a*\p,\n*\t-\p+\b*\p-\k);
		}
	}
	\foreach \t in {2} {
		\foreach \s in {0,...,3} {
			\foreach \y in {0,...,1} {
				\node at (-\p+\n*\s+\m*\y+\d,-\p+\n*\t-\k+\d) [circle,fill=black,inner sep=0pt,minimum size=0.08cm] {};
			}
		}
	}
	\foreach \t in {2} {
		\foreach \s in {1,3} {
			\node at (-\p+\n*\t-\k+\d,-\p+\n*\s-\k+\d) [circle,fill=black,inner sep=0pt,minimum size=0.08cm] {};		
		}
	}
	\foreach \t in {1,3} {
		\foreach \s in {2} {
			\node at (-\p+\n*\t-\k+\d,-\p+\n*\s-\k+\d) [circle,fill=black,inner sep=0pt,minimum size=0.08cm] {};		
		}
	}

	\foreach \t in {2} {
		\foreach \s in {2} {
			\draw [rounded corners=1pt,fill=black!40,fill opacity=0.5] 
			(\n*\t-\p-\k+\oo,\n*\s-\p-\k+\oo) rectangle (\n*\t-\p-\k+\b*\p,\n*\s-\p-\k+\b*\p);	
			\node at (-\p+\n*\t-\k+\d,-\p+\n*\s-\k+\d) [circle,fill=black,inner sep=0pt,minimum size=0.08cm] {};		
		}
	}	
\end{tikzpicture}}
\subfloat[Compression step.]{	\begin{tikzpicture}
	\def \p{1.48} 
	\def \m{1.15} 
	\def \d{0.1} 
	\def \n{0.8} 
	\def \k{0.22} 
	\def \o{0.01} 
	\def \oo{0.03} 
	\def \a{0.9} 
	\def \b{0.115} 
	\draw (-\p,-\p) rectangle (\p,\p);
	\foreach \t in {2} {
		\foreach \s in {0,2} {
			\draw[Dashed3] (\n*\t-\p-\k-\o,-\p) -- (\n*\t-\p-\k-\o,\p);
			\draw[Dashed3] (\n*\t-\p-\o,-\p) -- (\n*\t-\p-\o,\p);
			\draw [rounded corners=1pt,fill=black!40,fill opacity=0.5] 
			(\n*\t-\p-\k+\oo,\n*\s-\p+\oo) rectangle (\n*\t-\p+\b*\p-\k,\n*\s-\p+\a*\p);
		}
	}
	\foreach \t in {2} {
		\foreach \s in {0,2} {
			\foreach \y in {0,...,1} {
				\node at (-\p+\n*\t-\k+\d,-\p+\n*\s+\m*\y+\d) [circle,fill=black,inner sep=0pt,minimum size=0.08cm] {};
			}
		}
	}
	\foreach \t in {2} {
		\foreach \s in {0,2} {
			\draw[Dashed3] (-\p,\n*\t-\p-\k-\o) -- (\p,\n*\t-\p-\k-\o);	
			\draw[Dashed3] (-\p,\n*\t-\p-\k+2*\d+\o) -- (\p,\n*\t-\p-\k+2*\d+\o);
			\draw [rounded corners=1pt,fill=black!40,fill opacity=0.5] 
			(\n*\s-\p+\oo,\n*\t-\p-\k+\oo) rectangle (\n*\s-\p+\a*\p,\n*\t-\p+\b*\p-\k);
		}
	}
	\foreach \t in {2} {
		\foreach \s in {0,2} {
			\foreach \y in {0,...,1} {
				\node at (-\p+\n*\s+\m*\y+\d,-\p+\n*\t-\k+\d) [circle,fill=black,inner sep=0pt,minimum size=0.08cm] {};
			}
		}
	}
	\foreach \t in {2} {
		\foreach \s in {2} {
			\draw [rounded corners=1pt,fill=black!40,fill opacity=0.5] 
			(\n*\t-\p-\k+\oo,\n*\s-\p-\k+\oo) rectangle (\n*\t-\p-\k+\b*\p,\n*\s-\p-\k+\b*\p);	
			\node at (-\p+\n*\t-\k+\d,-\p+\n*\s-\k+\d) [circle,fill=black,inner sep=0pt,minimum size=0.08cm] {};		
		}
	}
	\foreach \t in {2} {
		\foreach \s in {2} {
			\node at (-\p+\n*\t-\k+\d,-\p+\n*\s-\k+\d) [circle,fill=black,inner sep=0pt,minimum size=0.08cm] {};		
		}
	}
\end{tikzpicture}}\,
\subfloat[Last partition.]{	\begin{tikzpicture}
	\def \p{1.48} 
	\def \m{1.15} 
	\def \d{0.1} 
	\def \n{0.8} 
	\def \k{0.22} 
	\def \o{0.01} 
	\def \oo{0.03} 
	\def \a{1.975} 
	\def \b{0.115} 
	\draw (-\p,-\p) rectangle (\p,\p);
	\foreach \t in {0} {
		\foreach \s in {0} {		
			\draw [rounded corners=3pt,fill=black!40,fill opacity=0] 
			(\n*\t-\p+\oo,\n*\s-\p+\oo) rectangle (\n*\t-\p+\a*\p,\n*\s-\p+\a*\p);
		}
	}
	\foreach \t in {2} {
		\foreach \s in {0,2} {
			\foreach \y in {0,...,1} {
				\node at (-\p+\n*\t-\k+\d,-\p+\n*\s+\m*\y+\d) [circle,fill=black,inner sep=0pt,minimum size=0.08cm] {};
			}
		}
	}
	\foreach \t in {2} {
		\foreach \s in {0,2} {
			\foreach \y in {0,...,1} {
				\node at (-\p+\n*\s+\m*\y+\d,-\p+\n*\t-\k+\d) [circle,fill=black,inner sep=0pt,minimum size=0.08cm] {};
			}
		}
	}
	\foreach \t in {2} {
		\foreach \s in {2} {
			\node at (-\p+\n*\t-\k+\d,-\p+\n*\s-\k+\d) [circle,fill=black,inner sep=0pt,minimum size=0.08cm] {};		
		}
	}
\end{tikzpicture}}\,
\subfloat[Sparsity patterns of the middle (trailing) matrices ($\Amat$, $\Amat_{(1/2)}$, or $\Amat_{(1)}$), corresponding to the algorithm steps shown above.]{
\begin{tikzpicture}
\node[inner sep=0pt] at (0,0) {
\includegraphics[width=0.22\textwidth]{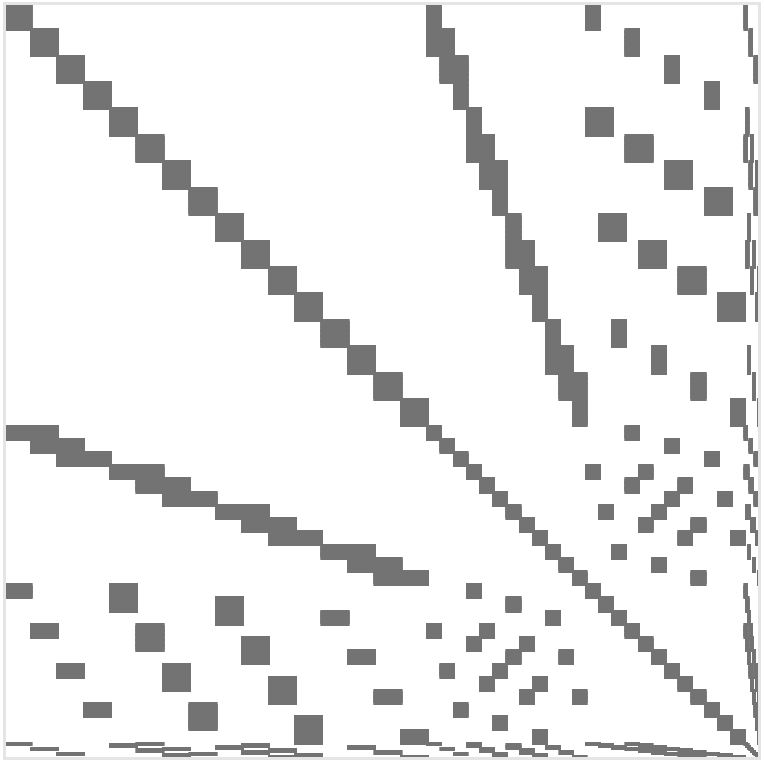}};
\node[inner sep=0pt] at (3.2,0) {
\includegraphics[width=0.22\textwidth]{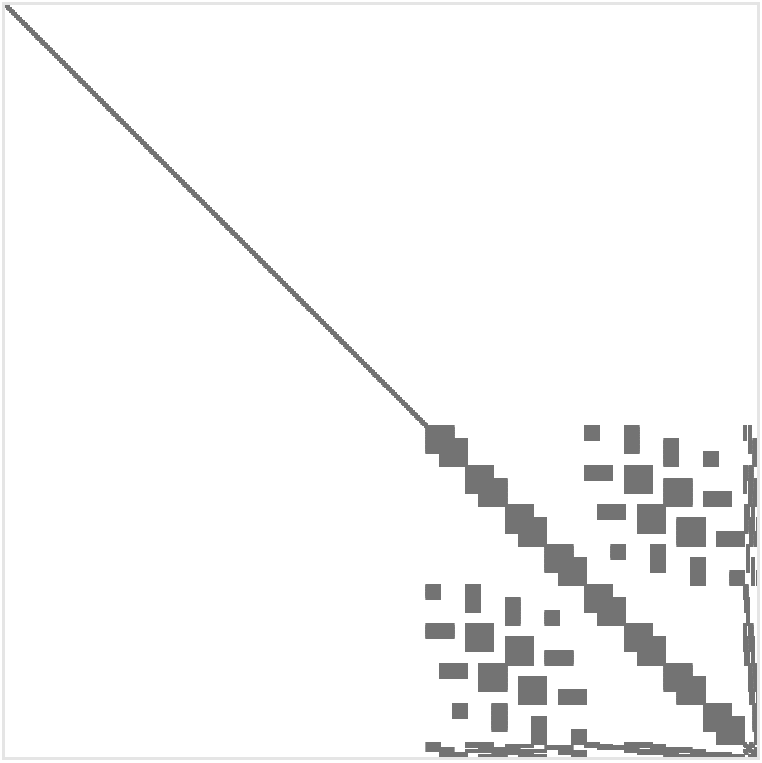}};
\node[inner sep=0pt] at (6.4,0) {
\includegraphics[width=0.22\textwidth]{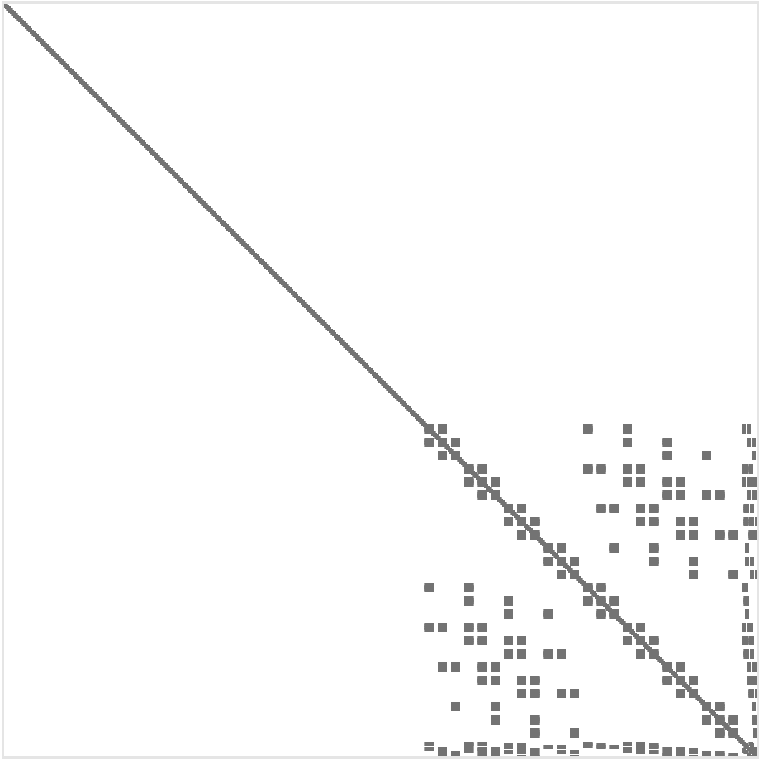}};
\node[inner sep=0pt] at (-1.6,-3.5) {
\includegraphics[width=0.22\textwidth]{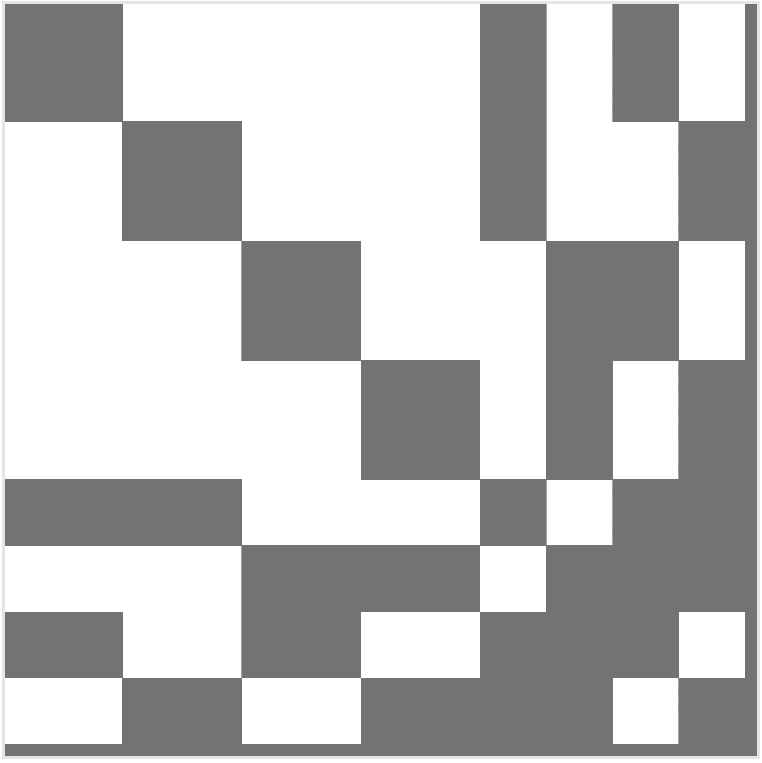}};
\draw[thick,dashed] (7.15,-0.8) circle (0.95cm);
\draw[thick,dashed,-{Latex[length=2mm]}] (6.4,-1.4) -- (-0.2,-2.0);
\node[inner sep=0pt] at (1.6,-3.5) {
\includegraphics[width=0.22\textwidth]{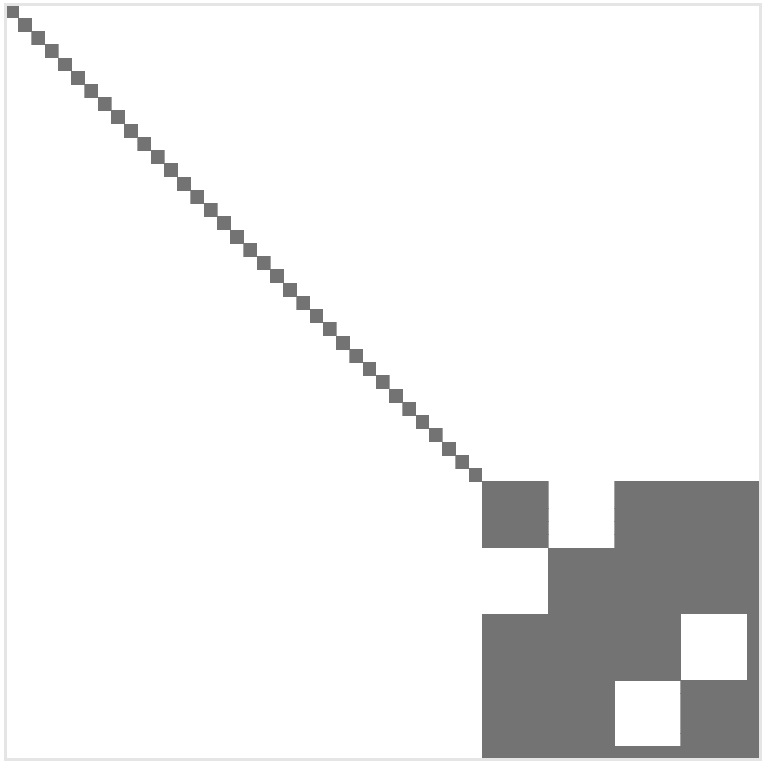}};
\node[inner sep=0pt] at (4.8,-3.5) {
\includegraphics[width=0.22\textwidth]{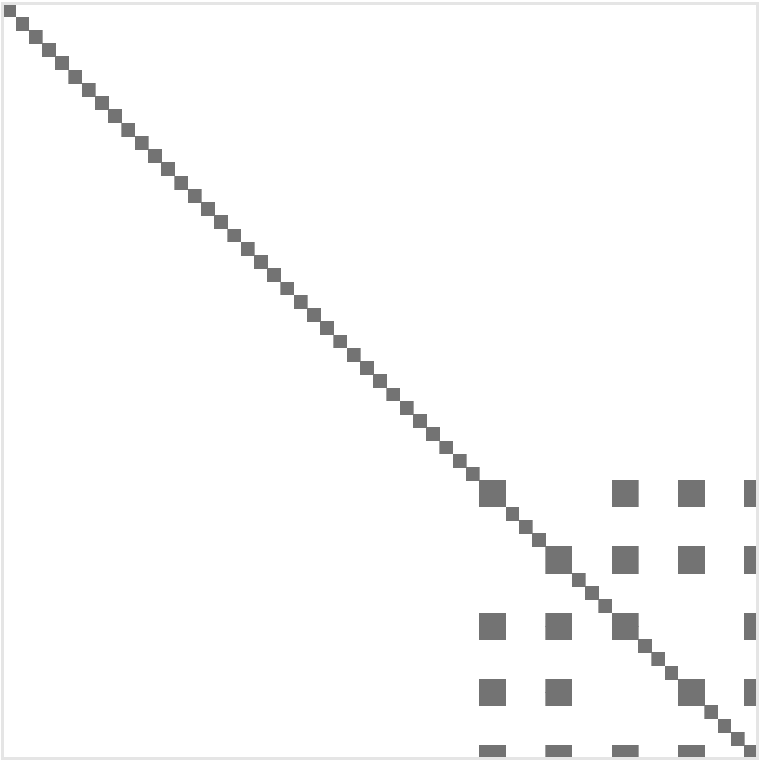}};
\node[inner sep=0pt] at (8.0,-3.5) {
\includegraphics[width=0.22\textwidth]{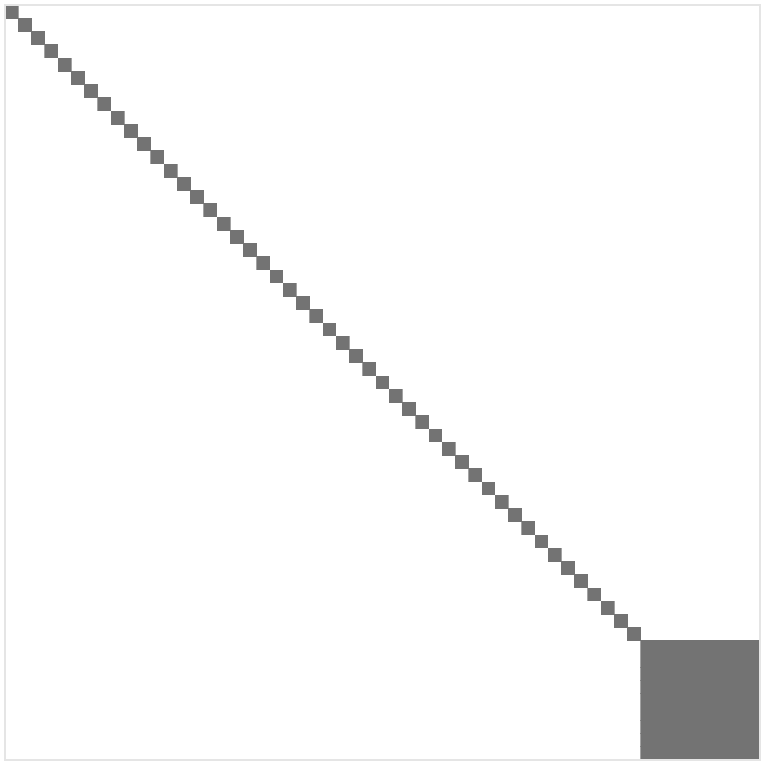}};
\end{tikzpicture}}
\caption{Illustration of the \algfull{} algorithm. In (a)--(g), the  uneliminated unknowns are denoted by black dots. Also visible is the distinction between seaparators (gray) and interiors (white). The separators ensure that interiors do not interact with each other. This is visible in the matrix sparsity patterns shown in (h) where interiors, if any, precede separators in the index ordering.}
\label{fig:illustration}
\end{figure}

The key step of the algorithm is Compression (\emph{Step~\ref{it:comp}}). Without it, the factorization would become an exact block Cholesky decomposition. The role of compression is to repeatedly eliminate additional variables from each selected node, to ensure $\mathcal{O}(n)$ complexity of applying the approximate operator $\Amat_{\ell}^{-1}$. This elimination is approximate, that is, each compression introduces small error; however, the way $\Amat$ acts on piecewise polynomial functions of the grid, is preserved, thus retaining accuracy on smooth eigenvectors. Compression is described in \autoref{sec:compression}.
\subsection{Domain partition and partition graph}\label{sec:defs}
We denote the set of row indices of $\Amat$ (called \emph{unknowns} or \emph{variables}) by $V = \{1,2, \ldots, n \}$.

\begin{definition}\label{def:partition} A \emph{partition} of $V = \{1,2,\ldots,n \}$ is a finite collection $P(V) = \{ B_1, B_2, \ldots, B_t \}$ of disjoint subsets $B_i \subseteq V$ such that $V = B_1 \cup B_2 \cup B_3 \cup \cdots \cup B_t$. \end{definition}

\begin{definition}\label{def:partition-coarsening} We say that a partition $P'(V)$ is a \emph{coarsening} of partition $P(V)$, (or a \emph{coarser partition}) which we denote by $P(V) \prec P'(V)$, if for each set $D \in P'(V)$, there is a subcollection $\{ C_1, C_2, \ldots, C_s \} \subseteq P(\Omega)$ such that $D = C_1 \cup C_2 \cup \cdots \cup C_s$.
\end{definition}

In the sequel, we will distinguish the variables from $B_i$ that have not been eliminated since the eliminated ones no longer play a role in the factorization.

\begin{definition}\label{def:index-set} Let $P_k(V)$ denote the partition in the level $k$ of the algorithm outlined in \autoref{sec:overview}. The \emph{active subset} of $B \in P_k(V)$, denoted by $\Boxvar$, is the subset of variables in $B$ that are not eliminated at the current stage of the algorithm.
\end{definition}

\begin{definition}\label{def:partition-graph} Let $P_k(V) = \{ B_1, B_2, \ldots, B_{t_k} \}$ be the partition in the $k$-th level of the algorithm outlined in \autoref{sec:overview}. The associated set of \emph{nodes} is defined as $N_k(V) = \{ (B_1,\Boxvar_1), (B_2,\Boxvar_2), \ldots, (B_{t_k},\Boxvar_{t_k}) \},$ where $\Boxvar_i$ is the active subset of $B_i$. \end{definition}

We assume that $\Amat_{\mathcal{B}_i \mathcal{B}_j}$ (the interaction matrix between two nodes) can be nonzero only if grid stencils corresponding to some elements of $B_i$ and $B_j$, are close to each other. The number of variables $|\Boxvar|$ will be referred to as the \emph{size} of the node.


\section{Key idea: polynomial compression}\label{sec:compression}
The key difference between our approach and similar methods performing sparse approximate factorizations (\cite{pouransari2017fast,ho2016hierarchical,cambier2019spand,sushnikova2018compress,minden2017recursive,chen2019robust}), is the design of compression step. The idea of compression is to drop interactions of many variables while retaining accuracy on the near-kernel eigenspace. In the case of second-order elliptic differential operators, it is known that the eigenvectors corresponding to smallest eigenvalues are smooth. We will therefore assume that an eigenvector with small associated eigenvalue can be locally approximated well by a discretized low-degree polynomial of the grid (see below). More generally, the subspace approximating the near-kernel eigenspace can be supplied by the user.


\subsection{Basis of discretized polynomials}\label{sec:polynomial-basis}
 We denote by $\Poly^0$ the $n \times 1$ matrix (a vector of length $n$) of ones. If every unknown has a specified location in $\mathbb{R}^3$, e.g., the location of the underlying grid vertex, let $(x_i, y_i, z_i) \in \mathbb{R}^3$ denote the location of the $i$-th unknown, for $i \in V = \{ 1,2,3, \ldots , n \}$. If that be the case, we define:
\begin{displaymath}
\begin{matrix}
	\Poly^1 := 
	\begin{pmatrix}
		1 & x_1 & y_1 & z_1 \\
		1 & x_2 & y_2 & z_2 \\
		\vdots & \vdots & \vdots & \vdots \\
		1 & x_n & y_n & z_n 
	\end{pmatrix}, &
	\Poly^2 := 
	\begin{pmatrix}
		\begin{matrix}
		\Poly^1
		\end{matrix}		
		&
		\begin{matrix}
		x_1^2 & y_1^2 & z_1^2 & x_1 y_1 & y_1 z_1 & z_1 x_1 \\
		x_2^2 & y_1^2 & z_1^2 & x_2 y_2 & y_2 z_2 & z_2 x_2 \\
		\vdots & \vdots & \vdots & \vdots & \vdots & \vdots \\
		x_n^2 & y_n^2 & z_n^2 & x_n y_n & y_n z_n & z_n x_n \\
		\end{matrix}
	\end{pmatrix}
\end{matrix}	
\end{displaymath}
In other words, the columns of the matrix $\Poly^j$ for $j = 0,1,2$, are a basis of the space of discretized real polynomials of degree $j$.
Once the degree has been chosen, we write $\Poly := \Poly^j$. We denote the number of columns of $\Poly$ by $\pi$. 
The definition of $\Matrix{\Poly}$ may need to be modified when more than one variable has the same underlying location, for example in the case of a vector PDE (see \autoref{test:linel}).

\subsection{Compression step}\label{sec:compress-details}
Given an SPD matrix $\Matrix{M} \in \mathbb{R}^{n \times n}$ and a node $(B,\Boxvar) \in N(V)$, let $\Boxvars{N}$ denote the set of variables interacting with $\Boxvar$ and let $\Boxvars{W}$ denote the set of all the remaining variables. We have:
\begin{equation}
	\Matrix{M} =
	\begin{pmatrix}
	\Matrix{M}_{\Boxvar \Boxvar} & \Matrix{M}_{\Boxvar \Boxvars{N}} & \\
	\Matrix{M}_{\Boxvars{N} \Boxvar} & \Matrix{M}_{\Boxvars{N}\Boxvars{N}} & \Matrix{M}_{\Boxvars{N}\Boxvars{W}} \\
	 & \Matrix{M}_{\Boxvars{W} \Boxvars{N}} & \Matrix{M}_{\Boxvars{W} \Boxvars{W}}
	\end{pmatrix}
\end{equation}
Compressing interactions of $(B,\Boxvar)$ will not modify the interactions of $\Boxvars{W}$ and so to keep compact notation, we consider only the upper left $2 \times 2$ block. Also, for simplicity, assume that there are only two nodes having nonzero interactions with $(B,\Boxvar)$, and denote them by $(N_1,\Boxvars{N}_1),$ and $(N_2, \Boxvars{N}_2)$, respectively. 

First, we scale the block row and column corresponding to $(B,\Boxvar)$:
\begin{equation}\label{eq:compr-scale}
\Matrix{M} =
	 \begin{pmatrix}
	\Matrix{L} & \\
					& \Matrix{I} \\
	 \end{pmatrix}
\begin{pmatrix}
\Matrix{I} & \Matrix{\widehat{M}}_{\mathcal{B}  \Boxvars{N}_1} & \Matrix{\widehat{M}}_{\mathcal{B}  \Boxvars{N}_2}  \\
 \Matrix{\widehat{M}}_{\Boxvars{N}_1 \Boxvar} & \Matrix{M}_{\Boxvars{N}_1  \Boxvars{N}_1} & \Matrix{M}_{\Boxvars{N}_1  \Boxvars{N}_2}  \\
  \Matrix{\widehat{M}}_{\Boxvars{N}_2 \Boxvar} & \Matrix{M}_{\Boxvars{N}_2  \Boxvars{N}_1} & \Matrix{M}_{\Boxvars{N}_2  \Boxvars{N}_2} \\
\end{pmatrix}
\begin{pmatrix}
	\Matrix{L}^T & \\
					& \Matrix{I} \\
	 \end{pmatrix}
\end{equation}
where $ \Matrix{M}_{\Boxvar \Boxvar} = \Matrix{L} \Matrix{L}^T$ denotes the Cholesky decomposition, and $\Matrix{\widehat{M}}_{\mathcal{B}  \Boxvars{N}} = \Matrix{L}^{-1}  \Matrix{M}_{\mathcal{B}  \Boxvars{N}}.$
Now, let $\Matrix{\Phi}_B = \Poly(\Boxvar,:),$ and likewise for $(N_1,\Boxvars{N}_1)$ and $(N_2,\Boxvars{N}_2)$. Then
\begin{displaymath}
\Matrix{\Phi} = 
\begin{pmatrix} 
	\Matrix{\Phi}_B & & \\
	& \Matrix{\Phi}_{N_1} & \\
	& & \Matrix{\Phi}_{N_2}
\end{pmatrix}
\end{displaymath}
 spans the space of piecewise (discretized) polynomial functions. Consider the $|\Boxvar| \times 3 \pi$ matrix (which we call the \emph{\nmat}):
\begin{equation}\label{eq:qr-overview}
\Matrix{N} = \begin{pmatrix} \Matrix{L}^T \Matrix{\Phi}_{B} &  \Matrix{\widehat{M}}_{\Boxvar \Boxvars{N}_1} \Matrix{\Phi}_{N_1} & 
\Matrix{\widehat{M}}_{\Boxvar \Boxvars{N}_2} \Matrix{\Phi}_{N_2} 
\end{pmatrix}
\end{equation}
and an orthogonal basis $\Matrix{Q}_1$ for its range, e.g., from the column-pivoted QR decomposition:
\begin{equation}\label{eq:QR}
\Matrix{N}\Matrix{P} = \Matrix{Q} \Matrix{R} = 
\begin{pmatrix} \Matrix{Q}_1 & \Matrix{Q}_2 \end{pmatrix} 
\begin{pmatrix} \Matrix{R}_1 \\ { } \end{pmatrix}
\end{equation}
The matrix $\Matrix{N}$ will typically have a small number of columns, and therefore $\Matrix{Q}_1$, spanning its range, is thin. Notice also that $\Matrix{Q}_2^T\Matrix{\widehat{M}}_{\Boxvar \Boxvars{N}_1}  \Matrix{\Phi}_{N_1}$, $\Matrix{Q}_2^T \Matrix{\widehat{M}}_{\Boxvar \Boxvars{N}_2} \Matrix{\Phi}_{N_2},$ and $\Matrix{Q}_2^T \Matrix{L}^T \Matrix{\Phi}_B$ are zero matrices because $\Matrix{Q}_2$ spans the space orthogonal to the range of $\Matrix{N}$. With
\[
\Matrix{B} = 
\begin{pmatrix}
 \Matrix{LQ } & \\
 & \Matrix{I}  
\end{pmatrix}
\]
we therefore have:
\begin{align*} \label{eq:APhi}
\Matrix{M} \Matrix{\Phi} &
=
\Matrix{B}
\begin{pmatrix}
    \Matrix{I} &\Matrix{Q}^T \Matrix{\widehat{M}}_{\Boxvar \Boxvars{N}_1} & \Matrix{Q}^T  \Matrix{\widehat{M}}_{\Boxvar \Boxvars{N}_2} \\
    \Matrix{\widehat{M}}_{ \Boxvars{N}_1 \Boxvar } \Matrix{Q}  &\Matrix{M}_{\Boxvars{N}_1 \Boxvars{N}_1 } &  \Matrix{M}_{\Boxvars{N}_1 \Boxvars{N}_2} & \\
    \Matrix{\widehat{M}}_{\Boxvars{N}_2 \Boxvar}  \Matrix{Q} & \Matrix{M}_{\Boxvars{N}_2 \Boxvars{N}_1}  &  \Matrix{M}_{\Boxvars{N}_2 \Boxvars{N}_2}
\end{pmatrix}
\Matrix{B}^T
\begin{pmatrix} 
    \Matrix{\Phi}_B & & \\
    & \Matrix{\Phi}_{N_1} & \\
    & & \Matrix{\Phi}_{N_2}
\end{pmatrix}  \\
 = &
 \Matrix{B}
\begin{pmatrix}
    \Matrix{I} & &\Matrix{Q}_1^T \Matrix{\widehat{M}}_{\Boxvar \Boxvars{N}_1} & \Matrix{Q}_1^T  \Matrix{\widehat{M}}_{\Boxvar \Boxvars{N}_2} \\
    & \Matrix{I} & \Matrix{Q}_2^T \Matrix{\widehat{M}}_{\Boxvar \Boxvars{N}_1} & \Matrix{Q}_2^T  \Matrix{\widehat{M}}_{\Boxvar \Boxvars{N}_2} \\
    \Matrix{\widehat{M}}_{\Boxvars{N}_1 \Boxvar} \Matrix{Q}_1 & \Matrix{\widehat{M}}_{\Boxvars{N}_1 \Boxvar} \Matrix{Q}_2 &\Matrix{M}_{\Boxvars{N}_1 \Boxvars{N}_1} &  \Matrix{M}_{\Boxvars{N}_1 \Boxvars{N}_2} \\
    \Matrix{\widehat{M}}_{\Boxvars{N}_2 \Boxvar}  \Matrix{Q}_1 & \Matrix{\widehat{M}}_{\Boxvars{N}_1 \Boxvar} \Matrix{Q}_2 & \Matrix{M}_{\Boxvars{N}_2 \Boxvars{N}_1}  &  \Matrix{M}_{\Boxvars{N}_2 \Boxvars{N}_2 }
\end{pmatrix}
\begin{pmatrix} 
    \Matrix{Q}_1^T\Matrix{L}^T\Matrix{\Phi}_B & & \\
    \Matrix{Q}_2^T\Matrix{L}^T\Matrix{\Phi}_B & & \\
    & \Matrix{\Phi}_{N_1} & \\
    & & \Matrix{\Phi}_{N_2}
\end{pmatrix}
 \notag \\
= &
\Matrix{B}
\begin{pmatrix}
     \Matrix{Q}_1^T \Matrix{L}^T  \Matrix{\Phi}_B & \Matrix{Q}_1^T \Matrix{\widehat{M}}_{\Boxvar \Boxvars{N}_1 } \Matrix{\Phi}_{N_1} & \Matrix{Q}_1^T  \Matrix{\widehat{M}}_{\Boxvar \Boxvars{N}_2} \Matrix{\Phi}_{N_2}\\
      \Matrix{Q}_2^T \Matrix{L}^T \Matrix{\Phi}_B   &  \Matrix{Q}_2^T \Matrix{\widehat{M}}_{\Boxvar \Boxvars{N}_1 } \Matrix{\Phi}_{N_1}  & \Matrix{Q}_2^T  \Matrix{\widehat{M}}_{\Boxvar \Boxvars{N}_2} \Matrix{\Phi}_{N_2} \\
    \Matrix{\widehat{M}}_{\Boxvars{N}_1 \Boxvar} \Matrix{Q}_1 \Matrix{Q}_1^T \Matrix{L}^T \Matrix{\Phi}_B + \Matrix{\widehat{M}}_{\Boxvars{N}_1\Boxvar} \Matrix{Q}_2 \Matrix{Q}_2^T\Matrix{L}^T \Matrix{\Phi}_B  &\Matrix{M}_{\Boxvars{N}_1 \Boxvars{N}_1} \Matrix{\Phi}_{N_1}  &  \Matrix{M}_{\Boxvars{N}_1 \Boxvars{N}_2} \Matrix{\Phi}_{N_2} \\
    \Matrix{\widehat{M}}_{\Boxvars{N}_2 \Boxvar}  \Matrix{Q}_1\Matrix{Q}_1^T\Matrix{L}^T \Matrix{\Phi}_B + \Matrix{\widehat{M}}_{\Boxvars{N}_2 \Boxvar}  \Matrix{Q}_2 \Matrix{Q}_2^T\Matrix{L}^T \Matrix{\Phi}_B  & \Matrix{M}_{\Boxvars{N}_2 \Boxvars{N}_1} \Matrix{\Phi}_{N_1}   &  \Matrix{M}_{\Boxvars{N}_2 \Boxvars{N}_2} \Matrix{\Phi}_{N_2}
\end{pmatrix}
\notag \\
= &
\Matrix{B}
\begin{pmatrix}
    \Matrix{I} & & \Matrix{Q}_1^T \Matrix{\widehat{M}}_{\Boxvar N_1} & \Matrix{Q}_1^T  \Matrix{\widehat{M}}_{\Boxvar \Boxvars{N}_2} \\
        & \Matrix{I} & &  \\
    \Matrix{\widehat{M}}_{\Boxvars{N}_1 \Boxvar} \Matrix{Q}_1 &  &\Matrix{M}_{N_1 \Boxvars{N}_1} &  \Matrix{M}_{\Boxvars{N}_1 \Boxvars{N}_2} & \\
    \Matrix{\widehat{M}}_{\Boxvars{N}_2 \Boxvar}  \Matrix{Q}_1& & \Matrix{M}_{\Boxvars{N}_2 \Boxvars{N}_1}  &  \Matrix{M}_{\Boxvars{N}_2 \Boxvars{N}_2}
\end{pmatrix}
\Matrix{B}^T
\Matrix{\Phi}%
\notag
\end{align*}
Based on the equation above, we can drop the blocks $\Matrix{Q}_2^T \Matrix{\widehat{M}}_{\Boxvar \Boxvars{N}_2}$ and $\Matrix{Q}_2^T \Matrix{\widehat{M}}_{\Boxvar  \Boxvars{N}_2}$ (and their transposes) when approximating $\Matrix{M}$ as they do not affect its action on piecewise polynomial vectors. In other words, we approximate:
\begin{displaymath}
\Matrix{M} \approx
	\Matrix{B}
\begin{pmatrix}
    \Matrix{I} & & \Matrix{Q}_1^T \Matrix{\widehat{M}}_{\Boxvar N_1} & \Matrix{Q}_1^T  \Matrix{\widehat{M}}_{\Boxvar \Boxvars{N}_2} \\
        & \Matrix{I} &  \\
    \Matrix{\widehat{M}}_{\Boxvars{N}_1 \Boxvar} \Matrix{Q}_1 &  &\Matrix{M}_{N_1 \Boxvars{N}_1} &  \Matrix{M}_{\Boxvars{N}_1 \Boxvars{N}_2} \\
    \Matrix{\widehat{M}}_{\Boxvars{N}_2 \Boxvar}  \Matrix{Q}_1& & \Matrix{M}_{\Boxvars{N}_2 \Boxvars{N}_1}  &  \Matrix{M}_{\Boxvars{N}_2 \Boxvars{N}_2}
\end{pmatrix}
\Matrix{B}^T
:= 
\Matrix{B} \Matrix{M}_{(+)} \Matrix{B}^T
\end{displaymath}
A number of variables from $\Boxvar$ get eliminated from the system as they no longer interact with other variables in $\Matrix{M}_{(+)}$, which is strictly sparser than $\Matrix{M}$. The smaller the rank of $\Matrix{Q}_1,$ the larger the number of eliminated variables. We then perform a cheap update:
\begin{equation}\label{eq:phi-update} \Matrix{\Phi}_{B} \leftarrow \Matrix{Q}_1^T \Matrix{L}^T \Matrix{\Phi}_{B} \end{equation} and we can recurse on $\Matrix{M}_{(+)}$ to sparsify interactions of the subsequent node. After interactions of $(N_1,\Boxvars{N}_1)$ and $(N_2, \Boxvars{N}_2)$ have also been compressed, subsets of variables from each of $\Boxvar$, $\Boxvars{N}_1$, and $\Boxvars{N}_2$ are eliminated, and we have $\Matrix{M} \approx \Matrix{U} \Matrix{M}_{(++)}\Matrix{U}^T$ where $\Matrix{M}_{(++)}$ is much sparser than $\Matrix{M},$ and $\Matrix{U}$ is a sparse block diagonal matrix.

Going back to the general case, assuming that there are $g$ nodes interacting with $(B,\Boxvar),$ compressing its interactions would be written as $\Matrix{M} \approx \Matrix{\widetilde{M}},$ where:
\begin{equation}\label{eq:M'}
\Matrix{\widetilde{M}} := 
\Matrix{B}
\begin{pmatrix}
 \Matrix{I} & & \Matrix{Q}_1^T \Matrix{\widehat{M}}_{\mathcal{B}  \Boxvars{N}_1} & \Matrix{Q}_1^T  \Matrix{\widehat{M}}_{\mathcal{B}  \Boxvars{N}_2} & \cdots & \Matrix{Q}_1^T  \Matrix{\widehat{M}}_{\mathcal{B}  \Boxvars{N}_g}  \\
 		& \Matrix{I} &  &  \\
  \Matrix{\widehat{M}}_{\Boxvars{N}_1\Boxvar} \Matrix{Q}_1 &  &\Matrix{M}_{\mathcal{N}_1  \mathcal{N}_1} &  \Matrix{M}_{\mathcal{N}_1  \Boxvars{N}_2} & \cdots &  \Matrix{M}_{\mathcal{N}_1  \mathcal{N}_g} & \Matrix{M}_{\Boxvars{N}_1  \Boxvars{W}}  \\
\Matrix{\widehat{M}}_{\Boxvars{N}_2  \mathcal{B}}  \Matrix{Q}_1&  & \Matrix{M}_{\Boxvars{N}_2  \mathcal{N}_1}  &  \Matrix{M}_{\Boxvars{N}_2 \Boxvars{N}_2} & \cdots & \Matrix{M}_{\Boxvars{N}_2 \mathcal{N}_g} & \Matrix{M}_{\Boxvars{N}_2  \Boxvars{W}} \\
\vdots & & \vdots & \vdots & \ddots & \vdots & \vdots \\
\Matrix{\widehat{M}}_{\Boxvars{N}_g  \Boxvar} \Matrix{Q}_1 &  & \Matrix{M}_{\Boxvars{N}_g  \Boxvars{N}_1}  &  \Matrix{M}_{\Boxvars{N}_g  \Boxvars{N}_2} & \cdots & \Matrix{M}_{\Boxvars{N}_g  \Boxvars{N}_g} & \Matrix{M}_{\Boxvars{N}_g  \Boxvars{W}}\\
 &  & \Matrix{M}_{\Boxvars{W} \Boxvars{N}_1}  &  \Matrix{M}_{\Boxvars{W}  \Boxvars{N}_2} & \cdots & \Matrix{M}_{\Boxvars{W}  \Boxvars{N}_g} & \Matrix{M}_{\Boxvars{W}  \Boxvars{W}}
\end{pmatrix}
\Matrix{B}^T
\end{equation}
The space unaffected by compression is formally described in the following.
\begin{lemma}\label{thm:lemma}
For an SPD matrix $\Matrix{M}$, let $\Matrix{\widetilde{M}}$ be defined as in \autoref{eq:M'}. Then for any $B \in P(V),$ we have:
\begin{equation}\label{eq:lemma}
\Matrix{M} \Matrix{\Pi}_B = \Matrix{\widetilde{M}}  \Matrix{\Pi}_B
\end{equation}
where $\Matrix{\Pi}_B$ is the $|B| \times \pi$ matrix defined by:
\begin{equation}\label{def:PiB}
(\Matrix{\Pi}_B)_{ij}=
  \begin{cases}
   \Matrix{\Pi}_{ij} & \text{if } i \in B \\
   0 & \text{otherwise. }
  \end{cases}
\end{equation}
Moreover, $\Matrix{\widetilde{M}}$ is also SPD. 
\end{lemma}
\begin{proof}
It remains to prove the preservation of the SPD property. But
 $\Matrix{\widetilde{M}}$ is SPD iff $\Matrix{M}_{(+)}$ is SPD, where $\Matrix{\widetilde{M}} = \Matrix{B} \Matrix{M}_{(+)} \Matrix{B}^T$ as above. But $\Matrix{M}_{(+)}$, up to a permutation of variables, is composed of two non-interacting diagonal blocks of which one is the identity matrix and the other one is, again up to a permutation of variables, a principal submatrix of $\Matrix{B}^{-1}\Matrix{M}\Matrix{B}^{-T}$, which is SPD.
\end{proof}


Note a minor technical detail. In \Cref{thm:lemma}, we use $B$ instead of $\Boxvar$ (i.e., the preserved vectors correspond to all unknowns in $B$). Initially, one would have $B = \Boxvar$, i.e., when no variables have been eliminated, one puts $ \Matrix{\Phi}_B = \Poly(\Boxvar,:) = \Poly(B,:)$. However, later on, when variables get eliminated, the globally preserved vectors are still given by \autoref{def:PiB}, as long as the matrices $\Matrix{\Phi}_B$ are properly updated. 

Notice also that the scaling by $\Matrix{L}$ before compression \autoref{eq:compr-scale} plays an important role in the proof of the SPD property of $\Matrix{\widetilde{M}}$. In particular, the factorization will never fail (at least in exact precision). Similar observation has been made in \cite{cambier2019spand,chen2019robust,xia2017effective}.

\subsection{Using low-rank approximation}\label{sec:low-rank}
We would like to note here that our methods do not preclude the use of low-rank approximations in compression. Namely, provided that the matrix $\Matrix{M}_{\Boxvar \Boxvars{N}}$ is numerically low-rank, the matrix $\Matrix{Q}_1$ can be ensured to also satisfy
$\Matrix{\widehat{M}}_{\Boxvar \Boxvars{N}} = \Matrix{Q}_1\Matrix{Y}^T_1 + \Matrix{Q}_2 \Matrix{Y}^T_2$,
 where $\| \Matrix{Y}_2^T \|_2$ is small.
In that case, the error $\| \Matrix{M} - \Matrix{\widetilde{M}} \|_2 $ is controlled more directly. In the absence of geometrical information, for instance, the constant vector, i.e., $\Poly = \Poly^0$ can be used, combined with a rank-revealing decomposition, to ensure that $\| \Matrix{Y}_2^T \|_2$ above is small.

\section{ \algfull{} (\algo)}\label{sec:detailed-description}

Denote by $\Amat_k$ the approximation to $\Amat$ after $k$ levels of the algorithm outlined in \autoref{sec:idea}. Eliminating many nodes before any compressions occur, can guarantee that (for any $k$):
\begin{equation}\label{eq:equivalence}
	\Amat_{k} \Matrix{Y} = \Amat \Matrix{Y}
\end{equation}
where $\Matrix{Y}$ is a matrix with $\Theta(n)$ orthogonal columns. However, we would particularly like \autoref{eq:equivalence}
to also hold for $\Matrix{Y}$ whose columns span the eigenspace of $\Amat$ corresponding to $\Theta(n)$ smallest eigenvalues. Then we could expect $\Amat^{-1}_{k}$ to be a very close approximation of $\Amat^{-1}$ (see \autoref{sec:app-quality}).
However, we do not know the eigenvectors a priori, and we want $\Amat_k^{-1}$ to be as sparse as possible. This will motivate the algorithm, which we now describe in detail. 

We denote by $P_i(V)$ the partition at level $i$ of the algorithm, and by $N_i(V)$ the set of the corresponding nodes. The nodes need to be divided into two subsets: the set of \emph{interiors}, denoted by $I_i$ and the set of \emph{separators}, denoted by $S_i$.
An interior node can only interact with separator nodes; interior nodes will be eliminated by Gaussian elimination. However, the interactions of selected separators will also be compressed (which does not violate their separating properties). With each node $(B,\Boxvar)$, we associate a matrix $\Matrix{\Phi}_B,$ where in the first partition we put $\Matrix{\Phi}_B \leftarrow \Poly_B(B,:) = \Poly(B,:)$. We denote the total number of levels by $\ell$. We also put $\Amat_{(0)} := \Amat$.
 
 \subsection{Eliminating interior nodes}\label{sec:elim-step}
 We eliminate the interior nodes, if any, using the standard block Cholesky algorithm, which at the $i$-th level we denote by:
 \begin{equation}\label{eq:elimination}
 	 \Amat _{(i-1)} = \left( \prod_{j=1}^{|I_i|} \Matrix{G}_{(j)} \right)\Amat_{(i-\frac{1}{2})}   \left( \prod_{j = |I_i|}^{j=1} { \Matrix{G}_{(j)}^T} \right) =  \Matrix{G}_{i} \Amat_{(i-\frac{1}{2})} \Matrix{G}_{i}^T \end{equation}
 
\subsection{Compressing selected separators}\label{sec:compr-step}
To sparsify $\Amat_{(i-\frac{1}{2})}$ while preserving some form of \autoref{eq:equivalence}, we apply the polynomial compression (\autoref{sec:compress-details}) for each selected node, which we denote by:

\begin{equation}\label{eq:compression}
\Amat_{(i-\frac{1}{2})} \approx \left( \prod_{j = 1}^{|S_i|} \Matrix{B}_{(i)} \right)  \Amat_{(i)}  \left( \prod_{j = |S_i|}^{1} \Matrix{B}_{(i)}^T \right) =  \Matrix{B}_{i} \Amat_{(i)} \Matrix{B}_{i}^T  
\end{equation}
We now describe this step and explain its correctness. From \autoref{sec:compress-details} we know that compressing interactions of a single node $(B_0,\Boxvar_0)$ with $g$ interacting nodes, leads to \autoref{eq:equivalence} for $k = i$, with $\Matrix{Y} = \Matrix{\Phi}$, which is our space approximating the near-kernel eigenspace:
\begin{equation}\label{def:Y-diag}
\Matrix{\Phi} = 
\begin{pmatrix}
\Matrix{\Phi}_{B_0} & & & \\
& \Matrix{\Phi}_{B_1} & & \\
& & \ddots & \\
& & & \Matrix{\Phi}_{B_g}
\end{pmatrix}
\end{equation}
The filtered interaction matrix from \autoref{eq:qr-overview} whose QR we need to compute, reads:
\begin{equation}\label{eq:N-definition}
\Matrix{N} =
\begin{pmatrix}
	\Matrix{L}^T \Matrix{\Phi}_{B_0} &
	\Matrix{\widehat{A}}_{\Boxvar_0 \Boxvar_1} \Matrix{\Phi}_{B_1} & \Matrix{\widehat{A}}_{\Boxvar_0 \Boxvar_2} \Matrix{\Phi}_{B_2} & \cdots & \Matrix{\widehat{A}}_{\Boxvar_0 \Boxvar_g} \Matrix{\Phi}_{B_g}
\end{pmatrix}
\end{equation}
This means that compression is practical only if the blocks $\Matrix{\Phi}_B$ have small numbers of columns. Otherwise no or very few interactions can be compressed.
After compressing interactions of $(B_0,\Boxvar_0)$, we need to update (as in \autoref{eq:phi-update}):
\begin{equation}\label{eq:Y-update} \Matrix{\Phi}_{B_0} \leftarrow \Matrix{Q}_1^T \Matrix{L}^T \Matrix{\Phi}_{B_0} \end{equation} 
The update \autoref{eq:Y-update} does not change the block structure of $\Matrix{\Phi}$, and so compressions in the given level can continue in the same manner.

Notice, however, that when eliminating interior nodes, it seems we also need to update $\Matrix{\Phi},$ since then:
\begin{displaymath}
\Amat _{(i-1)} \Matrix{\Phi} = \Matrix{G}_{i} \Amat_{(i - \frac{1}{2})} \Matrix{G}_{i}^T \Matrix{\Phi}
\end{displaymath}
The update $\Matrix{\Phi} \leftarrow \Matrix{G}_{i}^T\Matrix{\Phi}$ would ruin the block-diagonal-like structure of \autoref{def:Y-diag}. Fortunately, we do not need to perform this update at all because it only changes the rows of $\Matrix{\Phi}$ that correspond to variables just eliminated by block Cholesky.

\subsection{Constructing the new partition} After compressions in level $i$ have been completed, we need to define the new partition $P_{i+1}(V) \succ P_{i}(V)$ (see \cref{def:partition-coarsening}). Every $B \in P_{i}(V)$ must have a unique ``father'' $D \in P_{i+1}(V)$ such that $B \subseteq D$. 

\subsection{Forming the new $\Matrix{\Phi}$ matrix}\label{sec:newphi}

With $\Matrix{\Phi}$ defined as in \autoref{def:Y-diag} (with blocks defined by partition in the first level), when coarser partitions are formed, the numbers of columns of nonzero sub-blocks of $\Matrix{\Phi}$ corresponding to the new partition sets will also grow, and eventually the compression will become impractical. On the other hand, we can keep modifying $\Matrix{\Phi}$ so that for each set $D$ in the $k$-th partition, the matrix $\Matrix{Y}_D$ spanning the space of discretized polynomials over $D$ (of a pre-chosen degree equal for each level), satisfies \autoref{eq:equivalence}. Notice that if this is in fact true for level $k$, then it is also true at the beginning of level $k+1,$ because of the nature of polynomials, and the fact that each set $D$ at level $k+1$ is a union of sets from level $k$. 

However, we still need to make sure that the appropriate blocks in \autoref{def:Y-diag} correspond to polynomial bases of the sets at the new level. When a set $D$ is formed by merging the sets $B_1,B_2, \ldots, B_c$ in the algorithm, then its active subset $\Boxvars{D} =  \Boxvar_1 \cup \Boxvar_2 \cup \cdots \cup \Boxvar_c $, where $\Boxvar_j$ is the active subset of $\Boxvars{B}_j$. Again in light of the nature of updates \autoref{eq:Y-update}, this means that we can simply update:
\begin{equation} 
    \Matrix{\Phi}_{D} \leftarrow 
    \begin{pmatrix} \Matrix{\Phi}_{B_1}\\
    \Matrix{\Phi}_{B_2} \\ \vdots \\ \Matrix{\Phi}_{B_c} \end{pmatrix}
\end{equation} 
By definition, the number of columns of matrix $\Matrix{\Phi}_D$, for any set $D$ at any level, is bounded at all times. This is crucial for achieving $\OO{n}$ or $\OO{n \log{n}}$ complexity of factorization, and $\OO{n}$ complexity of applying $\Amat_{\ell}^{-1}$ (see \autoref{sec:solver-family}). We have:
\begin{equation}
\label{eq:PhiD} 
\Matrix{\Phi}_D = \Matrix{\tilde{\Pi}}_D (\Boxvars{D},:)
\end{equation}
where $ \Matrix{\tilde{\Pi}}_D = \Matrix{B}_i^T \Matrix{G}_i^T \cdots \Matrix{B}_2^T \Matrix{G}_2^T \Matrix{B}_1^T \Matrix{G}_1^T \Matrix{\Pi}_D$, for $\Matrix{\Pi}_D$ as in \autoref{def:PiB}. This means that $\Matrix{\Phi}_D$ is exactly the matrix needed in the compression at level $i+1$ of the algorithm, that is, it corresponds to the space of discretized polynomials over $D$.

\subsection{Eliminating the last node} \label{sec:factor-top} If the new partition is all of $V$, i.e., $P_{i}(V) = V$, we have reached the top level of the algorithm, i.e., $i = \ell$. Eliminating $(V,\Boxvars{V}) \in N_{\ell}(V)$ via exact Cholesky decomposition concludes the factorization.

\subsection{Pseudocode of the factorization}
The hierarchical factorization is summarized in the pseudocode below. 
\begin{algorithm}[H]
\caption{\algfull}
\label{alg:hierarchical}
\begin{algorithmic}
\REQUIRE{$P_1(V)$, $I_1$, $S_1$, $\Poly$}
\FORALL{$(B,\Boxvar) \in N_1(V)$} \{Define the polynomial bases of the nodes\}
\STATE{$\Matrix{\Phi_B} \leftarrow \Poly(\Boxvar,:)$} (see \autoref{sec:polynomial-basis})
\ENDFOR
\STATE {$i \leftarrow 1$}
\WHILE{$| P_i(V) | > 1$}
\FORALL {$(B,\Boxvar) \in I_i$} \{Eliminate interior nodes\}
\STATE{Eliminate $(B,\Boxvar)$} (\autoref{sec:elim-step})
\ENDFOR
\FORALL {$(B,\Boxvar) \in S_i$} \{Compress interactions of selected separators\}
\STATE{Compress interactions of $(B,\Boxvar)$} (\autoref{sec:compr-step})
\ENDFOR
\STATE{Construct coarser partition $P_{i+1}(V) \succ P_{i}(V)$} 
\STATE{Construct new $\Matrix{\Phi}$ matrices} (\autoref{sec:newphi})
\STATE{Choose interiors and separators $I_{i+1},S_{i+1}$} 
\STATE{$i \leftarrow i + 1$}
\ENDWHILE
\STATE{$\ell \leftarrow i $}
\STATE{Eliminate $(V,\Boxvars{V})$} (\autoref{sec:factor-top})
\RETURN { 
$\Amat_{\ell} =\Matrix{G}_1 \Matrix{B}_1
\Matrix{G}_2 \Matrix{B}_2 \cdots \Matrix{G}_{\ell - 1} \Matrix{B}_{\ell -1} \Matrix{G}_{\ell}
\Matrix{G}_{\ell}^T \Matrix{B}^T_{\ell -1} \Matrix{G}^T_{\ell -1} \cdots \Matrix{B}^T_2
\Matrix{G}^T_2 \Matrix{B}^T_1\Matrix{G}^T_1$ } 
\end{algorithmic}
\end{algorithm}

\subsection{Applying the preconditioner}
The approximate inverse operator reads:
\begin{equation*}
\begin{aligned}
\Amat^{-1} \approx \Amat_{\ell}^{-1} & = \\
 \Matrix{G}_{1}^{-T} \Matrix{B}_{1}^{-T} \Matrix{G}_{2}^{-T} \Matrix{B}_{3}^{-T}
 \cdots \Matrix{G}_{\ell-1}^{-T} \Matrix{B}_{\ell-1}^{-T}
   \Matrix{G}_{\ell}^{-T} & \Matrix{G}_{\ell}^{-1} \Matrix{B}_{\ell-1}^{-1} 
   \Matrix{G}_{\ell-1}^{-1} \cdots \Matrix{B}_{2}^{-1}  \Matrix{G}_{2}^{-1} \Matrix{B}_{1}^{-1} 
   \Matrix{G}_{1}^{-1}
\end{aligned}
\end{equation*}
From \autoref{eq:elimination} and \autoref{eq:compression} it should be clear that $\Amat_{\ell}^{-1}$ can be applied by a sequence of interspersed block triangular solves and multiplications by block-diagonal matrices.

\section{Family of preconditioners}\label{sec:solver-family} 
In this section, we propose a family of preconditioners composed of four realizations of the \algo{} algorithm.  To illustrate the ideas, we consider a cartesian grid with a relatively simple discretization scheme.
We would like to stress, however, that our methods are not limited to cartesian grids. There is no assumption about the grid structure in \autoref{sec:detailed-description}; the algorithm only requires a partition hierarchy and the $\Poly$ matrix. The appropriate partitions can be defined geometrically given a general grid, or purely using the matrix graph.
In fact, our methods are being incorporated into spaND \cite{cambier2019spand} which uses METIS \cite{karypis1998fast} to obtain appropriate graph partitionings.
To keep focus, however, in this paper we limit ourselves to cartesian grids where partitions can be explicitly defined


We therefore consider the case in which the discretized domain is of the form $ \Omega = (0,a_1) \times (0,a_2) \times (0,a_3)$ and the grid vertices are located at points $(\frac{kh_1}{2},\frac{lh_2}{2},\frac{mh_3}{2})$ where $h_1,h_2,h_3 > 0$ are small and fixed, and $k,l,m$ are natural numbers. We assume that the discretization is such that the matrix entry $\Amat_{ij}$ can be nonzero only if the grid vertex corresponding to the $j$-th unknown is in the $3 \times 3 \times 3$ subgrid of vertices around the vertex corresponding to the $i$-th unknown.

\subsection{Preconditioners using nested dissection partitions}
We consider the partitions $P_1(V) \prec P_2(V) \prec \cdots \prec P_{\ell}(V)$ inspired by nested dissection \cite{george1973nested,schmitz2014fast}. We choose a natural number $b > 1$. The sets in $P_t(V)$ are formed when the domain $\Omega$ is cut by the planes:
 	$\{x = (rd - 1)h_1 \},$ $\{ x = rdh_1 \},$ $\{ y = (sd - 1)h_2 \},$ $\{ y = sdh_2 \},$ $\{ z = (ud - 1)h_3 \},$ $\{ z = udh_3 \}$
for $r,s,u \in \{1,2,3, \ldots \}$, where $d = 2^{t-1} \, b$. Two variables $i,j \in V$ then belong to the same partition set iff their corresponding grid vertices are in the same set formed by the cutting planes. This is illustrated in \autoref{fig:part-nd}. 

\input{partitions-nd.tex}

Notice that there are three types of sets formed by the cutting planes, which we call 0-, 1-, and 2-cells. The 0-cell is a cube and contains a single grid vertex (the 0-cells touch the corners of eight large cubes). The 1-cell is typically a longitudinal parallelpiped containing $d-1$ grid vertices (1-cells touch the edges of four large cubes). The 2-cell is typically a flat parallelpiped containing $(d-1)^2$ grid vertices (the 2-cells touch the faces of two cubes). Finally, a 3-cell is typically a cube containing $(d-1)^3$ grid vertices (these are the large cubes in \autoref{fig:part-nd}). The 0- 1- and 2-cells naturally create a ``buffer'' between the 3-cells and the rest of the domain. In a 2D case (for example when the domain is $(0,1)^2$) there would be no 3-cells. In this case the 0- and 1-cells naturally separate the 2-dimensional cells from the rest of the domain. See \autoref{fig:illustration} for an illustration with active unknowns visible.

The set of 3-cells will form interiors and will be eliminated using the block Cholesky algorithm (\autoref{sec:elim-step}). Compression of the remaining cells then defines the preconditioner (here, we do not distinguish between a cell and a partition node):

\medskip

\begin{enumerate}[align=left]\setcounter{enumi}{-1}
	\item[\bf \schemeIa] All interactions of all cells are compressed
	\item[\bf \schemeI] All interactions of 2-cells are compressed but others are not
    \item[\bf \schemeIb] Only the interactions of 2-cells with non-adjacent cells are compressed
\end{enumerate} 

\medskip

Compressing only selected interactions in~\schemeIb{}  above is achieved by including in $\Matrix{Q}_1$ the span of the off-diagonal blocks connecting the given 2-cell with its adjacent 0- and 1-cells (so that most or all of the compressed interactions are those between 2-cells). 
Most other hierarchical algorithms similar to~\schemeIb{}---which compress only well-separated interactions---cannot guarantee the SPD property, and may break (for example LoRaSp \cite{pouransari2017fast}). In~\schemeIb{}, the compression ranks (numbers of columns of $\Matrix{Q}_1$) increase with each level (in contrast, for example, to \cite{pouransari2017fast,sushnikova2018compress} where the ranks can be bounded) but the SPD property is preserved.

\schemeI{} is analogous to the schemes used in the context of Hierarchical Interpolative Factorization and Sparsified Nested Dissection \cite{cambier2019spand,ho2016hierarchical} (the latter defines the partitions based on the locations of unknowns, or purely based on the matrix entries). These approaches base the compression step, however, on low-rank approximations, not polynomial compression (\cite{ho2016hierarchical} being also quite different algebraically).

\schemeIa{}, on the other hand, is quite different in that it also compresses interactions of 1-cells (compressing interactions of 0-cells does not make any difference and can be skipped). The off-diagonal block corresponding to a 1-cell in the matrix should not be expected to be low-rank. However, polynomial compression does not require the low-rank property. Intuitively speaking, it replaces each interaction between two nodes by a coarser one, preserving the action of the matrix on a chosen subspace (as opposed to exploiting the low-rank structure to represent the interaction accurately).

\subsection{Preconditioners using general partitions}\label{sec:simple-partitioning-family} The second type of domain partitions considered does not distinguish interior nodes that could be eliminated without introducing large fill-ins. Again, such partitions can be obtained by analyzing just the matrix entries \cite{pouransari2017fast}, using tools such as SCOTCH \cite{pellegrini1996scotch}. As an example, with the same setup as before, $P_t(V)$ could be a partition formed when the domain $\Omega$ is cut into cells containing the same, or roughly the same number of grid vertices:
\begin{displaymath} 
    D_{rsu} =  \left( (rdh_1, (r+1)dh_1) \times (sdh_2, (s+1)dh_2) 
    \times (udh_3, (u+1)dh_3) \right)
\end{displaymath}
This is shown in \autoref{fig:part-general}.
\begin{figure}[htbp]
\centering
\subfloat{
\begin{tikzpicture}
    \def \b{0.0}
    \def \a{0.6}
    \def \c{0.15}
    \def \x{-4*\c}
    \def \y{4*\c}
    \pgfmathsetmacro{\cubex}{\a}
    \pgfmathsetmacro{\cubey}{\a}
    \pgfmathsetmacro{\cubez}{\a}
    \draw[thick] (\x,-\b,0) coordinate (a) -- ++(-\cubex,0,0) coordinate (b);
    \draw[thick](b) -- ++(0,-\cubey,0) coordinate (c);
    \draw[thick](c) -- ++(\cubex,0,0) coordinate (d);
    \draw[thick](d) -- (a);
    \draw[thick](a) -- ++(0,0,-\cubez) coordinate (e);
    \draw[thick](b) -- ++(0,0,-\cubez) coordinate (f);
    \draw[draw=none] (c) -- ++(0,0,-\cubez) coordinate (g);
    \draw[draw=none] (d) -- ++(0,0,-\cubez) coordinate (h);
    \draw[thick](e) -- (f);
    \draw[DashedDraw] (c) -- (g) -- (f);
    \draw[DashedDraw] (g) -- (h);
    %
    \def \x{2*\b}
    \def \y{4*\b}
    \pgfmathsetmacro{\cubex}{\a}
    \pgfmathsetmacro{\cubey}{\a}
    \pgfmathsetmacro{\cubez}{\a}
    \draw[thick](\x,-\b,0) coordinate (a) -- ++(-\cubex,0,0) coordinate (b);
    \draw[thick](b) -- ++(0,-\cubey,0) coordinate (c);
    \draw[thick](c) -- ++(\cubex,0,0) coordinate (d);
    \draw[thick](d) -- (a);
    \draw[thick](a) -- ++(0,0,-\cubez) coordinate (e);
    \draw[draw=none] (b) -- ++(0,0,-\cubez) coordinate (f);
    \draw[draw=none] (c) -- ++(0,0,-\cubez) coordinate (g);
    \draw[draw=none] (d) -- ++(0,0,-\cubez) coordinate (h);
    \draw[thick](e) -- (f);
    \draw[DashedDraw] (c) -- (g) -- (f);
    \draw[DashedDraw] (g) -- (h);     
    \def \x{2*\b+\a}
    \def \y{4*\c + \a}
    \pgfmathsetmacro{\cubex}{\a}
    \pgfmathsetmacro{\cubey}{\a}
    \pgfmathsetmacro{\cubez}{\a}
    \draw[thick](\x,-\b,0) coordinate (a) -- ++(-\cubex,0,0) coordinate (b);
    \draw[draw=none] (b) -- ++(0,-\cubey,0) coordinate (c);
    \draw[thick](c) -- ++(\cubex,0,0) coordinate (d);
    \draw[thick](d) -- (a);
    \draw[thick](a) -- ++(0,0,-\cubez) coordinate (e);
    \draw[draw=none] (b) -- ++(0,0,-\cubez) coordinate (f);
    \draw[draw=none] (c) -- ++(0,0,-\cubez) coordinate (g);
    \draw[draw=none] (d) -- ++(0,0,-\cubez) coordinate (h);
    \draw[thick](e) -- (f);
    \draw[DashedDraw] (c) -- (g) -- (f);
    \draw[DashedDraw] (g) -- (h);    
    \def \x{2*\b + 2*\a}
    \def \y{4*\c + \a}
    \pgfmathsetmacro{\cubex}{\a}
    \pgfmathsetmacro{\cubey}{\a}
    \pgfmathsetmacro{\cubez}{\a}
    \draw[thick](\x,-\b,0) coordinate (a) -- ++(-\cubex,0,0) coordinate (b);
    \draw[draw=none] (b) -- ++(0,-\cubey,0) coordinate (c);
    \draw[thick](c) -- ++(\cubex,0,0) coordinate (d);
    \draw[thick](d) -- (a);
    \draw[thick](a) -- ++(0,0,-\cubez) coordinate (e);
    \draw[draw=none] (b) -- ++(0,0,-\cubez) coordinate (f);
    \draw[draw=none] (c) -- ++(0,0,-\cubez) coordinate (g);
    \draw[thick](d) -- ++(0,0,-\cubez) coordinate (h);
    \draw[thick](e) -- (f);
    \draw[thick](h) -- (e);
    \draw[DashedDraw] (c) -- (g) -- (f);
    \draw[DashedDraw] (g) -- (h);    
\foreach \j in {1,2,3} {
	\foreach \k in {0,1} {
	    \def \x{2*\b}
	    \def \y{-4*\c-\b}
	    \def \dx{4*\c + \b}
	    \pgfmathsetmacro{\cubex}{\a}
	    \pgfmathsetmacro{\cubey}{\a}
	    \pgfmathsetmacro{\cubez}{\a}
	    \draw[thick](\x-\k*\dx,\j*\y-\b,0) coordinate (a) -- ++(-\cubex,0,0) coordinate (b);
	    \draw[thick](b) -- ++(0,-\cubey,0) coordinate (c);
	    \draw[thick](c) -- ++(\cubex,0,0) coordinate (d);
	    \draw[thick](d) -- (a);
	    \draw[draw=none] (a) -- ++(0,0,-\cubez) coordinate (e);
	    \draw[draw=none] (b) -- ++(0,0,-\cubez) coordinate (f);
	    \draw[draw=none] (c) -- ++(0,0,-\cubez) coordinate (g);
	    \draw[draw=none] (d) -- ++(0,0,-\cubez) coordinate (h);
	    \draw[draw=none] (e) -- (f) -- (g) -- (h) -- cycle;
	    \draw[DashedDraw] (c) -- (g) -- (f);
            \draw[DashedDraw] (g) -- (h);    	    	    	    
	\foreach \i in {1,2,3} {
	    \def \x{2*\b}
	    \def \y{-4*\c-\b}
	    \def \z{-4*\c-\b}
	    \pgfmathsetmacro{\cubex}{\a}
	    \pgfmathsetmacro{\cubey}{\a}
	    \pgfmathsetmacro{\cubez}{\a}
	    \draw[draw=none] (\x-\k*\dx,\j*\y-\b,\i*\z) coordinate (a) -- ++(-\cubex,0,0) coordinate (b);
	    \draw[draw=none] (b) -- ++(0,-\cubey,0) coordinate (c);
	    \draw[draw=none](c) -- ++(\cubex,0,0) coordinate (d);
	    \draw[draw=none] (a) -- ++(0,0,-\cubez) coordinate (e);
	    \draw[draw=none] (b) -- ++(0,0,-\cubez) coordinate (f);
	    \draw[draw=none] (c) -- ++(0,0,-\cubez) coordinate (g);
	    \draw[draw=none] (d) -- ++(0,0,-\cubez) coordinate (h);
	    \draw[DashedDraw] (c) -- (g) -- (f);
            \draw[DashedDraw] (g) -- (h);     	    	       	      	      	        	        	        	    	
	    \def \x{2*\b}
	    \def \y{8*\c}
	    \def \z{-4*\c-\b}
            \pgfmathsetmacro{\cubex}{\a}
	    \pgfmathsetmacro{\cubey}{\a}
	    \pgfmathsetmacro{\cubez}{\a}
	    \draw[thick](\x-\k*\dx,-\b,\i*\z) coordinate (a) -- ++(-\cubex,0,0) coordinate (b);
	    \draw[draw=none] (b) -- ++(0,-\cubey,0) coordinate (c);
	    \draw[draw=none] (c) -- ++(\cubex,0,0) coordinate (d);
	    \draw[thick](a) -- ++(0,0,-\cubez) coordinate (e);
	    \draw[thick](b) -- ++(0,0,-\cubez) coordinate (f);
	    \draw[draw=none] (c) -- ++(0,0,-\cubez) coordinate (g);
	    \draw[draw=none] (d) -- ++(0,0,-\cubez) coordinate (h);
	    \draw[thick](e) -- (f);
	    \draw[DashedDraw] (c) -- (g) -- (f);
            \draw[DashedDraw] (g) -- (h);         
	    \def \x{2*\b+\a}
	    \def \y{8*\c}
	    \def \z{-4*\c-\b}
	    \pgfmathsetmacro{\cubex}{\a}
	    \pgfmathsetmacro{\cubey}{\a}
	    \pgfmathsetmacro{\cubez}{\a}
	    \draw[thick](\x,-\b,\i*\z) coordinate (a) -- ++(-\cubex,0,0) coordinate (b);
	    \draw[draw=none] (b) -- ++(0,-\cubey,0) coordinate (c);
	    \draw[draw=none] (c) -- ++(\cubex,0,0) coordinate (d);
	    \draw[draw=none] (d) -- (a);
	    \draw[draw=none] (a) -- ++(0,0,-\cubez) coordinate (e);
	    \draw[draw=none] (b) -- ++(0,0,-\cubez) coordinate (f);
	    \draw[draw=none] (c) -- ++(0,0,-\cubez) coordinate (g);
	    \draw[draw=none] (d) -- ++(0,0,-\cubez) coordinate (h);
	    \draw[thick](e) -- (f);
	    \draw[DashedDraw] (c) -- (g) -- (f);
            \draw[DashedDraw] (g) -- (h); 	
	    } 
	    }
	    }     
\foreach \j in {1,2,3} {
	    \def \x{2*\b+\a}
	    \def \y{-4*\c-\b}
	    \pgfmathsetmacro{\cubex}{\a}
	    \pgfmathsetmacro{\cubey}{\a}
	    \pgfmathsetmacro{\cubez}{\a}
	    \draw[thick](\x,\j*\y-\b,0) coordinate (a) -- ++(-\cubex,0,0) coordinate (b);
	    \draw[thick](b) -- ++(0,-\cubey,0) coordinate (c);
	    \draw[thick](c) -- ++(\cubex,0,0) coordinate (d);
	    \draw[thick](d) -- (a);
	    \draw[draw=none] (a) -- ++(0,0,-\cubez) coordinate (e);
	    \draw[draw=none] (b) -- ++(0,0,-\cubez) coordinate (f);
	    \draw[draw=none] (c) -- ++(0,0,-\cubez) coordinate (g);
	    \draw[draw=none] (d) -- ++(0,0,-\cubez) coordinate (h);
	    \draw[DashedDraw] (c) -- (g) -- (f);
            \draw[DashedDraw] (g) -- (h); 
	 \foreach \i in {1,2,3} {
	    \def \x{2*\b+\a}
	    \def \y{-4*\c-\b}
	    \def \z{-4*\c-\b}
	    \pgfmathsetmacro{\cubex}{\a}
	    \pgfmathsetmacro{\cubey}{\a}
	    \pgfmathsetmacro{\cubez}{\a}
	    \draw[draw=none] (\x,\j*\y-\b,\i*\z) coordinate (a) -- ++(-\cubex,0,0) coordinate (b);
	    \draw[draw=none] (b) -- ++(0,-\cubey,0) coordinate (c);
	    \draw[draw=none]  (c) -- ++(\cubex,0,0) coordinate (d);
	    \draw[draw=none]  (a) -- ++(0,0,-\cubez) coordinate (e);
	    \draw[draw=none] (b) -- ++(0,0,-\cubez) coordinate (f);
	    \draw[draw=none]  (c) -- ++(0,0,-\cubez) coordinate (g);
	    \draw[draw=none]  (d) -- ++(0,0,-\cubez) coordinate (h);
	    \draw[DashedDraw] (c) -- (g) -- (f);
            \draw[DashedDraw] (g) -- (h);  
	    \def \x{2*\b+2*\a}
	    \def \y{-4*\c-\b}
	    \def \z{-4*\c-\b}
	    \pgfmathsetmacro{\cubex}{\a}
	    \pgfmathsetmacro{\cubey}{\a}
	    \pgfmathsetmacro{\cubez}{\a}
	    \draw[draw=none] (\x,\j*\y-\b,\i*\z) coordinate (a) -- ++(-\cubex,0,0) coordinate (b);
	    \draw[draw=none]  (b) -- ++(0,-\cubey,0) coordinate (c);
	    \draw[draw=none]  (c) -- ++(\cubex,0,0) coordinate (d);
	    \draw[thick](d) -- (a);
	    \draw[thick](a) -- ++(0,0,-\cubez) coordinate (e);
	    \draw[draw=none]  (b) -- ++(0,0,-\cubez) coordinate (f);
	    \draw[draw=none]  (c) -- ++(0,0,-\cubez) coordinate (g);
	    \draw[thick](d) -- ++(0,0,-\cubez) coordinate (h);
	    \draw[thick](e) -- (h);
	    \draw[DashedDraw] (c) -- (g) -- (f);
            \draw[DashedDraw] (g) -- (h);	
	    \def \x{2*\b+2*\a}
	    \def \y{8*\c}
	    \def \z{-4*\c-\b}
	    \pgfmathsetmacro{\cubex}{\a}
	    \pgfmathsetmacro{\cubey}{\a}
	    \pgfmathsetmacro{\cubez}{\a}
	    \draw[thick](\x,-\b,\i*\z) coordinate (a) -- ++(-\cubex,0,0) coordinate (b);
	    \draw[draw=none]  (b) -- ++(0,-\cubey,0) coordinate (c);
	    \draw[draw=none]  (c) -- ++(\cubex,0,0) coordinate (d);
	    \draw[draw=none] (d) -- (a);
	    \draw[thick](a) -- ++(0,0,-\cubez) coordinate (e);
	    \draw[thick]  (b) -- ++(0,0,-\cubez) coordinate (f);
	    \draw[draw=none]  (c) -- ++(0,0,-\cubez) coordinate (g);
	    \draw[thick](d) -- ++(0,0,-\cubez) coordinate (h);
	    \draw[draw=none]  (e) -- (f) -- (g) -- (h) -- cycle;
	    \draw[thick](e) -- (h);
	    \draw[thick](e) -- (f);
	    \draw[DashedDraw] (c) -- (g) -- (f);
            \draw[DashedDraw] (g) -- (h);	    
	    }
	   \def \x{2*\b+2*\a}
	   \def \y{-4*\c-\b}
	    \pgfmathsetmacro{\cubex}{\a}
	    \pgfmathsetmacro{\cubey}{\a}
	    \pgfmathsetmacro{\cubez}{\a}
	    \draw[thick](\x,\j*\y-\b,0) coordinate (a) -- ++(-\cubex,0,0) coordinate (b);
	    \draw[thick](b) -- ++(0,-\cubey,0) coordinate (c);
	    \draw[thick](c) -- ++(\cubex,0,0) coordinate (d);
	    \draw[thick](d) -- (a);
	    \draw[thick](a) -- ++(0,0,-\cubez) coordinate (e);
	    \draw[draw=none]  (b) -- ++(0,0,-\cubez) coordinate (f);
	    \draw[draw=none]  (c) -- ++(0,0,-\cubez) coordinate (g);
	    \draw[thick](d) -- ++(0,0,-\cubez) coordinate (h);
	    \draw[thick](e) -- (h);
	    \draw[draw=none]  (e)--(f)--(g)--(h)-- cycle;
	    \draw[DashedDraw] (c) -- (g) -- (f);
            \draw[DashedDraw] (g) -- (h);
   	    }
\end{tikzpicture}}
\subfloat{
\begin{tikzpicture}
    \def \b{0.13333}
    \def \c{0.13333}
    \def \a{9*\c}
    \def \x{12*\c}
    \pgfmathsetmacro{\cubex}{\a}
    \pgfmathsetmacro{\cubey}{\a}
    \pgfmathsetmacro{\cubez}{\a}
    \draw[thick](\x,-\c,0) coordinate (a) -- ++(-\cubex,0,0) coordinate (b);
    \draw[thick] (b) -- ++(0,-\cubey,0) coordinate (c);
    \draw[thick](c) -- ++(\cubex,0,0) coordinate (d);
    \draw[thick](d) -- (a);
    \draw[thick](a) -- ++(0,0,-\cubez) coordinate (e);
    \draw[thick] (b) -- ++(0,0,-\cubez) coordinate (f);
    \draw[draw=none] (c) -- ++(0,0,-\cubez) coordinate (g);
    \draw[draw=none] (d) -- ++(0,0,-\cubez) coordinate (h);
    \draw[thick](e) -- (f);
    \draw[DashedDraw] (c) -- (g) -- (f);
    \draw[DashedDraw] (g) -- (h);     
    \def \x{21*\c}
    \pgfmathsetmacro{\cubex}{\a}
    \pgfmathsetmacro{\cubey}{\a}
    \pgfmathsetmacro{\cubez}{\a}
    \draw[thick](\x,-\c,0) coordinate (a) -- ++(-\cubex,0,0) coordinate (b);
    \draw[draw=none] (b) -- ++(0,-\cubey,0) coordinate (c);
    \draw[thick](c) -- ++(\cubex,0,0) coordinate (d);
    \draw[thick](d) -- (a);
    \draw[thick](a) -- ++(0,0,-\cubez) coordinate (e);
    \draw[draw=none] (b) -- ++(0,0,-\cubez) coordinate (f);
    \draw[draw=none] (c) -- ++(0,0,-\cubez) coordinate (g);
    \draw[thick](d) -- ++(0,0,-\cubez) coordinate (h);
    \draw[thick](e) -- (f);
    \draw[thick](h) -- (e);
    \draw[DashedDraw] (c) -- (g) -- (f);
    \draw[DashedDraw] (g) -- (h);    
\foreach \j in {1,2,3} {
	\foreach \k in {0,1} {
	    \def \dx{10*\c}
	\foreach \i in {1} {
	    \def \x{10*\c + 2*\b}
	    \def \z{-9*\c}
	    \pgfmathsetmacro{\cubex}{\a}
	    \pgfmathsetmacro{\cubey}{\a}
	    \pgfmathsetmacro{\cubez}{\a}
	    \draw[thick](\x,-\b,\i*\z) coordinate (a) -- ++(-\cubex,0,0) coordinate (b);
	    \draw[draw=none] (b) -- ++(0,-\cubey,0) coordinate (c);
	    \draw[draw=none] (c) -- ++(\cubex,0,0) coordinate (d);
	    \draw[draw=none] (d) -- (a);
	    \draw[draw=none] (a) -- ++(0,0,-\cubez) coordinate (e);
	    \draw[thick] (b) -- ++(0,0,-\cubez) coordinate (f);
	    \draw[draw=none] (c) -- ++(0,0,-\cubez) coordinate (g);
	    \draw[draw=none] (d) -- ++(0,0,-\cubez) coordinate (h);
	    \draw[thick](e) -- (f);
	    \draw[DashedDraw] (c) -- (g) -- (f);
            \draw[DashedDraw] (g) -- (h);       	    	
	    } 
	    }
	    }     
\foreach \j in {1} {
	    \def \x{11*\c+\b}
	    \def \y{-10*\c}
	    \pgfmathsetmacro{\cubex}{\a}
	    \pgfmathsetmacro{\cubey}{\a}
	    \pgfmathsetmacro{\cubez}{\a}
	    \draw[thick](\x,\j*\y,0) coordinate (a) -- ++(-\cubex,0,0) coordinate (b);
	    \draw[thick](b) -- ++(0,-\cubey,0) coordinate (c);
	    \draw[thick](c) -- ++(\cubex,0,0) coordinate (d);
	    \draw[thick](d) -- (a);
	    \draw[draw=none] (a) -- ++(0,0,-\cubez) coordinate (e);
	    \draw[draw=none] (b) -- ++(0,0,-\cubez) coordinate (f);
	    \draw[draw=none] (c) -- ++(0,0,-\cubez) coordinate (g);
	    \draw[draw=none] (d) -- ++(0,0,-\cubez) coordinate (h);
	    \draw[DashedDraw] (c) -- (g) -- (f);
            \draw[DashedDraw] (g) -- (h);   
	 \foreach \i in {1} {
	    \def \x{10*\c+ 2*\b}
	    \def \y{-9*\c}
	    \def \z{-9*\c}
	    \pgfmathsetmacro{\cubex}{\a}
	    \pgfmathsetmacro{\cubey}{\a}
	    \pgfmathsetmacro{\cubez}{\a}
	    \draw[draw=none] (\x,\j*\y-\b,\i*\z) coordinate (a) -- ++(-\cubex,0,0) coordinate (b);
	    \draw[draw=none] (b) -- ++(0,-\cubey,0) coordinate (c);
	    \draw[draw=none]  (c) -- ++(\cubex,0,0) coordinate (d);
	    \draw[draw=none]  (a) -- ++(0,0,-\cubez) coordinate (e);
	    \draw[draw=none] (b) -- ++(0,0,-\cubez) coordinate (f);
	    \draw[draw=none]  (c) -- ++(0,0,-\cubez) coordinate (g);
	    \draw[draw=none]  (d) -- ++(0,0,-\cubez) coordinate (h);
	    \draw[DashedDraw] (c) -- (g) -- (f);
            \draw[DashedDraw] (g) -- (h);
	    \def \x{21*\c}
	    \def \y{-9*\c}
	    \def \z{-9*\c}
	    \pgfmathsetmacro{\cubex}{\a}
	    \pgfmathsetmacro{\cubey}{\a}
	    \pgfmathsetmacro{\cubez}{\a}
	    \draw[draw=none] (\x,\j*\y-\b,\i*\z) coordinate (a) -- ++(-\cubex,0,0) coordinate (b);
	    \draw[draw=none]  (b) -- ++(0,-\cubey,0) coordinate (c);
	    \draw[draw=none]  (c) -- ++(\cubex,0,0) coordinate (d);
	    \draw[thick](d) -- (a);
	    \draw[thick](a) -- ++(0,0,-\cubez) coordinate (e);
	    \draw[draw=none]  (b) -- ++(0,0,-\cubez) coordinate (f);
	    \draw[draw=none]  (c) -- ++(0,0,-\cubez) coordinate (g);
	    \draw[thick](d) -- ++(0,0,-\cubez) coordinate (h);
	    \draw[thick](e) -- (h);
	    \draw[DashedDraw] (c) -- (g) -- (f);
            \draw[DashedDraw] (g) -- (h);	
	    \def \x{21*\c}
	    \def \y{8*\c}
   	    \def \z{-9*\c}
	    \pgfmathsetmacro{\cubex}{\a}
	    \pgfmathsetmacro{\cubey}{\a}
	    \pgfmathsetmacro{\cubez}{\a}
	    \draw[thick](\x,-\b,\i*\z) coordinate (a) -- ++(-\cubex,0,0) coordinate (b);
	    \draw[draw=none]  (b) -- ++(0,-\cubey,0) coordinate (c);
	    \draw[draw=none]  (c) -- ++(\cubex,0,0) coordinate (d);
	    \draw[draw=none] (d) -- (a);
	    \draw[thick](a) -- ++(0,0,-\cubez) coordinate (e);
	    \draw[thick]  (b) -- ++(0,0,-\cubez) coordinate (f);
	    \draw[draw=none]  (c) -- ++(0,0,-\cubez) coordinate (g);
	    \draw[thick](d) -- ++(0,0,-\cubez) coordinate (h);
	    \draw[draw=none]  (e) -- (f) -- (g) -- (h) -- cycle;
	    \draw[thick](e) -- (h);
	    \draw[thick](e) -- (f);
	    \draw[DashedDraw] (c) -- (g) -- (f);
            \draw[DashedDraw] (g) -- (h);	    
	    }
	    \def \x{21*\c}
	    \def \y{-9*\b}
	    \pgfmathsetmacro{\cubex}{\a}
	    \pgfmathsetmacro{\cubey}{\a}
	    \pgfmathsetmacro{\cubez}{\a}
	    \draw[thick](\x,\j*\y-\b,0) coordinate (a) -- ++(-\cubex,0,0) coordinate (b);
	    \draw[thick](b) -- ++(0,-\cubey,0) coordinate (c);
	    \draw[thick](c) -- ++(\cubex,0,0) coordinate (d);
	    \draw[thick](d) -- (a);
	    \draw[thick](a) -- ++(0,0,-\cubez) coordinate (e);
	    \draw[draw=none]  (b) -- ++(0,0,-\cubez) coordinate (f);
	    \draw[draw=none]  (c) -- ++(0,0,-\cubez) coordinate (g);
	    \draw[thick](d) -- ++(0,0,-\cubez) coordinate (h);
	    \draw[thick](e) -- (h);
	    \draw[draw=none]  (e)--(f)--(g)--(h)-- cycle;
	    \draw[DashedDraw] (c) -- (g) -- (f);
            \draw[DashedDraw] (g) -- (h);
   	    }
\end{tikzpicture}}
\subfloat{
	\begin{tikzpicture}
	    \def \b{0.12}
   	    \def \a{20*\b}
    	\def \x{7*\b}
  	    \def \y{8*\b}
    \pgfmathsetmacro{\cubex}{\a}
    \pgfmathsetmacro{\cubey}{\a}
    \pgfmathsetmacro{\cubez}{\a}
    \draw[thick](\x,-\b,0) coordinate (a) -- ++(-\cubex,0,0) coordinate (b);
    \draw[thick] (b) -- ++(0,-\cubey,0) coordinate (c);
    \draw[thick] (c) -- ++(\cubex,0,0) coordinate (d);
    \draw[thick] (d) -- (a);
    \draw[thick]( a) -- ++(0,0,-\cubez) coordinate (e);
    \draw[thick] (b) -- ++(0,0,-\cubez) coordinate (f);
    \draw[draw=none] (c) -- ++(0,0,-\cubez) coordinate (g);
    \draw[thick] (d) -- ++(0,0,-\cubez) coordinate (h);
    \draw[thick] (e) -- (f);
    \draw[thick] (e) -- (h);
    \draw[DashedDraw] (c) -- (g) -- (f);
    \draw[DashedDraw] (g) -- (h);
    \end{tikzpicture}}
    \caption{Cutting planes defining the general partitions $P_1(V) \prec P_2(V) \prec P_3(V).$}
    \label{fig:part-general}
\end{figure}
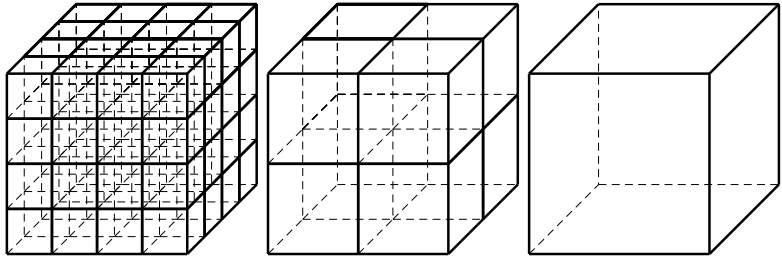

There are no interior nodes as before. Instead, at every level, the size of each node is reduced in compression, and then nodes are merged when forming the new partition for the subsequent level. This means that compression is the only step that includes Cholesky factorization (the scaling \autoref{eq:compr-scale}).
We obtain:

\medskip

\begin{enumerate}[align=left]
\item[\bf \schemeII] All interactions of all cells are compressed
\end{enumerate} 

\medskip

One should not expect approaches that use low-rank approximations in compression to perform well with \schemeII{} unless the matrix satisfies the strong admissibility criterion \cite{borm2002data}. The approaches that do use simple partitioning (\cite{pouransari2017fast,sushnikova2018compress,minden2017recursive,chen2019robust}) typically distinguish between neighboring and well-separated interactions, the latter assumed to be low-rank. Only the well-separated interactions are compressed. The methods of this paper can also be used in that context. The advantage of compressing all interactions of a node, however, is that we can obtain a much sparser operator, and we ensure that $\Amat_{\ell}$ is SPD.

\subsection{Accuracy comparison}
 In terms of accuracy of factorization, one can conceptually write:
\begin{displaymath}
\text{\schemeII} \prec \text{\schemeIa} \prec \text{\schemeI}\prec \text{\schemeIb}
\end{displaymath}
meaning that \schemeII{} is the least accurate (no interiors eliminated, all interactions compressed), and \schemeIb{} is the most accurate (interiors eliminated; only interactions of 2-cells with non-adjacent cells compressed).

\subsection{Complexities of factorization and applying the approximate inverse} As we describe next, \schemeIa{} and \schemeII{} can be expected to have $\OO{n}$ complexity. This is an advantage over \schemeI{} or \schemeIb{} used in \cite{cambier2019spand,ho2016hierarchical} which generally have $\OO{n \log{n}}$ complexity. The partitions used in our methods satisfy the assumptions of the proposition below, with $\delta=8$. The assumptions common to both considered cases should be satisfied by any reasonably balanced partitions. Similar (including more general) complexity analyses are described in \cite{cambier2019spand,pouransari2017fast,yang2016sparse}.

\begin{proposition}\label{thm:proposition-complexity} Assume that the polynomial compression is used with a fixed pre-chosen polynomial degree, and that the node sizes in the first level are bounded. Also assume that any given node in the algorithm has nonzero interactions with a bounded number of other nodes, and that a set $B \in P_k(V)$, is a union of a bounded number of sets from $B \in P_{k-1}(V)$. For $\delta > 1$, assume that there exists a constant $C_0 >0$ such that the partition $P_k(V)$ is composed of at most $C_0 n/\delta^{k-1}$ sets. Then:
\begin{enumerate}
	\item \label{comp-all} For a scheme compressing all interactions (such as \schemeIa{} and \schemeII{}) the complexities and memory requirements of computing the factorization, as well as applying $\Amat_{\ell}^{-1}$, are all $\OO{n}$.
	\item \label{comp-some} For any other scheme, if there exists a constant $C_1 > 0$ such that the size of any node at level $k$ is at most $C_1 \delta^{(k-1)/3}$, then the complexity of computing the factorization is $\OO{n \log{n}}$. Memory requirements as well as the complexity of applying $\Amat_{\ell}^{-1}$ are $\OO{n}$. This holds for \schemeI{} and \schemeIb{} with $\delta = 8$.
\end{enumerate}
\end{proposition}
\begin{proof}
Consider first case~\ref{comp-all}. Since the polynomial compression with fixed polynomial degree is used, and every node has a nonzero interaction with $\OO{1}$ other nodes, the ranks of the \nmats{} \autoref{eq:N-definition} are likewise $\OO{1}$, and therefore the sizes of nodes are also bounded. 
The $\OO{n}$ factorization complexity follows from the fact that $\sum_{k=0}^{\infty}{\frac{n}{\delta^k}} = \OO{n}$, and similarly for memory requirements.

Consider now case~\ref{comp-some}. The cost associated with eliminating or compressing a node this time is $\OO{\delta^k}$. Hence we obtain the complexity bound from $\sum_{k=0}^{\lceil \log_\delta C_0n \rceil}{\frac{n}{\delta^k}\delta^k} =$ $ \OO{n \log{n}}$. Since we now need $\OO{\delta^{2k/3}}$ memory to store a node and its interactions, we obtain the memory requirement by noticing that $\sum_{k=0}^{\lceil \log_\delta C_0n \rceil}{\frac{n}{\delta^k}\delta^{2k/3}} = \sum_{k=0}^{\lceil \log_\delta C_0n \rceil}{n\delta^{-k/3}} = \OO{n}$. The same argument applies to show that applying $\Amat_{\ell}^{-1}$ is $\OO{n}$. To see that case~\ref{comp-some} above holds for \schemeI{} and \schemeIb{}, notice that the number of grid vertices inside a 1-cell is $\OO{2^k}.$ \end{proof}

\section{Approximation error and preconditioner quality}\label{sec:app-quality}

A successful preconditioner for SPD systems must be accurate on the eigenvectors corresponding to the near-kernel (smallest) eigenvalues. The results in \cite{yang2016sparse} suggest that for the (unit-length) eigenvector $\Matrix{v}_i$ corresponding to the given eigenvalue $\lambda_i > 0$, the contribution of the error $\| (\Amat - \Amat_{\ell}) \Matrix{v}_i \|_2,$ to the condition number of the preconditioned system, is amplified by $\lambda_i^{-1}$. Results in \cite{bebendorf2013constraints} show that, to achieve a condition number independent of the problem size, the accuracy of $\Amat_{\ell}$ can be relatively crude provided that \autoref{eq:equivalence} holds for $\Matrix{Y}$ with columns approximating well the near-kernel eigenspace.

 We now describe how preserving piecewise polynomial vectors in our algorithm leads to a better preconditioner. 
 Let $\Amat_i$ again denote the approximation to $\Amat$ obtained after completing the $i$-th level. For $i \in \{1,2, \ldots, \ell \}$, define the error term:
$ \Matrix{E}_i := \Amat_{i} - \Amat_{i-1},$ where by definition $\Amat_{0} := \Amat$. We have:
\[ \Amat_{\ell} - \Amat  = \Matrix{E}_1 + \Matrix{E}_2 + \cdots + \Matrix{E}_{\ell} \]
The error terms $\Matrix{E}_i$ are the results of compression (gaussian elimination does not introduce errors in exact precision). The effect of applying polynomial compression is described in the following.
\begin{proposition}\label{prop:EPi}
Let $k \in \{1,2, \ldots, \ell \}$ and $B \in P_{k}(V)$, where $P_{k}(V)$ is the partition at level $k$ of the algorithm. Also, let $\Poly$ denote the global polynomial basis matrix described in \autoref{sec:polynomial-basis}. Then, for any $1 \leq i \leq k,$
\[
 \Matrix{E}_i \Matrix{\Pi}_B = 0
\]
where $\Matrix{\Pi}_B$ is the $|B| \times \pi$ matrix defined by:
\begin{equation}
(\Matrix{\Pi}_B)_{ij}=
  \begin{cases}
   \Matrix{\Pi}_{ij} & \text{if } i \in B \\
   0 & \text{otherwise.}
  \end{cases}
\end{equation}
In particular, the action of $\Amat$ on $\Poly$ is preserved:
\begin{equation}\label{eq:preserve-pi}
	\Amat_{\ell} \Poly = \Amat \Poly
\end{equation}
\end{proposition}
\begin{proof}
The proposition is true for $k = 1$ which follows from \cref{thm:lemma}. Now suppose that the proposition is true for a given $1 \leq k \leq \ell-1$. Then for $B \in P_{k+1}(V)$ and $1 \leq i \leq k$ we have $\Matrix{E}_i \Matrix{\Pi}_B = 0$. This is a consequence of the fact that $B$ is a union of sets from $P_{k}(V)$ (see \autoref{def:PiB}). To conclude the proof, we need to show that $\Matrix{E}_{k+1} \Matrix{\Pi}_B = 0$ but this is a consequence of \autoref{eq:PhiD} and the remarks that follow it.
\end{proof}
Denoting by $\Matrix{u}_k$ the piecewise polynomial approximation to $\Matrix{u} \in \mathbb{R}^n$ at the $k$-th level of the algorithm, we can expect that $\| \Matrix{u} - \Matrix{u}_k \|_2$ will be small if $\Matrix{u}$ is a smooth function of the grid. The approximation error on $\Matrix{u}$ at the $k$-th level will also be then small because $\Matrix{E}_k \Matrix{u} = \Matrix{E}_k (\Matrix{u} - \Matrix{u}_k)$. The following corollary, which is a direct consequence of \cref{prop:EPi}, makes this statement precise.
\begin{corollary}\label{col:approx}
For $\Matrix{u} \in \mathbb{R}^n$ and $k \in \{1, 2, \ldots, \ell \}$, let $\Matrix{u}_{k}$ denote the
projection of $\Matrix{u}$ onto $
\Upsilon_k := \text{span} \{ \Matrix{\Pi}_{B} \}_{B \in P_{k}(\Omega)}$. We then have:
 \[ \left( \Amat_{\ell} - \Amat \right) \Matrix{u} = \Matrix{E}_1 \left(\Matrix{u} -
 \Matrix{u}_{1} \right) + \Matrix{E}_2 \left(\Matrix{u} - \Matrix{u}_{2} \right) + \cdots +
 \Matrix{E}_{\ell} \left(\Matrix{u} - \Matrix{u}_{\ell} \right) \]
 Put differently, if $\Matrix{P}_k$ denotes the orthogonal projection matrix onto $\Upsilon_k^{\perp}$, then:
  \[ \Amat_{\ell} - \Amat = \Matrix{P}_1 \Matrix{E}_1 \Matrix{P}_1 + \Matrix{P}_2
  \Matrix{E}_2 \Matrix{P}_2 + \cdots + \Matrix{P}_{\ell} \Matrix{E}_{\ell} \Matrix{P}_{\ell} \]
\end{corollary}
Based on \Cref{col:approx}, we can write:
\begin{align}\label{eq:precondition}
\nonumber
\Amat_{\ell} = \Amat + \sum_{k=1}^{\ell}{\Matrix{E}_k} & = \Amat + \sum_{k=1}^{\ell}{ \Matrix{P}_k \Matrix{E}_k  \Matrix{P}_k} \\
\Amat^{-\frac{1}{2}} \Amat_{\ell}\Amat^{-\frac{1}{2}} & = \Matrix{I} + \sum_{k=1}^{\ell}{\Amat^{-\frac{1}{2}}  \Matrix{P}_k \Matrix{E}_k \Matrix{P}_k \Amat^{-\frac{1}{2}}}
\end{align}
where for the eigenvalue decomposition $\Amat = \Matrix{V} \Matrix{\Lambda} \Matrix{V}^T$ we write $\Amat^{\frac{1}{2}} = \Matrix{V} \Matrix{\Lambda^{\frac{1}{2}}} \Matrix{V}^T$. We would like the matrix \autoref{eq:precondition} to be as close to identity as possible. In our case, we also know that it is SPD (see \cref{thm:lemma}). Notice that:
\begin{equation}\label{eq:prec-error-term}
	\| \Amat^{-\frac{1}{2}} \Matrix{P}_k \Matrix{E}_k \Matrix{P}_k \Amat^{-\frac{1}{2}} \|_2 \leq
	\| \Matrix{E}_k \|_2 \| \Matrix{P}_k \Amat^{-1} \Matrix{P}_k \|_2
\end{equation}
 The norm $\| \Matrix{E}_k \|_2$  can be made small using a low-rank approximation as described in \autoref{sec:low-rank}. However, if $\Amat$ is ill-conditioned, then $\Amat^{-1}$ will have large eigenvalues, and the norm on the left hand side of \autoref{eq:prec-error-term} may become large even if $\| \Matrix{E}_k \|_2$ is relatively small. Ensuring that the range of $\Matrix{P}_k$ approximates well the range of the eigenvectors of $\Amat^{-1}$ with small associated eigenvalues, we can therefore directly target the critical eigenspace of ill-conditioned systems, and make small the norm $\| \Matrix{P}_k \Amat^{-1} \Matrix{P}_k \|_2$. In particular (see also \cite{bebendorf2016spectral}, Lemma 2.1), if:
\begin{equation}\label{eq:error-condition}
\sum_{k=1}^{\ell}{	\| \Matrix{E}_k \|_2 \|\Matrix{P}_k \Amat^{-1} \Matrix{P}_k \|_2} \leq \mu < 1
\end{equation}
then from Weyl's inequality for eigenvalues we obtain for the condition number:
\begin{equation*}
	\kappa(\Amat_{\ell}^{-\frac{1}{2}} \Amat \Amat_{\ell}^{-\frac{1}{2}}) \leq \frac{1+\mu}{1 - \mu}
\end{equation*}
This suggests that to obtain a bounded condition number of the preconditioned system, one needs the spaces $\Upsilon_k = \text{span} \{ \Matrix{\Pi}_{B} \}_{B \in P_{k}(\Omega)}$ to approximate well the eigenspace corresponding to the smalles eigenvalues of $\Amat$.

\section{Numerical results}\label{sec:experiments} We apply our methods to problems of importance in engineering. Our goals in this section are:

\begin{enumerate}
	\item To test the efficiency of our methods when applied to large systems. Ideally, we would like to observe the optimal $\OO{n}$ scalings of solution times.
	\item To compare our approach to the one in which compression is based on the standard low-rank approximation of the off-diagonal blocks.
	\item To benchmark the different preconditioners described in \autoref{sec:solver-family}, with varying orders of polynomial compression (piecewise constant, piecewise linear, and piecewise quadratic; that is, when $j$ in the definition of $\Poly^j$ from \autoref{sec:polynomial-basis}, is $j=0,1,2$, respectively).
\end{enumerate}

When testing \schemeI, \schemeIa, and \schemeIb{}, we use $b = 3$, i.e., the 3-cell at level $k$ is typically a cube encompassing $(3 \cdot 2^{k-1} - 1)^3$ grid vertices (see \autoref{sec:solver-family}); we also skip the compression in the first two levels, while the nodes are still small. 

When testing \schemeII{} we choose $b=3$, $b=4$, and $b=5$ when using respectively, piecewise constant, piecewise linear, and piecewise quadratic compression. In other words, the partition sets at level $k$ typically contain variables corresponding to, respectively, $(3 \cdot 2^{k-1})^3$, $(4 \cdot 2^{k-1})^3$, and $(5 \cdot 2^{k-1})^3$ grid vertices. We use compression starting from the first level.

In each test, when the number of not-yet-eliminated unknowns during factorization is small enough (smaller than the size of the largest node encountered so far), we factorize the remaining system exactly using block Cholesky decomposition (i.e., form the last partition). This does not change the maximal node size.

We solve the equation $\Amat \Matrix{x} = \Matrix{b}$ using Conjugate Gradient (CG) with operator $\Amat_{\ell}^{-1}$ used as a preconditioner. Each time we choose a random right hand side $\Matrix{b}$ to ensure that every eigenvector of $\Amat$ contributes to $\Matrix{b}$.

Throughout this section, we use the following notation:
\begin{itemize}[leftmargin=*]
	\item $n$ is the number of unknowns in the analyzed system;
	\item $it_C$ is the number of iterations of Conjugate Gradient needed to converge to a relative residual of 2-norm below $10^{-10}$;
	\item $t_F$ is the factorization time, i.e., the CPU time taken to compute $\Amat_{\ell}^{-1}$, in seconds;
	\item $t_S$ is the solution time, i.e., the CPU time taken by CG to converge, in seconds;
	\item $m_R$ is the maximum memory usage during the computation, in GB.
\end{itemize}

Our implementation is sequential and was written in Python 3.6.1., exploiting NumPy 1.14.3 and SciPy 1.1.0 for numerical computations \cite{oliphant2006guide,Scipy2001}. The tests were run on CPUs with Intel(R) Xeon(R) E5-2640v4 (2.4GHz) with up to 1024 GB RAM. All pivoted QR decompositions were performed by calling LAPACK's \verb|geqp3| \cite{anderson1999lapack}.

\subsection{Descriptions of test cases}
\subsubsection{3D Poisson equation}\label{test:poisson}
The first test case is the classical 3D (constant-coefficient) Poisson equation in a cube:
\begin{equation}\label{eq:poisson}
\begin{matrix} \Delta u(x) = f & \forall{x \in \Omega \in \left[ 0,1\right]^3}, & u|_{\partial \Omega} = 0
\end{matrix}	
\end{equation}
with Dirichlet boundary conditions, discretized using the standard 7-point stencil method. 
The largest system has approximately $16 \cdot 10^6$ unknowns.

\subsubsection{Incompressible flow in the SPE10 Reservoir}\label{test:spe}
The second test case is the 3D flow equation (Darcy's law) of an incompressible single-phase fluid, in an incompressible porous medium:

\begin{equation}\label{eq:darcy}
\begin{matrix}
\nabla \cdot \left( \lambda \cdot \nabla u(x) \right) = 0, & \forall{x \in \Omega}
\end{matrix} 
\end{equation}
discretized using the finite volume method (the 2-point flux approximation), with mixed Dirichlet and Neumann boundary conditions.

\begin{figure}[htb]
	\centering
	\begin{minipage}{.5\textwidth}
	\centering
	\includegraphics[width=0.8\textwidth]{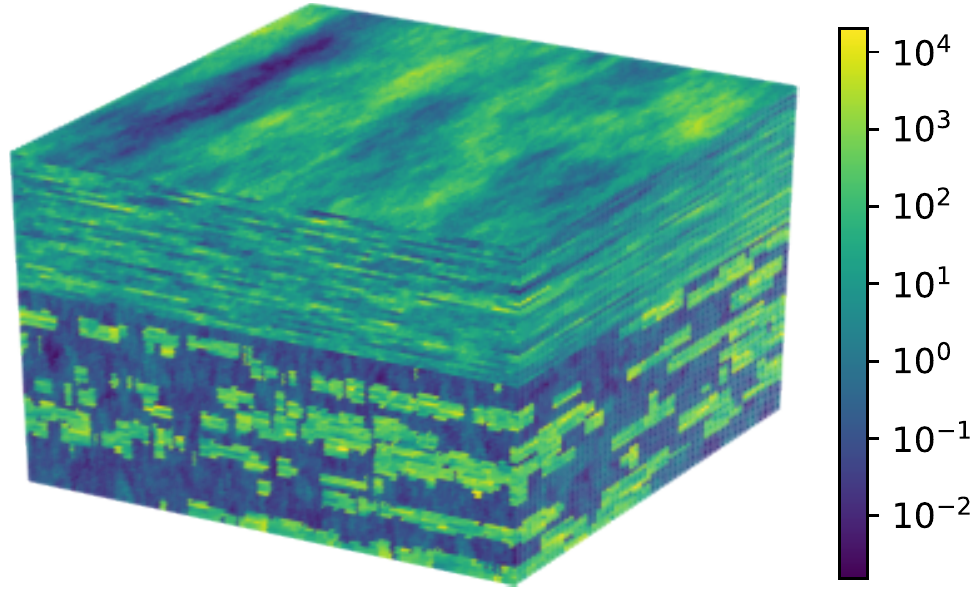}
	\caption{Mobility field in the \\ SPE10 benchmark reservoir.}
	\label{fig:spe10}
	\end{minipage}
	\begin{minipage}{.4\textwidth}
	\centering
	\includegraphics[width=0.75\textwidth]{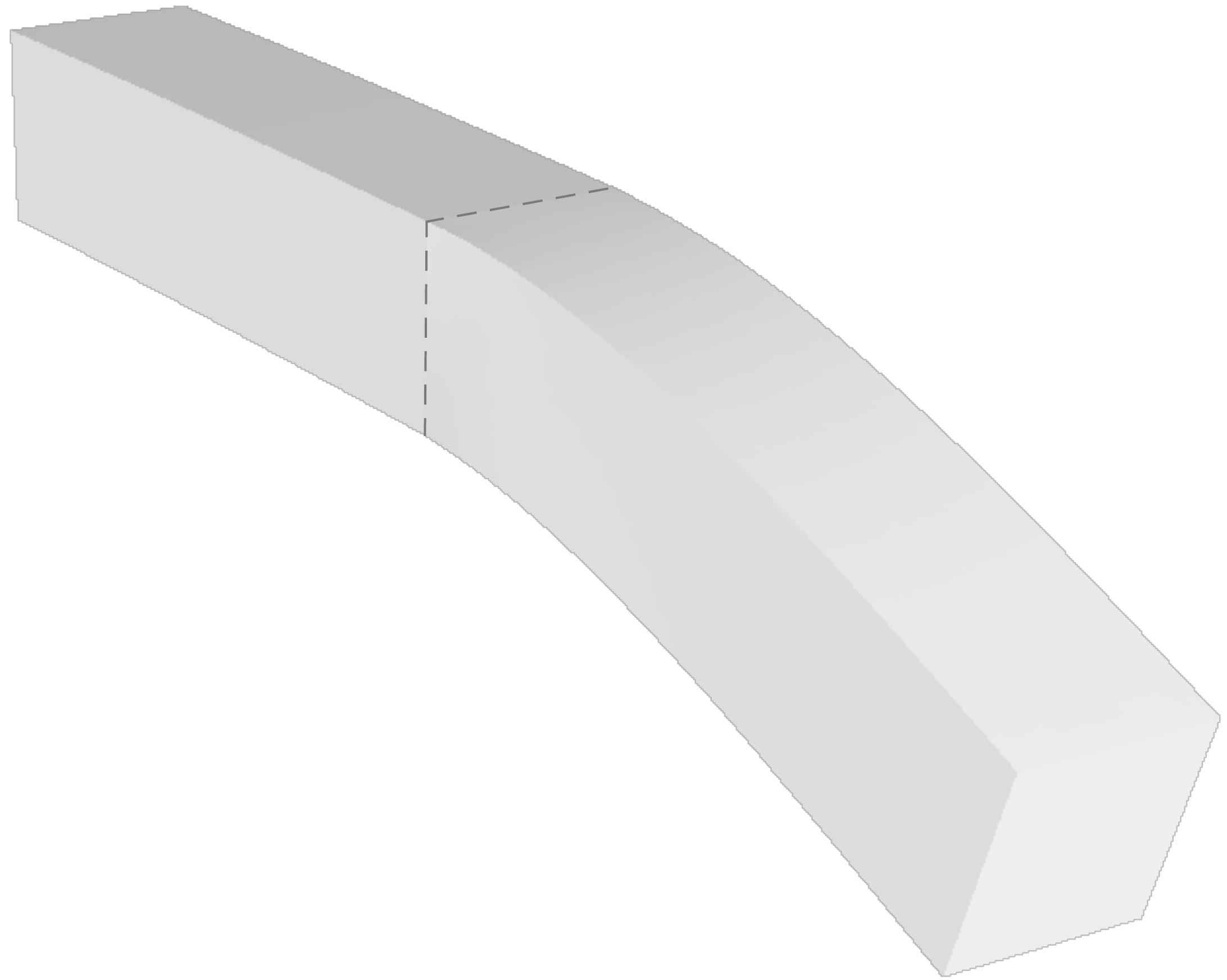}
	\caption{The cantilever beam used in the linear elasticity test.}
	\label{fig:beam}
	\end{minipage}
\end{figure}

To define $\Omega$ and the field of coefficients $\lambda$ (the mobility field), we use the SPE10 Reservoir \cite{christie2001tenth} which is an important benchmark problem
in the petroleum engineering community. The field $\lambda$ varies smoothly in some layers of the reservoir, and is highly discontinuous in other layers (see \autoref{fig:spe10}). 
The result is a Poisson-like equation whose corresponding discretized system is very ill-conditioned.

The smallest test case has approximately $0.4 \cdot 10^6$ variables and is obtained by considering only the upper layers of the SPE10 Reservoir, where $\lambda$ changes relatively smoothly. The case with approximately $1.1 \cdot 10^6$ variables is obtained from the full original SPE10 Reservoir (with grid dimensions $220 \times 80 \times 65$). Larger cases are obtained by periodically tiling the original reservoir in each direction to obtain cubes of the desired size (based on idea from \cite{manea2016parallel}). The matrices are obtained by fixing the pressure value on one outer side of the reservoir, and imposing a constant flow on the opposite side, with no-flow conditions on other sides, but the resulting system $\Amat \Matrix{x} = \Matrix{b}$ in each case is tested with a random right hand side, as described above.

\subsubsection{Linear elasticity}\label{test:linel}
The third test case is a finite-element approximation to the weak form of the linear elasticity equation:
\begin{equation}\label{eq:lin_el} 
\begin{matrix}
- \nabla \cdot \sigma(u(x)) = 0, &  \forall{x \in \Omega}
\end{matrix}
\end{equation}	
where $u: \Omega \rightarrow \mathbb{R}^3$ is the displacement field and $\sigma$ is the stress tensor satisfying:
\begin{equation*} 
\sigma(u(x)) = \lambda (\nabla \cdot u) I + \mu (\nabla u + \nabla u^T)
\end{equation*}
with $\lambda$ and $\mu$ denoting the material Lam\'{e} constants. 

 Our test problem is a cantilever beam composed of two segments, as in \autoref{fig:beam}. The constants corresponding to the left hand side, and the right hand side segments are dimensionless $(\lambda_\text{LHS},\mu_\text{LHS}) = (1.0,1.0)$ and $(\lambda_{\text{RHS}}, \mu_{\text{RHS}} = 50.0,50.0),$ respectively. The boundary conditions are a fixed zero displacement on the left side of the boundary, i.e., $u = 0$ there, and a vertical constant pull down force applied on the other side. The test case (including the discretization using a regular grid with tetrahedral elements) is obtained from the MFEM library \cite{mfem-library} where the reader may find all details. We test systems of increasing sizes by uniformly refining the grid in each dimension; the largest case consists of approximately $6.5 \cdot 10^6$ variables.
 

This test differs from the previous ones in that 
each grid vertex $v$ is now associated with a displacement vector $\Matrix{u}_v = (v_x,v_y,v_z) $ $\in \mathbb{R}^3$ (which is a subvector of the global solution vector $\Matrix{u}$).  As mentioned in \autoref{sec:polynomial-basis}, the definition of the polynomial basis has to be modified. 
Assuming that the variable indices are ordered so that the indices corresponding to all the $v_x$ above (of all grid vertices) come first, then the indices corresponding to all the $v_y$ and then all the $v_z$ (in each case retaining the same order of the underlying grid vertices), we naturally define:
\begin{equation}\label{def:poly-linel}
    \Poly := 
    \begin{pmatrix}
        \Poly^{j} & & \\
                    & \Poly^{j} & \\
                    & & \Poly^{j}
    \end{pmatrix}
\end{equation} 
where $j$ is the chosen polynomial degree, and $\Poly^j$ is defined as in \autoref{sec:polynomial-basis}. Conceptually, the displacement component in each direction is then independent of the displacement components in other directions.


\subsection{Preserving the rigid body modes}
The elasticity test \autoref{eq:lin_el} is a loosely constrained body. In such a case, the so called rigid body modes are known to be in the near-kernel subspace (see for instance \cite{lutz1998elimination}). With $\Poly$ defined as above (\autoref{def:poly-linel}), $\Amat_{\ell}$ preserves the action of $\Amat$ on the rigid body modes (more precisely, they are preserved exactly when using at least piecewise linear compression, which follows from \autoref{eq:preserve-pi} and \autoref{def:poly-linel}). When developing preconditioners for the linear elasticity equation, typically a special care has to be taken to reproduce the action of $\Amat$ on the rigid body modes. In our case, the preservation of these modes is in fact a by-product of the design of our methods.

\subsection{Factorization times and memory usage}
 The theoretical scalings predicted by \Cref{thm:proposition-complexity} are confirmed by our tests. In \autoref{fig:factorization_scaling_poisson} we show examples of hierarchical factorization timings. In \autoref{fig:memory_lin_elast} we show example memory usage.  All plots are well within $\OO{n}$ bounding lines.
 
\begin{figure}[htb]
    \centering
    \begin{tikzpicture}
        \begin{groupplot}[
            group style={
                group name=poisson_1,
                group size=2 by 1,
                xlabels at=edge bottom,
                xticklabels at=edge bottom,
                vertical sep=0.2cm,
                horizontal sep=0.4cm,
                ylabels at=edge left,
                yticklabels at=edge left,
            },
            xmode=log,ymode=log,
            ymin=10,ymax=4000,
            xmin=400000,xmax=17000000,
            xtick={1000000,10000000},
            xticklabels={$10^6$,$10^7$},
        ]
%
            
        \nextgroupplot[width=5.5cm,height=4.75cm,ylabel style={align=center},ylabel={Factorization time \\$t_F$ (sec.)},
        	legend entries=
        	{Pcw. const., Pcw. lin., Pcw. quad., $\OO{n}$, $\OO{n^{3/2}}$},
        	legend columns = 1,xlabel={Problem size $n$},
        	legend style={at={(2.85,1.0)}}]
            \addplot[PcwConst]
            table[x=n, y=tf]{fig/1b/poisson_pcwconst_1b.dat};
            \addplot[PcwLin]
            table[x=n, y=tf]{fig/1b/poisson_pcwlin_1b.dat};
            \addplot[PcwQuad]
            table[x=n, y=tf]{fig/1b/poisson_pcwquad_1b.dat};
            \addplot [domain=500000:16000000,dashed]{3e-4 * x};
            \addplot [domain=500000:16000000,Dashed2]{1e-7 * x^(1.5)};
            \addplot [domain=500000:16000000,dashed]{3e-5 * x};
            \node at (axis cs:1100000,2500) {\schemeIb};            

       \nextgroupplot[width=5.5cm,height=4.75cm,
       		xlabel={Problem size $n$}]
            \addplot[PcwConst]
            table[x=n, y=tf]{fig/2/poisson_pcwconst_2.dat};
            \addplot[PcwLin]
            table[x=n, y=tf]{fig/2/poisson_pcwlin_2.dat};
            \addplot[PcwQuad]
            table[x=n, y=tf]{fig/2/poisson_pcwquad_2.dat};
            \addplot [domain=500000:16000000,dashed]{3e-4 * x};
            \addplot [domain=500000:16000000,Dashed2]{1e-7 * x^(1.5)};
            \addplot [domain=500000:16000000,dashed]{3e-5 * x};        
            \node at (axis cs:1100000,2500) {\schemeII};
            
        \end{groupplot}
    \end{tikzpicture}
    \caption{Factorization times for the Poisson equation \autoref{eq:poisson}, using \schemeIb{} and \schemeII{} schemes which had, respectively, highest, and lowest factorization timings among the tested schemes.}
    \label{fig:factorization_scaling_poisson}
\end{figure}
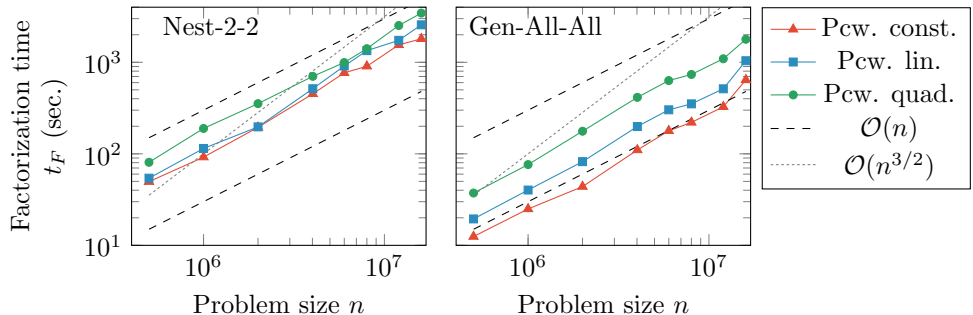

\begin{figure}[htbp]
    \centering
    \begin{tikzpicture}
        \begin{groupplot}[
            group style={
                group name=lin_elast_1,
                group size=2 by 1,
                ylabels at=edge left,
                yticklabels at=edge left,
                xlabels at=edge bottom,
                xticklabels at=edge bottom,
                vertical sep=0.2cm,
                horizontal sep=0.4cm,
            },
            xmode=log,
            ymode=log,
            xtick={10000,100000,1000000,10000000},
            xticklabels={$10^4$,$10^5$,$10^6$,$10^7$},
            ytick={1,10,100,1000,10000},
            yticklabels={$10^0$,$10^1$,$10^2$,$10^3$,$10^4$},
            ymin=0.1,ymax=950,
        ]
        \nextgroupplot[width=5.5cm,height=4.75cm,ylabel style={align=center},ylabel={Memory, $m_R$ (GB)},
            xlabel={Problem size $n$},
        	legend entries=
        	{Pcw. const., Pcw. lin., Pcw. quad., $\OO{n}$,
        	$\OO{n^{3/2}}$},
        	legend columns = 1,
        	legend style={cells={align=left},at={(2.85,1.0)}}]
            \addplot[PcwConst]
            table[x=n, y=mem]{fig/1a/lin_el_pcwconst_1a.dat};
            \addplot[PcwLin]
            table[x=n, y=mem]{fig/1a/lin_el_pcwlin_1a.dat};
            \addplot[PcwQuad]
            table[x=n, y=mem]{fig/1a/lin_el_pcwquad_1a.dat};
            
            \addplot [domain=15795:6502275,dashed]{2e-4 * x};
            \addplot [domain=15795:6502275,Dashed2]{6e-7 * x^(1.5)};
            \addplot [domain=15795:6502275,dashed]{5e-6 * x};
            \node at (axis cs:60000,500) {\schemeIa};
            
        \nextgroupplot[width=5.5cm,height=4.75cm,xlabel={Problem size $n$},]
            \addplot[PcwConst] 
            table[x=n, y=mem]{fig/1/lin_el_pcwconst_1.dat};
            \addplot[PcwLin]
            table[x=n, y=mem]{fig/1/lin_el_pcwlin_1.dat};
            \addplot[PcwQuad]
            table[x=n, y=mem]{fig/1/lin_el_pcwquad_1.dat};           
            
            \addplot [domain=15795:6502275,dashed]{2e-4 * x};
            \addplot [domain=15795:6502275,Dashed2]{6e-7 * x^(1.5)};
            \addplot [domain=15795:6502275,dashed]{5e-6 * x};
            \node at (axis cs:60000,500) {\schemeI};

%
        \end{groupplot}
    \end{tikzpicture}
    \caption{Example maximum memory usage for the linear elasticity equation \autoref{eq:lin_el}. 
    }
    \label{fig:memory_lin_elast}
\end{figure}
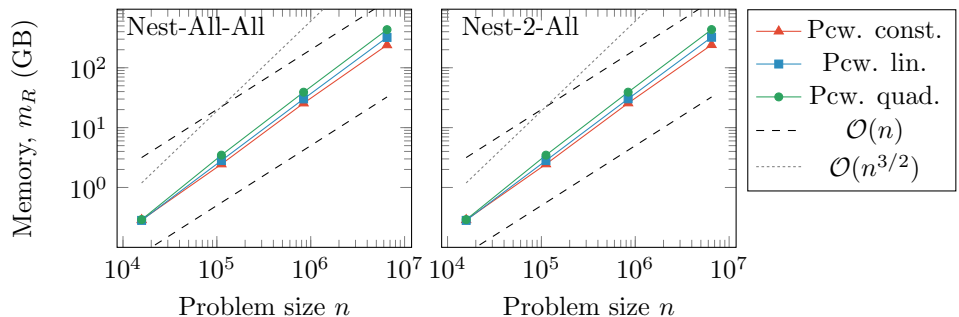

\subsection{Comparison to standard low-rank approximation}
For each tested scheme (\schemeIa, \schemeI, \schemeIb, \schemeII), we compare our results to the ones obtained when polynomial compression is replaced by a standard low-rank approximation (used by other hierarchical solvers, e.g., \cite{pouransari2017fast,cambier2019spand,chen2019robust,sushnikova2018compress,feliu2018recursively,ho2016hierarchical}). Namely, when computing the factorization with the given elimination and compression strategy to obtain $\Amat_{\ell},$ we record the sizes of all $\Matrix{Q}_1$ matrices from \autoref{eq:QR} used in compression. We then compute the factorization again with exactly the same set-up, except this time, to obtain the matrix $\Matrix{Q}$ from \autoref{eq:QR}, we find a low-rank approximation to the off-diagonal blocks using the column-pivoted rank-revealing QR \cite{chan1987rank,golub2012matrix}:
\begin{equation}\label{def:rrqr2}
    \Matrix{Q}\Matrix{R} = 
		\Matrix{\widehat{A}}_{\Boxvar \Boxvars{N}}
    \Matrix{P}
\end{equation}
However, the split $\Matrix{Q} = \begin{pmatrix} \Matrix{Q}_1 & \Matrix{Q}_2 \end{pmatrix}$ is such that the number of columns of $\Matrix{Q}_1$ is the same as before. As a result, the sizes of all nonzero matrix blocks when computing the factorization are the same, except inside the compression step. Denoting the obtained operator by $\Matrix{\hat{A}}_{\ell}$ we have that applying $\Amat_{\ell}^{-1}$ or $\Matrix{\hat{A}}_{\ell}^{-1}$ involves exactly the same cost. We call the obtained preconditioner the \emph{standard low-rank} equivalent which is algebraically very close to the recently introduced methods \cite{chen2019robust,cambier2019spand}.
We note here that more accurate rank-revealing factorizations, such as SVD, could be used but are often impractical because of their computationally expensive iterative nature. 

\subsubsection{Improved accuracy on smooth near-kernel eigenvectors}
\bgroup
\setlength{\tabcolsep}{0.35em}
\begin{table}[htbp]
\renewcommand{\arraystretch}{1.1}
\small
\centering {
	\begin{tabular}{@{} c l l l l @{} }
		 	 & \multicolumn{2}{l}{$E_1 = \frac{1}{\lambda_1}\|(\Amat - \Amat_{\ell}) \Matrix{v}_1 \|_2$}
		 	 & \multicolumn{2}{l}{$E_n = \frac{1}{\lambda_n}\|(\Amat - \Amat_{\ell}) \Matrix{v}_n \|_2$}  \\ 
	 $n$ & Pcw. quad. &  Low-rank & Pcw. quad.  & Low-rank \\    
	\hline					
	4.0M & $4.5 \cdot 10^{-2}$ & $4.4 \cdot 10^{-2}$ & $1.2 \cdot 10^{0}$ & $4.3 \cdot 10^2$ \\
	6.0M & $4.5 \cdot 10^{-2}$ & $4.4 \cdot 10^{-2}$ & $1.3 \cdot 10^{0}$ & $5.6 \cdot 10^2$ \\
	8.0M & $4.5 \cdot 10^{-2}$ & $4.4 \cdot 10^{-2}$ & $1.2 \cdot 10^{0}$ & $6.8 \cdot 10^2$ \\
	12.0M & $4.5 \cdot 10^{-2}$ & $4.4 \cdot 10^{-2}$ & $1.1 \cdot 10^{0}$ & $8.9 \cdot 10^2$ \\
	16.0M & $4.7 \cdot 10^{-2}$ & $4.4 \cdot 10^{-2}$ & $1.4 \cdot 10^{0}$ & $1.1 \cdot 10^3$ \\
	\hline
	\end{tabular}
	}
\caption{Accuracy on the unit-length eigenvectors corresponding to the largest eigenvalue ($E_1$), and the smallest eigenvalue ($E_n$), for \schemeIa{} with polynomial compression, and its standard low-rank equivalent, applied to Poisson equation \autoref{eq:poisson}.}
\label{table:poisson-eig-accuracy}
\end{table}
\egroup

Since in the case of the constant-coefficient Poisson equation eigenvectors and eigenvalues are known exactly, we can compute the relative backward error $\frac{1}{\lambda_i}\|(\Amat-\Amat_{\ell})\Matrix{v}_i\|_2$ for the unit-length eigenvectors $\Matrix{v}_1$ and $\Matrix{v}_n$ corresponding to, respectively, the largest, and the smallest, eigenvalues $\lambda_1$ and $\lambda_n$, to observe the effect of \cref{col:approx}. Example results for \schemeIa{} using polynomial compression, and its standard low-rank equivalent, are shown in \autoref{table:poisson-eig-accuracy}.
The error on $\Matrix{v}_1$ is small and similar in both cases; on the other hand, the error on $\Matrix{v}_n$ is nearly constant when using polynomial compression but for its standard low-rank equivalent, the error grows with increasing system sizes and becomes three orders of magnitude larger than when using polynomial compression.

\subsubsection{Improved iteration counts}
By design, applying our standard low-rank equivalent of the given preconditioner involves identical computational cost as applying the original preconditioner. Therefore, looking at the CG iteration counts is an exact way to compare the two approaches in terms of the quality of the resulting preconditioning operators. On all test cases, we record the CG iteration counts needed for convergence. 
The iteration counts are nearly constant or grow very slowly with increasing problem sizes when using our approaches, which is often not true for their standard low-rank equivalents. The plots of CG iteration counts are shown in \autoref{fig:cg_poisson}, \autoref{fig:cg_spe}, and \autoref{fig:cg_lin_elast}. Combined with the theoretical $\OO{n}$ complexity of applying $\Amat_{\ell}^{-1},$ the results mean that we can expect our algorithms to achieve scalings close to the optimal linear scaling of the total solution time.

\begin{figure}[htbp]
    \centering
    \begin{tikzpicture}
        \begin{groupplot}[
            group style={
                group name=poisson_2,
                group size=2 by 2,
                xlabels at=edge bottom,
                xticklabels at=edge bottom,
                vertical sep=0.1cm,
                horizontal sep=0.2cm,
                ylabels at=edge left,
                yticklabels at=edge left,
            },
            xmode=log,ymode=log,ymin=5,ymax=750,
            ytick={10,30,100,300,1000},
            yticklabels={$1 \cdot 10^1$,$3 \cdot 10^1$,$1 \cdot 10^2$,$3 \cdot 10^2$,$1 \cdot 10^3$},
            xtick={10000,100000,1000000,10000000},
            xticklabels={$10^4$,$10^5$,$10^6$,$10^7$},
        ]
        \nextgroupplot[width=5.5cm,height=5cm,ylabel style={align=center},
        	ylabel={CG iterations, $it_C$ \\ {}},
        	legend entries={Pcw. const.,  Low-rank \\ equiv., Pcw. lin., Low-rank \\ equiv., Pcw. quad., Low-rank \\ equiv.},
        	legend columns = 1,
        	legend style={cells={align=left},at={(2.8,1.0)}}]
        	\addplot[PcwConst] table[x=n, y=cg]{fig/1a/poisson_pcwconst_1a.dat};
        	\addplot[PcwConstComp] table[x=n, y=cg]{fig/1a/poisson_pcwconst_1a_compare.dat};
            \addplot[PcwLin] table[x=n, y=cg]{fig/1a/poisson_pcwlin_1a.dat};
            \addplot[PcwLinComp] table[x=n, y=cg]
            {fig/1a/poisson_pcwlin_1a_compare.dat};
            \addplot[PcwQuad] table[x=n, y=cg]{fig/1a/poisson_pcwquad_1a.dat};
            \addplot[PcwQuadComp] table[x=n, y=cg]
            {fig/1a/poisson_pcwquad_1a_compare.dat};
            \node at (axis cs:1000000,500) {\schemeIa};
 
        \nextgroupplot[width=5.5cm,height=5cm]
            \addplot[PcwConst] table[x=n, y=cg]{fig/1/poisson_pcwconst_1.dat};
            \addplot[PcwConstComp] table[x=n, y=cg]{fig/1/poisson_pcwconst_1_compare.dat};            
            \addplot[PcwLin] table[x=n, y=cg]{fig/1/poisson_pcwlin_1.dat};
            \addplot[PcwLinComp] table[x=n, y=cg]
            {fig/1/poisson_pcwlin_1_compare.dat};            
            \addplot[PcwQuad] table[x=n, y=cg]{fig/1/poisson_pcwquad_1.dat};
            \addplot[PcwQuadComp] table[x=n, y=cg]
            {fig/1/poisson_pcwquad_1_compare.dat};
            \node at (axis cs:1000000,500) {\schemeI};
            
        \nextgroupplot[width=5.5cm,height=5cm,
        	ylabel style = {align=center}, 
        	ylabel={CG iterations, $it_C$ \\ {}},
        	xlabel={Problem size $n$},]
            \addplot[PcwConst] table[x=n, y=cg]{fig/1b/poisson_pcwconst_1b.dat};
            \addplot[PcwLin] table[x=n, y=cg]{fig/1b/poisson_pcwlin_1b.dat};
            \addplot[PcwQuad] table[x=n, y=cg]{fig/1b/poisson_pcwquad_1b.dat};
            \addplot[PcwConstComp] table[x=n, y=cg]{fig/1b/poisson_pcwconst_1b_compare.dat};
            \addplot[PcwLinComp] table[x=n, y=cg]
            {fig/1b/poisson_pcwlin_1b_compare.dat};
            \addplot[PcwQuadComp] table[x=n, y=cg]
            {fig/1b/poisson_pcwquad_1b_compare.dat};
            \node at (axis cs:1000000,500) {\schemeIb};
            
       \nextgroupplot[width=5.5cm,height=5cm,
       		xlabel={Problem size $n$}]
            \addplot[PcwConst] table[x=n, y=cg]{fig/2/poisson_pcwconst_2.dat};
            \addplot[PcwLin] table[x=n, y=cg]{fig/2/poisson_pcwlin_2.dat};
            \addplot[PcwQuad] table[x=n, y=cg]{fig/2/poisson_pcwquad_2.dat};
            \addplot[PcwConstComp] table[x=n, y=cg]{fig/2/poisson_pcwconst_2_compare.dat};
            \addplot[PcwLinComp] table[x=n, y=cg]
            {fig/2/poisson_pcwlin_2_compare.dat};
            \addplot[PcwQuadComp] table[x=n, y=cg]
            {fig/2/poisson_pcwquad_2_compare.dat};
            \node at (axis cs:1000000,500) {\schemeII};
        \end{groupplot}
    \end{tikzpicture}
    \caption{CG iteration counts for the 3D Poisson equation \autoref{eq:poisson}. Our methods exhibit iteration counts almost independent of the grid size which is in general not true about their low-rank equivalents.}
    \label{fig:cg_poisson}
\end{figure}
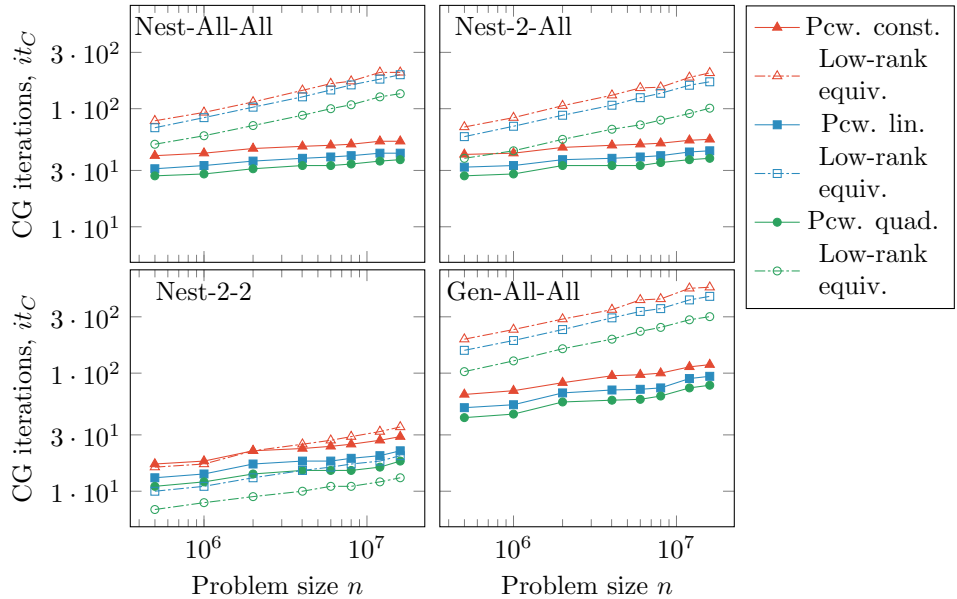

\begin{figure}[htbp]
    \centering
    \begin{tikzpicture}
        \begin{groupplot}[
            group style={
                group name=spe_1,
                group size=2 by 2,
                xlabels at=edge bottom,
                xticklabels at=edge bottom,
                vertical sep=0.2cm,
                horizontal sep=0.2cm,
                ylabels at=edge left,
                yticklabels at=edge left,
            },
            xmode=log,ymode=log,ymin=5,ymax=3000,
            ytick={10,30,100,300,1000,3000},
            yticklabels={$1 \cdot 10^1$,$3 \cdot 10^1$,$1 \cdot 10^2$,$3 \cdot 10^2$,$1 \cdot 10^3$,$3 \cdot 10^3$},
            xmin=350000,xmax=5000000,
            xtick={500000,1500000,4000000},
            xticklabels={$5 \cdot 10^5$,$1.5 \cdot 10^6$,$4 \cdot 10^6$},
        ]
        \nextgroupplot[width=5.5cm,height=4.75cm,ylabel style={align=center},
        	ylabel={CG iterations, $it_C$ \\ {}},
        	legend entries=
        	{Pcw. const., Low-rank\\equiv.,Pcw. lin., Low-rank\\equiv., Pcw. quad., Low-rank\\equiv.},
        	legend columns = 1,
        	legend style={cells={align=left},at={(2.8,1.0)}}]
            \addplot[PcwConst] table[x=n, y=cg]{fig/1a/spe_pcwconst_1a.dat};
            \addplot[PcwConstComp] table[x=n, y=cg]{fig/1a/spe_pcwconst_1a_compare.dat};
            \addplot[PcwLin] table[x=n, y=cg]{fig/1a/spe_pcwlin_1a.dat};
            \addplot[PcwLinComp] table[x=n, y=cg]
            {fig/1a/spe_pcwlin_1a_compare.dat};
            \addplot[PcwQuad] table[x=n, y=cg]{fig/1a/spe_pcwquad_1a.dat};
            \addplot[PcwQuadComp] table[x=n, y=cg]
            {fig/1a/spe_pcwquad_1a_compare.dat};
            \node at (axis cs:2500000,10) {\schemeIa};
            
        \nextgroupplot[width=5.5cm,height=4.75cm]
            \addplot[PcwConst] table[x=n, y=cg]{fig/1a/spe_pcwconst_1a.dat};
            \addplot[PcwConstComp] table[x=n, y=cg]{fig/1/spe_pcwconst_1_compare.dat};            
            \addplot[PcwLin] table[x=n, y=cg]{fig/1/spe_pcwlin_1.dat};
            \addplot[PcwLinComp] table[x=n, y=cg]
            {fig/1/spe_pcwlin_1_compare.dat};            
            \addplot[PcwQuad] table[x=n, y=cg]{fig/1/spe_pcwquad_1.dat};
            \addplot[PcwQuadComp] table[x=n, y=cg]
            {fig/1/spe_pcwquad_1_compare.dat};
            \node at (axis cs:2500000,10) {\schemeI};
            
        \nextgroupplot[width=5.5cm,height=4.75cm,
        	ylabel style = {align=center}, 
        	ylabel={CG iterations, $it_C$ \\ {}},
        	xlabel={Problem size $n$},]
            \addplot[PcwConst] table[x=n, y=cg]{fig/1b/spe_pcwconst_1b.dat};
            \addplot[PcwLin] table[x=n, y=cg]{fig/1b/spe_pcwlin_1b.dat};
            \addplot[PcwQuad] table[x=n, y=cg]{fig/1b/spe_pcwquad_1b.dat};
            \addplot[PcwConstComp] table[x=n, y=cg]{fig/1b/spe_pcwconst_1b_compare.dat};
            \addplot[PcwLinComp] table[x=n, y=cg]
            {fig/1b/spe_pcwlin_1b_compare.dat};
            \addplot[PcwQuadComp] table[x=n, y=cg]
            {fig/1b/spe_pcwquad_1b_compare.dat};
            \node at (axis cs:2500000,10) {\schemeIb};
            
       \nextgroupplot[width=5.5cm,height=4.75cm,
       		xlabel={Problem size $n$}]
            \addplot[PcwConst] table[x=n, y=cg]{fig/2/spe_pcwconst_2.dat};
            \addplot[PcwLin] table[x=n, y=cg]{fig/2/spe_pcwlin_2.dat};
            \addplot[PcwQuad] table[x=n, y=cg]{fig/2/spe_pcwquad_2.dat};
            \addplot[PcwConstComp] table[x=n, y=cg]{fig/2/spe_pcwconst_2_compare.dat};
            \addplot[PcwLinComp] table[x=n, y=cg]
            {fig/2/spe_pcwlin_2_compare.dat};
            \addplot[PcwQuadComp] table[x=n, y=cg]
            {fig/2/spe_pcwquad_2_compare.dat};
            \node at (axis cs:2500000,10) {\schemeII};
        \end{groupplot}
    \end{tikzpicture}
    \caption{CG iteration counts for the incompressible flow equation \autoref{eq:darcy}.} 
    \label{fig:cg_spe}
\end{figure}
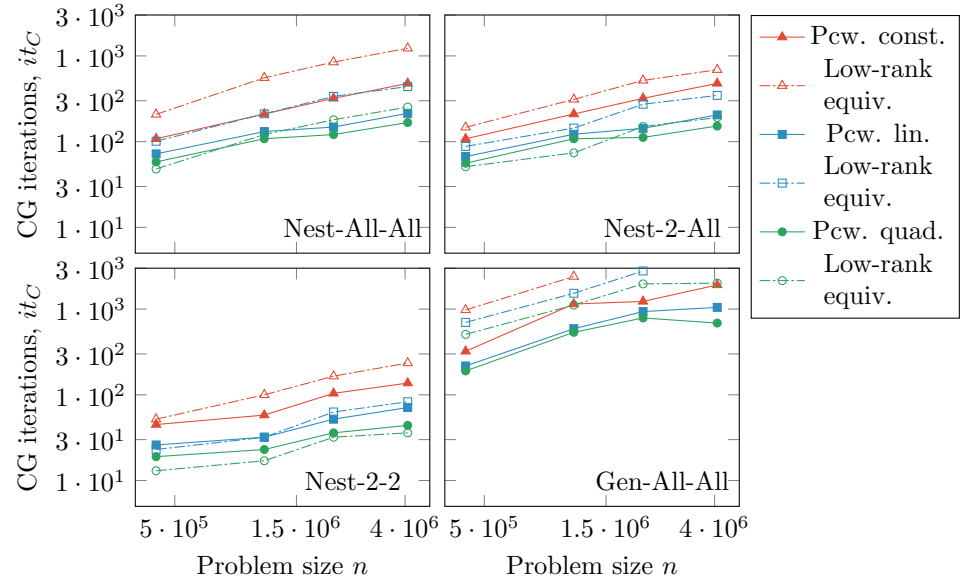

\begin{figure}[htb]
    \centering
    \begin{tikzpicture}
        \begin{groupplot}[
            group style={
                group name=lin_elast_2,
                group size=2 by 2,
                xlabels at=edge bottom,
                xticklabels at=edge bottom,
                vertical sep=0.2cm,
                horizontal sep=0.3cm,
                ylabels at=edge left,
                yticklabels at=edge left,
            },
            xmode=log,ymode=log,ymin=1,ymax=3000,
            ytick={3,10,30,100,300,1000,3000},
            yticklabels={$3 \cdot 10^0$,$1 \cdot 10^1$,$3 \cdot 10^1$,$1 \cdot 10^2$,$3 \cdot 10^2$,$1 \cdot 10^3$,$3 \cdot 10^3$},
            xtick={10000,100000,1000000,10000000},
            xticklabels={$10^4$,$10^5$,$10^6$,$10^7$},
        ]
        \nextgroupplot[width=5.5cm,height=4.75cm,ylabel style={align=center},
        	ylabel={CG iterations, $it_C$ \\ {}},   	
        	legend entries=
        	{Pcw. const.,Low-rank\\equiv.,Pcw.lin.,Low-rank\\equiv., Pcw. quad.,Low-rank\\equiv.},
        	legend columns = 1,
        	legend style={cells={align=left},at={(2.8,1.0)}}]
            \addplot[PcwConst] table[x=n, y=cg]{fig/1a/lin_el_pcwconst_1a.dat};
            \addplot[PcwConstComp] table[x=n, y=cg]{fig/1a/lin_el_pcwconst_1a_compare.dat};
            \addplot[PcwLin] table[x=n, y=cg]{fig/1a/lin_el_pcwlin_1a.dat};
            \addplot[PcwLinComp] table[x=n, y=cg]
            {fig/1a/lin_el_pcwlin_1a_compare.dat};
            \addplot[PcwQuad] table[x=n, y=cg]{fig/1a/lin_el_pcwquad_1a.dat};  
            \addplot[PcwQuadComp] table[x=n, y=cg]
            {fig/1a/lin_el_pcwquad_1a_compare.dat};
            \node at (axis cs:60000,1500) {\schemeIa};
            
        \nextgroupplot[width=5.5cm,height=4.75cm]
            \addplot[PcwConst] table[x=n, y=cg]{fig/1/lin_el_pcwconst_1.dat};
            \addplot[PcwConstComp] table[x=n, y=cg]{fig/1/lin_el_pcwconst_1_compare.dat};            
            \addplot[PcwLin] table[x=n, y=cg]{fig/1/lin_el_pcwlin_1.dat};
            \addplot[PcwLinComp] table[x=n, y=cg]{fig/1/lin_el_pcwlin_1_compare.dat};            
            \addplot[PcwQuad] table[x=n, y=cg]{fig/1/lin_el_pcwquad_1.dat};
            \addplot[PcwQuadComp] table[x=n, y=cg]
            {fig/1/lin_el_pcwquad_1_compare.dat};
            \node at (axis cs:60000,1500) {\schemeI};
            
        \nextgroupplot[width=5.5cm,height=4.75cm,
        	ylabel style = {align=center}, 
        	ylabel={CG iterations, $it_C$ \\ {}},
        	xlabel={Problem size $n$},]
            \addplot[PcwConst] table[x=n, y=cg]{fig/1b/lin_el_pcwconst_1b.dat};
            \addplot[PcwLin] table[x=n, y=cg]{fig/1b/lin_el_pcwlin_1b.dat};
            \addplot[PcwQuad] table[x=n, y=cg]{fig/1b/lin_el_pcwquad_1b.dat};
            \addplot[PcwConstComp] table[x=n, y=cg]{fig/1b/lin_el_pcwconst_1b_compare.dat};
            \addplot[PcwLinComp] table[x=n, y=cg]
            {fig/1b/lin_el_pcwlin_1b_compare.dat};
            \addplot[PcwQuadComp] table[x=n, y=cg]
            {fig/1b/lin_el_pcwquad_1b_compare.dat};
            \node at (axis cs:60000,1500) {\schemeIb};
            
       \nextgroupplot[width=5.5cm,height=4.75cm,
       		xlabel={Problem size $n$}]
            \addplot[PcwConst] table[x=n, y=cg]{fig/2/lin_el_pcwconst_2.dat};
            \addplot[PcwLin] table[x=n, y=cg]{fig/2/lin_el_pcwlin_2.dat};
            \addplot[PcwQuad] table[x=n, y=cg]{fig/2/lin_el_pcwquad_2.dat};
            \addplot[PcwConstComp] table[x=n, y=cg]{fig/2/lin_el_pcwconst_2_compare.dat};
            \addplot[PcwLinComp] table[x=n, y=cg]
            {fig/2/lin_el_pcwlin_2_compare.dat};
            \addplot[PcwQuadComp] table[x=n, y=cg]
            {fig/2/lin_el_pcwquad_2_compare.dat};
            \node at (axis cs:50000,1500) {\schemeII};
        \end{groupplot}
    \end{tikzpicture}
    \caption{CG iterations counts for the linear elasticity \autoref{eq:lin_el}.
    Computations using low-rank equivalent that did not converge within the time limit, were not recorded.}
    \label{fig:cg_lin_elast}
\end{figure}
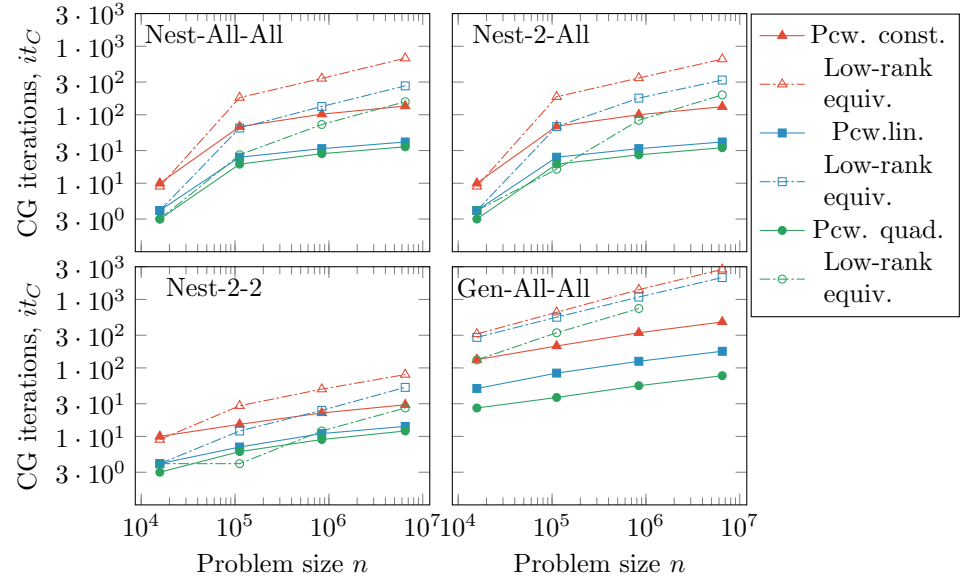

\subsubsection{Nearly optimal scalings of solution times}
In \autoref{fig:solution_time_lin_elast} and \autoref{fig:total_time_scaling_spe}, we plot the CPU times needed to solve the two more challenging problems (\autoref{eq:darcy}, \autoref{eq:lin_el}).
The graphs corresponding to our methods appear parallel to the $\OO{n}$ bounding lines which is not true for many of the standard low-rank equivalents.

\begin{figure}[htb]
    \centering
    \begin{tikzpicture}
        \begin{groupplot}[
            group style={
                group name=lin_elast_3,
                group size=2 by 2,
                xlabels at=edge bottom,
                xticklabels at=edge bottom,
                vertical sep=0.2cm,
                horizontal sep=0.3cm,
                ylabels at=edge left,
                yticklabels at=edge left,
            },
            xmode=log,ymode=log,ymin=0,ymax=30000,
            xtick={10000,100000,1000000,10000000},
            xticklabels={$10^4$,$10^5$,$10^6$,$10^7$},
        ]
        \nextgroupplot[width=5.5cm,height=4.75cm,ylabel style={align=center},
        	ylabel={Solution time \\ $t_S$ (sec.)},
        	legend entries=
        	{Pcw. const., Low-rank\\equiv., Pcw. lin.,  Low-rank\\equiv., Pcw. quad., Low-rank\\equiv., $\OO{n}$,$\OO{n^{3/2}}$},
        	legend columns = 1,
        	legend style={cells={align=left},at={(2.8,1.0)}}]
            \addplot[PcwConst] table[x=n, y=ts]{fig/1a/lin_el_pcwconst_1a.dat};
            \addplot[PcwConstComp] table[x=n, y=ts]{fig/1a/lin_el_pcwconst_1a_compare.dat};
            \addplot[PcwLin] table[x=n, y=ts]{fig/1a/lin_el_pcwlin_1a.dat};
            \addplot[PcwLinComp] table[x=n, y=ts]
            {fig/1a/lin_el_pcwlin_1a_compare.dat};
            \addplot[PcwQuad] table[x=n, y=ts]{fig/1a/lin_el_pcwquad_1a.dat};
            \addplot[PcwQuadComp] table[x=n, y=ts]
            {fig/1a/lin_el_pcwquad_1a_compare.dat};
            \addplot [domain=15795:6502275,dashed]{2e-3 * x};
            \addplot [domain=15795:6502275,Dashed2]{6e-6 * x^(1.5)};
            \addplot [domain=15795:6502275,dashed]{1e-5 * x}; 
            \node at (axis cs:60000,10000) {\schemeIa};
            
        \nextgroupplot[width=5.5cm,height=4.75cm]
            \addplot [domain=15795:6502275,dashed]{2e-3 * x};
            \addplot [domain=15795:6502275,Dashed2]{6e-6 * x^(1.5)};
            \addplot [domain=15795:6502275,dashed]{1e-5 * x};        	
            \addplot[PcwConst] table[x=n, y=ts]{fig/1/lin_el_pcwconst_1.dat};
            \addplot[PcwConstComp] table[x=n, y=ts]
            {fig/1/lin_el_pcwconst_1_compare.dat};            
            \addplot[PcwLin] table[x=n, y=ts]{fig/1/lin_el_pcwlin_1.dat};
            \addplot[PcwLinComp] table[x=n, y=ts]
            {fig/1/lin_el_pcwlin_1_compare.dat};            
            \addplot[PcwQuad] table[x=n, y=ts]{fig/1/lin_el_pcwquad_1.dat};
            \addplot[PcwQuadComp] table[x=n, y=ts]
            {fig/1/lin_el_pcwquad_1_compare.dat};
            
            \addplot [domain=15795:6502275,dashed]{2e-3 * x};
            \addplot [domain=15795:6502275,Dashed2]{6e-6 * x^(1.5)};
            \addplot [domain=15795:6502275,dashed]{1e-5 * x};
            
            \node at (axis cs:60000,10000) {\schemeI};

        \nextgroupplot[width=5.5cm,height=4.75cm,
        	ylabel style = {align=center}, 
        	ylabel={Solution time \\ $t_S$ (sec.)},
        	xlabel={Problem size $n$},]
            \addplot[PcwConst] table[x=n, y=ts]{fig/1b/lin_el_pcwconst_1b.dat};
            \addplot[PcwLin] table[x=n, y=ts]{fig/1b/lin_el_pcwlin_1b.dat};
            \addplot[PcwQuad] table[x=n, y=ts]{fig/1b/lin_el_pcwquad_1b.dat};
            gnhtnym\addplot[PcwConstComp] table[x=n, y=ts]{fig/1b/lin_el_pcwconst_1b_compare.dat};
            \addplot[PcwLinComp] table[x=n, y=ts]
            {fig/1b/lin_el_pcwlin_1b_compare.dat};
            \addplot[PcwQuadComp] table[x=n, y=ts]
            {fig/1b/lin_el_pcwquad_1b_compare.dat};
            \node at (axis cs:60000,10000) {\schemeIb};

            \addplot [domain=15795:6502275,dashed]{2e-3 * x};
            \addplot [domain=15795:6502275,Dashed2]{6e-6 * x^(1.5)};
            \addplot [domain=15795:6502275,dashed]{1e-5 * x};
       
       \nextgroupplot[width=5.5cm,height=4.75cm,
       		xlabel={Problem size $n$}]
           \addplot[PcwConst] table[x=n, y=ts]{fig/2/lin_el_pcwconst_2.dat};
            \addplot[PcwLin] table[x=n, y=ts]{fig/2/lin_el_pcwlin_2.dat};
            \addplot[PcwQuad] table[x=n, y=ts]{fig/2/lin_el_pcwquad_2.dat};
            \addplot[PcwConstComp] table[x=n, y=ts]{fig/2/lin_el_pcwconst_2_compare.dat};
            \addplot[PcwLinComp] table[x=n, y=ts]
            {fig/2/lin_el_pcwlin_2_compare.dat};
            \addplot[PcwQuadComp] table[x=n, y=ts]
            {fig/2/lin_el_pcwquad_2_compare.dat};
            \node at (axis cs:50000,10000) {\schemeII};

            \addplot [domain=15795:6502275,dashed]{2e-3 * x};
            \addplot [domain=15795:6502275,Dashed2]{6e-6 * x^(1.5)};
            \addplot [domain=15795:6502275,dashed]{1e-5 * x};            
            
        \end{groupplot}
    \end{tikzpicture}
    \caption{Solution times for the linear elasticity equation \autoref{eq:lin_el}. 
    }
    \label{fig:solution_time_lin_elast}
\end{figure}
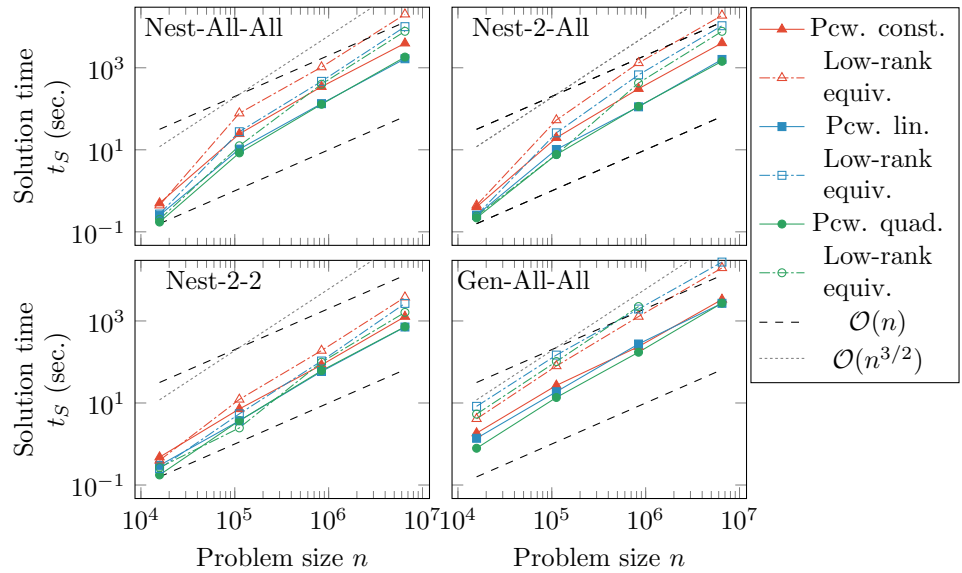

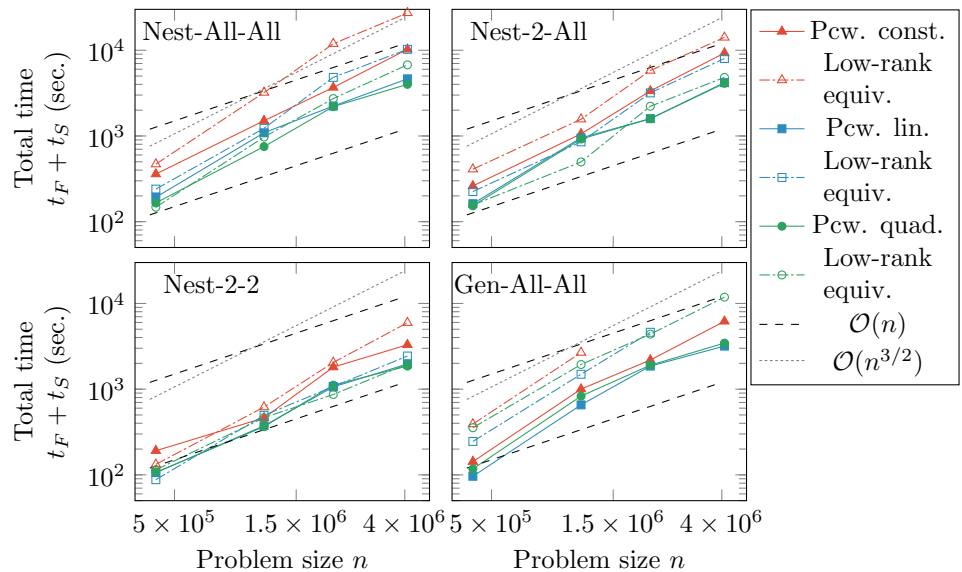
\begin{figure}[htb]
    \centering
    \begin{tikzpicture}
        \begin{groupplot}[
            group style={
                group name=spe_2,
                group size=2 by 2,
                xlabels at=edge bottom,
                xticklabels at=edge bottom,
                vertical sep=0.2cm,
                horizontal sep=0.3cm,
                ylabels at=edge left,
                yticklabels at=edge left,
            },
            xmode=log,ymode=log,ymin=50,ymax=30000,
            xmin=350000,xmax=5000000,
            xtick={500000,1500000,4000000},
            xticklabels={$5 \times 10^5$,$1.5 \times 10^6$,$4 \times 10^6$},
        ]
        \nextgroupplot[width=5.5cm,height=4.75cm,
        	ylabel style={align=center},ylabel={Total time \\$t_F + t_S$ (sec.)},
        	legend entries=
        	{Pcw. const., Low-rank\\equiv., Pcw. lin., Low-rank\\equiv., Pcw. quad., Low-rank\\equiv.,  $\OO{n}$, $\OO{n^{3/2}}$},
        	legend columns = 1,
        	legend style={cells={align=left},at={(2.8,1.0)}}]      
            \addplot[PcwConst]
            table[x=n, y=tf+ts]{fig/1a/spe_pcwconst_1a.dat};
            \addplot[PcwConstComp] 
            table[x=n, y=tf+ts]{fig/1a/spe_pcwconst_1a_compare.dat};            
            \addplot[PcwLin]
            table[x=n, y=tf+ts]{fig/1a/spe_pcwlin_1a.dat};
            \addplot[PcwLinComp]
            table[x=n, y=tf+ts]{fig/1a/spe_pcwlin_1a_compare.dat};            
            \addplot[PcwQuad]
            table[x=n, y=tf+ts]{fig/1a/spe_pcwquad_1a.dat};
            \addplot[PcwQuadComp]
            table[x=n, y=tf+ts]{fig/1a/spe_pcwquad_1a_compare.dat};
            
            \addplot [domain=400000:4000000,dashed]{3e-3 * x};
            \addplot [domain=400000:4000000,Dashed2]{3e-6 * x^(1.5)};
            \addplot [domain=400000:4000000,dashed]{3e-4 * x};
            \node at (axis cs:700000,17000) {\schemeIa};
            
        \nextgroupplot[width=5.5cm,height=4.75cm]
            \addplot[PcwConst] 
            table[x=n, y=tf+ts]{fig/1/spe_pcwconst_1.dat};
            \addplot[PcwConstComp] 
            table[x=n, y=tf+ts]{fig/1/spe_pcwconst_1_compare.dat};            
            \addplot[PcwLin]
            table[x=n, y=tf+ts]{fig/1/spe_pcwlin_1.dat};
            \addplot[PcwLinComp]
            table[x=n, y=tf+ts]{fig/1/spe_pcwlin_1_compare.dat};            
            \addplot[PcwQuad]
            table[x=n, y=tf+ts]{fig/1/spe_pcwquad_1.dat};
            \addplot[PcwQuadComp]
            table[x=n, y=tf+ts]{fig/1/spe_pcwquad_1_compare.dat};                      
            
            \addplot [domain=400000:4000000,dashed]{3e-3 * x};
            \addplot [domain=400000:4000000,Dashed2]{3e-6 * x^(1.5)};
            \addplot [domain=400000:4000000,dashed]{3e-4 * x};
            \node at (axis cs:700000,17000) {\schemeI};

        \nextgroupplot[width=5.5cm,height=4.75cm,
        	ylabel style = {align=center}, 
        	ylabel={Total time \\$t_F + t_S$ (sec.)},
        	xlabel={Problem size $n$},]
            \addplot[PcwConst]
            table[x=n, y=tf+ts]{fig/1b/spe_pcwconst_1b.dat};
            \addplot[PcwLin]
            table[x=n, y=tf+ts]{fig/1b/spe_pcwlin_1b.dat};
            \addplot[PcwQuad]
            table[x=n, y=tf+ts]{fig/1b/spe_pcwquad_1b.dat};
            
            \addplot[PcwConstComp] 
            table[x=n, y=tf+ts]{fig/1b/spe_pcwconst_1b_compare.dat};
            \addplot[PcwLinComp]
            table[x=n, y=tf+ts]{fig/1b/spe_pcwlin_1b_compare.dat};
            \addplot[PcwQuadComp]
            table[x=n, y=tf+ts]{fig/1b/spe_pcwquad_1b_compare.dat};            
            
            \addplot [domain=400000:4000000,dashed]{3e-3 * x};
            \addplot [domain=400000:4000000,Dashed2]{3e-6 * x^(1.5)};
            \addplot [domain=400000:4000000,dashed]{3e-4 * x};
            \node at (axis cs:700000,17000) {\schemeIb};
            
       \nextgroupplot[width=5.5cm,height=4.75cm,
       		xlabel={Problem size $n$}]
            \addplot[PcwConst]
            table[x=n, y=tf+ts]{fig/2/spe_pcwconst_2.dat};
            \addplot[PcwLin]
            table[x=n, y=tf+ts]{fig/2/spe_pcwlin_2.dat};
            \addplot[PcwQuad]
            table[x=n, y=tf+ts]{fig/2/spe_pcwquad_2.dat};

            \addplot[PcwConstComp] 
            table[x=n, y=tf+ts]{fig/2/spe_pcwconst_2_compare.dat};
            \addplot[PcwLinComp]
            table[x=n, y=tf+ts]{fig/2/spe_pcwlin_2_compare.dat};
            \addplot[PcwQuadComp]
            table[x=n, y=tf+ts]{fig/2/spe_pcwquad_2_compare.dat};                        
            
            \addplot [domain=400000:4000000,dashed]{3e-3 * x};
            \addplot [domain=400000:4000000,Dashed2]{3e-6 * x^(1.5)};
            \addplot [domain=400000:4000000,dashed]{3e-4 * x};        
            \node at (axis cs:650000,17000) {\schemeII};
        \end{groupplot}
    \end{tikzpicture}
    \caption{Total solution times (including factorization and the CG iteration) for the incompressible flow problem \autoref{eq:darcy}.}
    \label{fig:total_time_scaling_spe}
\end{figure}

\subsection{Choosing the appropriate preconditioner}
\schemeIb{} exhibits lowest and nearly constant CG iteration counts in all tested cases. Inevitably, it has largest memory requirements as well as the cost of applying $\Amat_{\ell}^{-1}$. \schemeI{} and \schemeIa{} performed almost identically. Therefore, \schemeIa{} is likely preferable of the two, with $\OO{n}$ factorization complexity guarantees (instead of $\OO{n \log{n}}$). Polynomial compression allows therefore for bounding the sizes of the nodes appearing in \schemeIa{}, without losing accuracy (see the proof of \cref{thm:proposition-complexity}). We expect this to be of importance for parallel scaling of the algorithm. 

In the tests, \schemeII{} performed competitively in terms of solution times despite higher iteration counts. Notice that $\Amat_{\ell}^{-1}$ is a product of block diagonal matrices with one block per each node. This suggests that \schemeII{} also has very promising properties for massive parallelization.
However, higher iteration counts may make \schemeII{} less practical on very ill-conditioned problems.


\section{Conclusions and future work}\label{sec:conclusions}
The numerical results show that \algfull{} (\algo) gives rise to robust preconditioners, exhibiting optimal or near optimal scalings of solution times. Our methods are based on a special treatment of the near-kernel smooth eigenvectors, directly targeting the fundamental limitations of hierarchical matrix approaches. Moreover, the factorization is guaranteed to succeed in exact arithmetic, and the preconditioners are SPD operators, so Conjugate Gradient can always be used. Also, polynomial compression can allow avoiding costly rank-revealing decompositions---which have been reported as the main computational bottleneck of hierarchical solvers \cite{cambier2019spand,pouransari2017fast,sushnikova2018compress}. We did not find polynomial compression to be a bottleneck in our implementation.

Other strategies for designing preconditioners based on \algo{} than described in this paper, can be considered. While our examples were discretized on regular cartesian grids, partitionings on general grids are applicable (see e.g. \cite{cambier2019spand} for details). Also, combining the polynomial compression with low-rank approximation may result in more efficient methods. In the absence of geometrical information (e.g., when only the system matrix is known), piecewise constant vectors would be then used. Moreover, polynomial compression can be applied almost verbatim in an approach that would only compress the well-separated interactions in a general domain partitioning (such approaches include \cite{pouransari2017fast,sushnikova2018compress}).



Similar to \cite{ho2016hierarchical,cambier2019spand} our algorithms have very promising parallel properties. Thanks to compression, most computations are performed in the initial levels of the algorithm, contrasting direct solvers such as nested dissection multifrontal, where the exact factorization of the largest block can dominate the computations. We believe that the bounded sizes of nodes ensured by polynomial compression may be of further importance for parallel scaling of the algorithm. Efficient parallel implementations of hierarchical approaches such as ours is a topic of active research \cite{chen2008algorithm,li2017distributed}.



\FloatBarrier
\section*{Acknowledgments}
This work was partly funded by the Stanford University Petroleum Research Institute's Reservoir Simulation Industrial Affiliates Program (SUPRI-B). We are very grateful to prof.\ Hamdi Tchelepi for interest in our methods and useful suggestions. We would also like to thank Pavel Tomin and Sergey Klevtsov for helping us prepare the SPE10 reservoir benchmark case, and L\'{e}opold Cambier for invaluable discussions about the hierarchical factorizations and test cases.

The computing for this project was performed on the Stanford University Sherlock cluster. We would like to thank the Stanford Research Computing Center for providing computational resources and support.

\bibliographystyle{siamplain}
\bibliography{references}
\end{document}